\documentclass[11 pt,reqno]{amsart}
\usepackage{amsmath,amsthm,amsfonts,amssymb,mathrsfs,bm,graphicx,stmaryrd,dsfont}

\usepackage[usenames]{color}
\usepackage[colorlinks=true,linkcolor=blue]{hyperref}
\usepackage[top=1in,bottom=1in,left=1in,right=1in]{geometry}
\parskip 	\smallskipamount

\newtheorem{theorem}{Theorem}[section]
\newtheorem{lemma}[theorem]{Lemma}

\newtheorem{proposition}[theorem]{Proposition}

\newtheorem{remark}[theorem]{Remark}

\newtheorem{definition}[theorem]{Definition}

\def\ind{{\mathbf 1}}
\def\N{\mathbb{N}}
\def\P{\mathbb{P}}
\def\Z{\mathbb{Z}}
\def\R{\mathbb{R}}

\def\E{\mathbb{E}}
\def\e{\varepsilon}

\newcommand{\ccW}{\mathcal{W}}
\newcommand{\sT}{\mathscr{T}}
\newcommand{\cF}{\mathcal{F}}
\newcommand{\scM}{\mathscr{M}}
\newcommand{\scA}{\mathscr{A}}
\newcommand{\scX}{\mathscr{X}}

\newcommand{\clB}{\mathcal{B}}
\newcommand{\scU}{\mathscr{U}}
\newcommand{\scY}{\mathscr{Y}}
\newcommand{\scZ}{\mathscr{Z}}
\newcommand{\ccR}{\mathcal{R}}
\newcommand{\ccQ}{\mathcal{Q}}
\newcommand{\sP}{\mathscr{P}}
\newcommand{\scE}{\mathscr{E}}
\newcommand{\scF}{\mathscr{F}}
\newcommand{\cR}{\mathrm{R}}
\newcommand{\cB}{\mathrm{B}}
\newcommand{\cE}{\mathrm{E}}
\newcommand{\cW}{\mathrm{W}}
\newcommand{\cZ}{\mathcal{Z}}

\newcommand{\cS}{\mathcal{S}}

\newcommand\floor[1]{\lfloor#1\rfloor}
\newcommand\ceil[1]{\lceil#1\rceil}
\newcommand\old[1]{}

\begin{document}
\title[Formation of structure by competitive erosion]{Formation of large-scale random structure by competitive erosion}

\author{
Shirshendu Ganguly}
\address{Shirshendu Ganguly, Department of Statistics, UC Berkeley, Berkeley, CA, USA}
\email{sganguly@berkeley.edu}
\author{
Lionel Levine}
\address{Lionel Levine, Department of Mathematics, Cornell University, Ithaca, NY, USA}
\email{levine@math.cornell.edu}
\author{
Sourav Sarkar}
\address{Sourav Sarkar, Department of Statistics, UC Berkeley, Berkeley, CA, USA}
\email{souravs@berkeley.edu}

\date{}
\maketitle

\begin{abstract} We study the following one-dimensional model of annihilating particles.
Beginning with all sites of $\Z$ uncolored, a blue particle performs simple random walk from $0$ until it reaches a nonzero red or uncolored site, and turns that site blue; then, a red particle performs simple random walk from $0$ until it reaches a nonzero blue or uncolored site, and turns that site red. We prove that after $n$ blue and $n$ red particles alternately perform such walks, the total number of colored sites is of order $n^{1/4}$. The resulting random color configuration, after rescaling by $n^{1/4}$ and taking $n\to \infty$, has an explicit description in terms of alternating extrema of Brownian motion (the global maximum on a certain interval, the global minimum attained after that maximum, etc.).
\end{abstract} 

\section{Introduction and main results}

\emph{Competitive erosion} models a random interface sustained in equilibrium by equal and opposite pressures on each side of the interface. When the sources of opposite pressure are far apart, the resulting interface remains in a predictable position with high probability \cite{GLPP}. When the sources are located at the \emph{same} point, a much more intricate behavior emerges, with a macroscopically random interface.  The aim of this paper is to characterize the limiting distribution of this interface in one dimension. We will find an exact description for the interface in terms of alternating maxima and minima of Brownian motion.

\begin{figure}[h]
    \centering
    \includegraphics[width=1.0\textwidth]{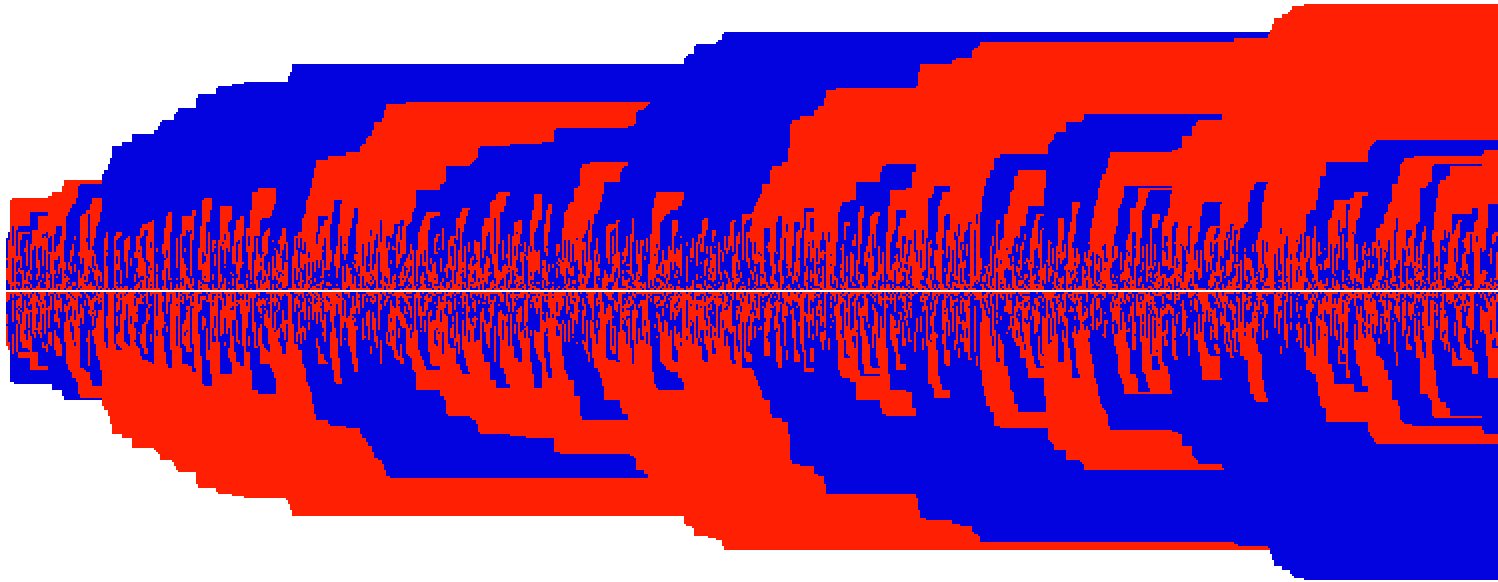}
    \caption{A sample of competitive erosion in $\Z$. 
    Time increases from left to right: each column depicts the random configuration of red and blue sites in $\Z$ after the addition of $10^5$ more particles of each color.   
    Sites near $0$ (the white horizontal bar) change color often, while those far from $0$ change only rarely.   According to Theorem~\ref{informal}, the total number of colored sites at time $n$ is of order $n^{1/4}$.
   }
    \label{f.sim}
\end{figure}

\subsection{Competitive erosion in one dimension}

We begin with an informal description of the model; formal definitions are in \textsection\ref{fdn}. 

All sites in $\Z$ begin uncolored. At every odd time step, a blue particle is emitted from $0$ and at every even time step a red particle is emitted from $0$. The most recently emitted particle performs a simple symmetric random walk on $\Z$ until it hits a site in $\Z\setminus \{0\}$ which is either uncolored or has a particle of the opposite color. In the former case it occupies the site and in the latter case it annihilates the other particle and occupies its place. 
These random walks happen sequentially: each walk finishes before the next particle is emitted. 

The main goals of this paper are to understand the growth rate of the number of sites explored (Theorem~\ref{informal}) and the scaling limit of the resulting red and blue patches (Theorem~\ref{mainruns}).

Let $S(n)$ be the total number of sites explored by competitive erosion on $\Z$ after $n$ particles in turn have performed their random walks (see Figure~\ref{f.sim}, where $n$ increases from left to right, and $S(n)$ is the length of the corresponding column of colored sites). 

\begin{theorem}\label{informal} There is a constant $C>0$ such that $S(n)/n^{1/4}$
converges in distribution to the random variable $C\sqrt{\scX_1}$, 
where 
\begin{equation}\label{e:defX1}
\scX_1:=\sup\{a:T_a+T'_a\leq 1\},
\end{equation}
where $T_a$ and $T'_a$ are hitting times of $a$ for the absolute value processes $|\clB(\cdot)|$ and $|\clB'(\cdot)|$ of  independent standard Brownian motions $\clB(\cdot)$ and $\clB'(\cdot)$.  
\end{theorem}

The exact value of the constant above turns out to be $C = 2\sqrt{2} (\frac{1}{2}-\sum_{j=1}^\infty \frac{1}{j(j+1)^2(j+2)})^{-1/4}$. This value arises from a comparison of two time scales, Theorem~\ref{t:mainparticle} below.  

\subsection{Related models}

To put our model in context, consider first competitive erosion without any red particles. Each blue particle released from the origin performs simple random walk until reaching an uncolored site, then turns that site blue. 
In $\Z^d$, the resulting cluster of 
blue sites grows asymptotically as Euclidean ball \cite{lbg}, with square root fluctuations in dimension $1$ and logarithmic fluctuations in higher dimensions \cite{ag,ag2,ag3, jls1,jls2,jls3}. 
This model (known as Internal Diffusion-Limited Aggregation: ``internal'' because the particles start inside the cluster, ``aggregation'' because the cluster grows, ``diffusion-limited'' because the mechanism of growth is for random walkers to reach the boundary) fits into a $2\times 2$ family:

	\begin{table}[h]
	\begin{tabular}{c|cc|}
	~ & Internal & External \\
	\hline
	Aggregation & \emph{smoothing} & \emph{roughening} \\
	Erosion & \emph{roughening} & \emph{smoothing} \\
	\hline
	\end{tabular}
\old{	
	\begin{tabular}{c|cc|}
	~ & Internal & External \\
	\hline
	Aggregation & \emph{smoothing} (fjords attenuate) & \emph{roughening} (peninsulas extend) \\
	Erosion & \emph{roughening} (fjords extend) & \emph{smoothing} (peninsulas attenuate)  \\
	\hline
	\end{tabular}
}
	\medskip
	\caption{\label{table:2x2} Four types of diffusion-limited boundary dynamics: Random walkers can start either inside or outside the cluster $A$; when they reach the boundary $\partial A$ they can either aggregate (enlarge $A$ by one point) or erode (delete one point from $A$). Internal Aggregation and External Erosion are ``smoothing'' in the sense that fjords and peninsulas become shorter. The other two combinations are ``roughening'' in the sense that fjords or peninsulas become longer.}
	\end{table}
	
To obtain a smoothing model of boundary dynamics that is symmetric the random cluster $A$ and its complement, Jim Propp proposed alternating steps of Internal Aggregation with External Erosion.
If we color each site blue or red according to whether it belongs to $A$ or $A^c$, then each blue walker erodes a red site and each red walker erodes a blue site, hence the name Competitive Erosion. The first study of competitive erosion was on the cylinder $\Z_n \times \{0,1,\ldots,n\}$ with each red walker started at a uniform point on top layer $\Z_n \times \{n\}$, and each blue walker started at a uniform point on the bottom layer $\Z_n \times \{0\}$. The main result of \cite{GLPP} is that 
the stationary distribution 
concentrates, with probability exponentially close to $1$, on configurations with $o(n)$ fluctuations around a flat interface. 

Competitive erosion on discretized plane domains is studied in \cite{gp15}, where the limiting shape of the interface is shown to be invariant under conformal maps.

The techniques of \cite{GLPP,gp15} rely on red and blue walkers starting far apart. The present paper is motivated by the variant mentioned at the end of \cite{GLPP}, in which red and blue walkers instead start at the same point.  The dynamics of competitive erosion (with strictly alternating red and blue walkers) ensure that this mutual starting point 
remains on the interface between red and blue.
So the model studied in this paper is intermediate between the ``Internal'' and ``External'' columns of Table~\ref{table:2x2} in the sense that all walkers start on the boundary. It marries the smoothness of internal DLA (which grows asymptotically as ball, with only logarithmic fluctuations) with the wildness of external DLA (which is believed to grow fractal arms \cite{hk1, benDLA}). 

We remark that a two-color growth process very different from competitive erosion is the ``oil and water'' model \cite{oilwater} in which each random walker is permitted to move only in the presence of an oppositely colored walker. In that model, $n$ red and $n$ blue walkers started at the origin spread somewhat further (to distance $n^{1/3}$ instead of $n^{1/4}$), and the colors display no macroscopic structure.
 
Competing particle systems, modeling co-existence of various species etc, have been the subject of intense study in physical sciences as well as mathematics: see, for example,  \cite{bra1,bra2} for annihilating random walks, and \cite{cr2,cr1} for a two-species Richardson model (first-passage percolation).
\begin{figure}[h] 
\centering
\includegraphics[scale=.75]{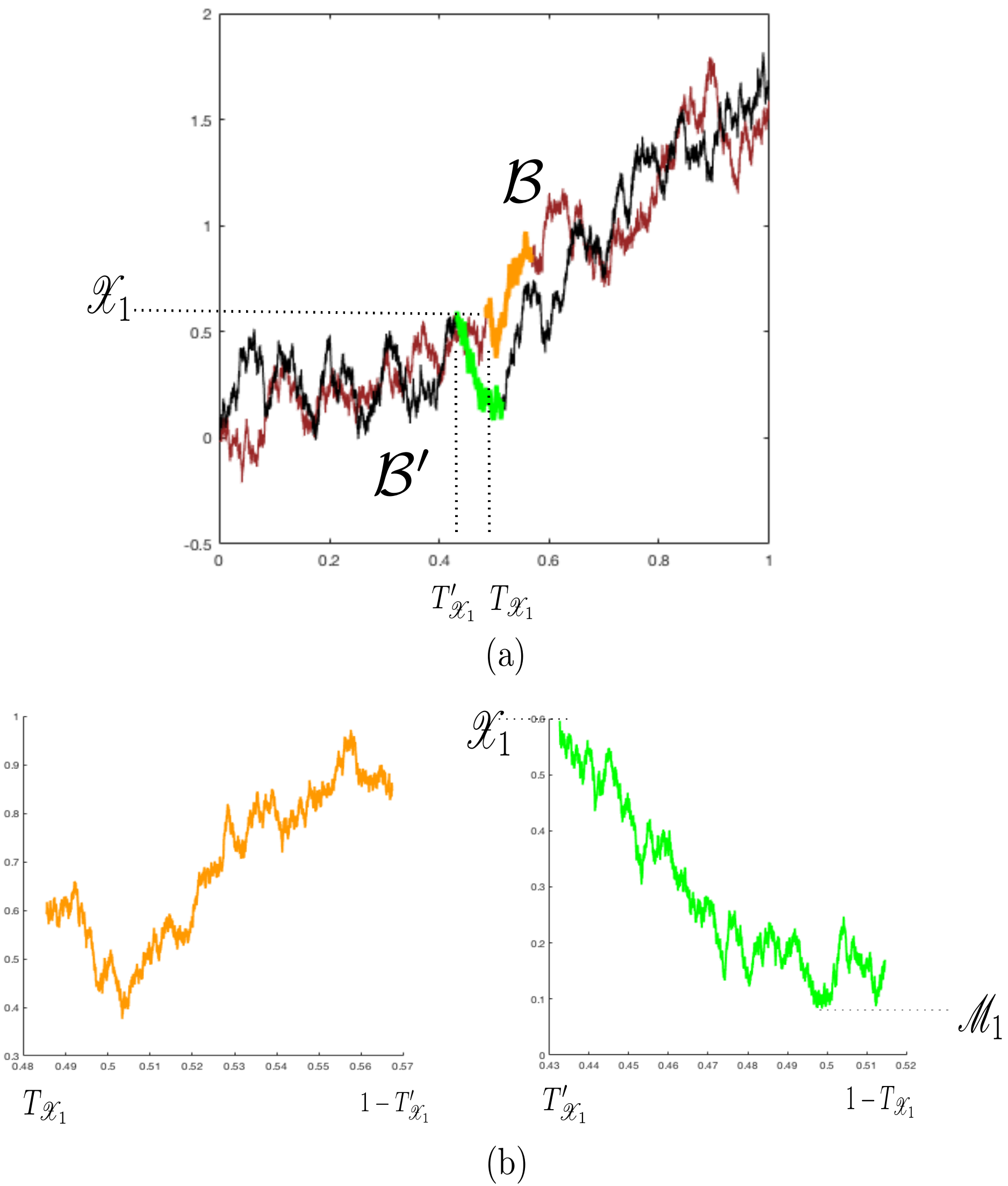}
\caption{(a) Simulations of two independent Brownian motions $\clB$ (in brown) and $\clB'$ (in black) appearing in the statement of Theorem \ref{informal}.  
(b) The yellow path shows the restriction of $\clB$ to the interval $[T_{\scX_1}, 1-T'_{\scX_1}]$, and the green path shows the restriction of $\clB'$ to $[T'_{\scX_1}, 1-T_{\scX_1}]$, where $T_{\scX_1}$ and $T'_{\scX_1}$ are the hitting times as defined in \eqref{e:defX1}. 
In this example: the first of the two events in \eqref{alm464} holds: the green path is maximized at $T'_{\scX_1}$, so the global maximum of $\clB'$ on the interval $[0,1-T_{\scX_1}]$ is attained at time $T'_{\scX_1}$.  Theorem \ref{alm464} relates the length $E(n,1)$ of the rightmost monochromatic interval in competitive erosion to the square root of the difference between the maximum $\scX_1$ and minimum $\scM_1$ of the green path.
}
\label{brow}
 \end{figure}

\begin{figure}[h]
    \centering
    \includegraphics[width=0.85\textwidth]{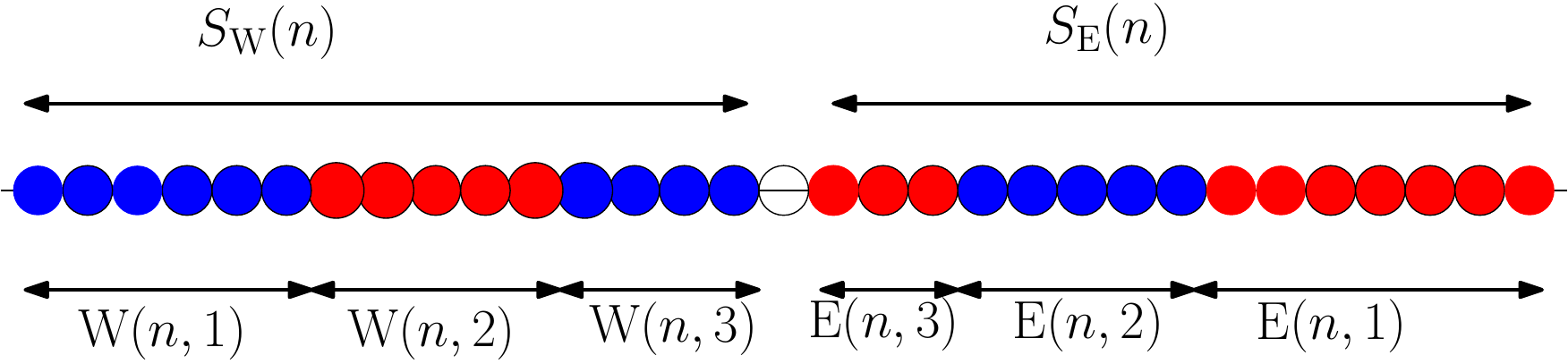}
    \caption{Schematic of the color configuration after $n$ particles have finished walking.   $S_\cE(n)$ and $S_\cW(n)$ denote the number of explored sites on the negative and positive integer line respectively. 
 $\cE(n,1), \cE(n,2), \ldots$ denote the length of the monochromatic runs starting furthest from the origin on the positive integer line.  $\cW(n,1), \cW(n,2), \ldots$ denote the analogous quantities on the negative line.     
The white circle in the middle denotes the origin.}
    \label{fig:Ball2def1}
\end{figure}

\subsection{Scaling limit of the color configuration}

Our next result describes how to read off the scaling limit of the final color configuration of competitive erosion on $\Z$, in terms of the random variable $\scX_1$ appearing in Theorem~\ref{informal} and certain extremal values of the Brownian paths $\clB,\clB'$. 
Later in the article (see Lemma \ref{l:localmax}), it is shown that, almost surely, exactly one of the following occurs:
\begin{equation}\label{alm464}
\scX_1=\max \{|\clB'(s)|:s\in [0,1-T_{\scX_1}]\} \quad \text{      or      } \quad \scX_1=\max \{|\clB(s)|:s\in [0,1-T'_{\scX_1}]\}.
\end{equation}
Assume, without loss of generality, that the former holds and moreover that $\clB'(T'_{\scX_1})>0,$ 
so that $T'_{\scX_1}$ is the location of the global maximum of $\clB'$ on the interval $[0,1-T_{\scX_1}]$).  It thus  follows that the part of  $\clB'$ from $T'_{\scX_1}$ onwards is an excursion beyond the level $\scX_1$ carrying on beyond the time interval $[T'_{\scX_1},1-T_{\scX_1}]$ (observe that $1-T_{\scX_1}\geq T'_{\scX_1},$ by definition).
Now define the following alternating sequence of global minima and maxima  $$\scX_1=\scU_0, \scM_1,\scU_1,\scM_2,\scU_2,\ldots,$$ in the following inductive way: For any $j>0,$  $\scM_j$ is the minimum value attained by $\clB'$ between the time it attained $\scU_{j-1}$ and $1-T_{\scX_1}$; and $\scU_j$ is the maximum value attained by $\clB'$ between the time it attained $\scM_{j-1}$ and $1-T_{\scX_1}$. See 
the discussion following \eqref{formal56de} for the unicity of the times of attaining the values $\scM_1,\scU_1,\scM_2,\scU_2,\ldots.$  Moreover let $$\scX_2=\scX_1-\scM_1, \scX_3=\scU_1-\scM_1,\scX_4=\scU_1-\scM_2, \ldots$$ (see Figure \ref{brow} for an illustration. For formal definitions see \eqref{formal56de}).
Also let $\cE(n,1),\cE(n,2),\ldots$  and $\cW(n,1),\cW(n,2),\ldots$ respectively denote the lengths of the consecutive monochromatic runs starting furthest on the positive integer line and negative integer lines respectively (see  Figure \ref{fig:Ball2def1}). 
With the above preparation, we now state a refinement of Theorem~\ref{informal}.
 \begin{theorem}\label{mainruns} Let $C$ be as in Theorem \ref{informal}. Then,
\[\frac{(\cE(n,1),\cE(n,2),\ldots,\cE(n,k))}{n^{1/4}}\overset{d}{\Rightarrow} \frac{C}{2}\left(\sqrt{\scX_1}-\sqrt{\frac{\scX_2}{2}},\sqrt{\frac{\scX_2}{2}}-\sqrt{\frac{\scX_3}{2}},\ldots,\sqrt{\frac{\scX_k}{2}}-\sqrt{\frac{\scX_{k+1}}{2}}\right).\]
\end{theorem}
Above and throughout the paper, $\overset{d}{\Rightarrow} $ denotes  convergence in distribution. We will prove in Section \ref{s:combinatorics} that $\cW(n,i)$ and $\cE(n,i)$ differ by at most one and hence it suffices to consider just $\cE(\cdot,\cdot).$  
Similar statistics related to alternating extrema of a one dimensional Brownian motion were also the object of study in \cite{bianeyor}.

Although competitive erosion is far from any physical model, we have found two physical metaphors at times inspiring.

\subsection{The origin of large-scale structure in the universe}
 
It is thought that most matter annihilated with antimatter in the early universe, and that the large-scale structure of the remaining matter has its origin in random (quantum) fluctuations of the early universe.  The model considered in this paper resembles that scenerio in that nearly all particles are destroyed by a particle of opposite color, leaving only on the order $n^{1/4}$ surviving particles 
out of an initial $n$ red and $n$ blue.  These surviving particles are structured into long monochromatic intervals, even though the only mechanism for generating structure in our model is simple symmetric random walks performed by the particles.

Why does the universe apparently contain more matter than antimatter? 
Many offered explanations involve exotic physics  
beyond the Standard Model.  Another possibility to be considered, however, is a patchwork universe dominated by matter in some regions and by antimatter in others.  The most obvious sign that we live in a patchwork universe would be radiation emitted from the region boundaries where matter and antimatter meet. Measurements of the cosmic diffuse gamma-ray background imply that if such regions exist they must be large, on the same scale as the observable universe itself; see \cite{antimatter} and references therein.  
The question then arises whether it is possible, even in a mathematical toy universe, for an initially symmetric configuration of matter and antimatter with only local interactions to evolve \emph{macroscopic} asymmetries (in contrast to the mesoscopic fluctuations of the Gaussian and KPZ universality classes). 
Theorem~\ref{mainruns} shows that in one spatial dimension, the answer is yes.   

While our proof method is restricted to one spatial dimension, simulations suggest that the macroscopic structure in competitive erosion persists also in two and three dimensions (Figure~\ref{f.spiral}).

\subsection{Layering in sedimentary rock}

Our second metaphor comes from geology. In rock composed of distinct layers of accumulated sediment, 
the layers close to the earth's surface tend to be younger (i.e., deposited more recently) than the deeper layers. Each layer was deposited over a short period of geological time, but there can be large gaps in time between adjacent layers. These gaps reflect periods of alternating accumulation and erosion of sediment. The time gaps between deep layers tend to be longer than those between shallow layers (a phenomenon called the \emph{Sadler effect}, after \cite{Sadler}) so that the age of the rock increases faster than linearly with depth.  

In competitive erosion, we can think of each monochromatic interval of sites as a layer of rock. 
Define the ``age'' of a given site as the elapsed time since its most recent change in color.  At any given time, adjacent sites of the same color are likely to be close in age, but there is a large gap in age between differently colored adjacent sites. One can see these different time scales in Figure~\ref{f.sim}: as $n$ increases, it happens relatively often that the interval of colored sites expands, but its endpoints change color only very rarely.

\section{Key ideas and outline of the proofs}\label{s:outline}

A ``microstep'' in competitive erosion is a single random walk step of a single particle.
\begin{center}
\textbf{Throughout this article the total number of elapsed microsteps will be denoted by $t$, and the total number of particles emitted will be denoted by $n$.}
\end{center}
For $x \in \Z$, let $\sigma^t(x) \in \{-1,0,1\}$ indicate the color of site $x$ after $t$ microsteps (with $-1$ indicating Red, $0$  uncolored, and $+1$ Blue). Let $\widehat{\sigma}^t(x) \in \{-1,0,1\}$ indicate the color and position of the currently active particle after $t$ microsteps (with $-1$ indicating that $x$ has an active Red particle, $0$ that $x$ has no active particle, and $+1$ that $x$ has an active Blue particle). Since only one particle at a time is active, $\widehat{\sigma}^t(x)$ is nonzero for exactly one site $x \in \Z$. A key object in our analysis is the \emph{signed sum of positions},
\begin{equation}\label{martingale35}
 M^t = \sum_{x \in \Z} x \sigma^t(x) + 2 \sum_{x \in \Z} x \widehat{\sigma}^t(x).
\end{equation}
The first term is the signed sum of positions of all colored sites, and the second term is twice the signed position of the currently active particle.

An easily verified, but important, property of $M^t$ is that its increments $M^{t}-M^{t-1}$ are independent $\pm 2$ with probability $\frac12$ each, except at those microtimes $t$ when a previously uncolored site becomes colored.
The factor of two in \eqref{martingale35} ensures that this martingale property holds even at times when a colored site is converted to the opposite color. For example, when a Blue particle converts a Red site $x$ and a new active particle is born at $0$, the color conversion increases the first sum by $2x$, but the second sum decreases by $2x$ (as the position of the currently active particle is now $0$ instead of $x$).

\begin{figure}[h]
    \centering
    \includegraphics[width=0.65\textwidth]{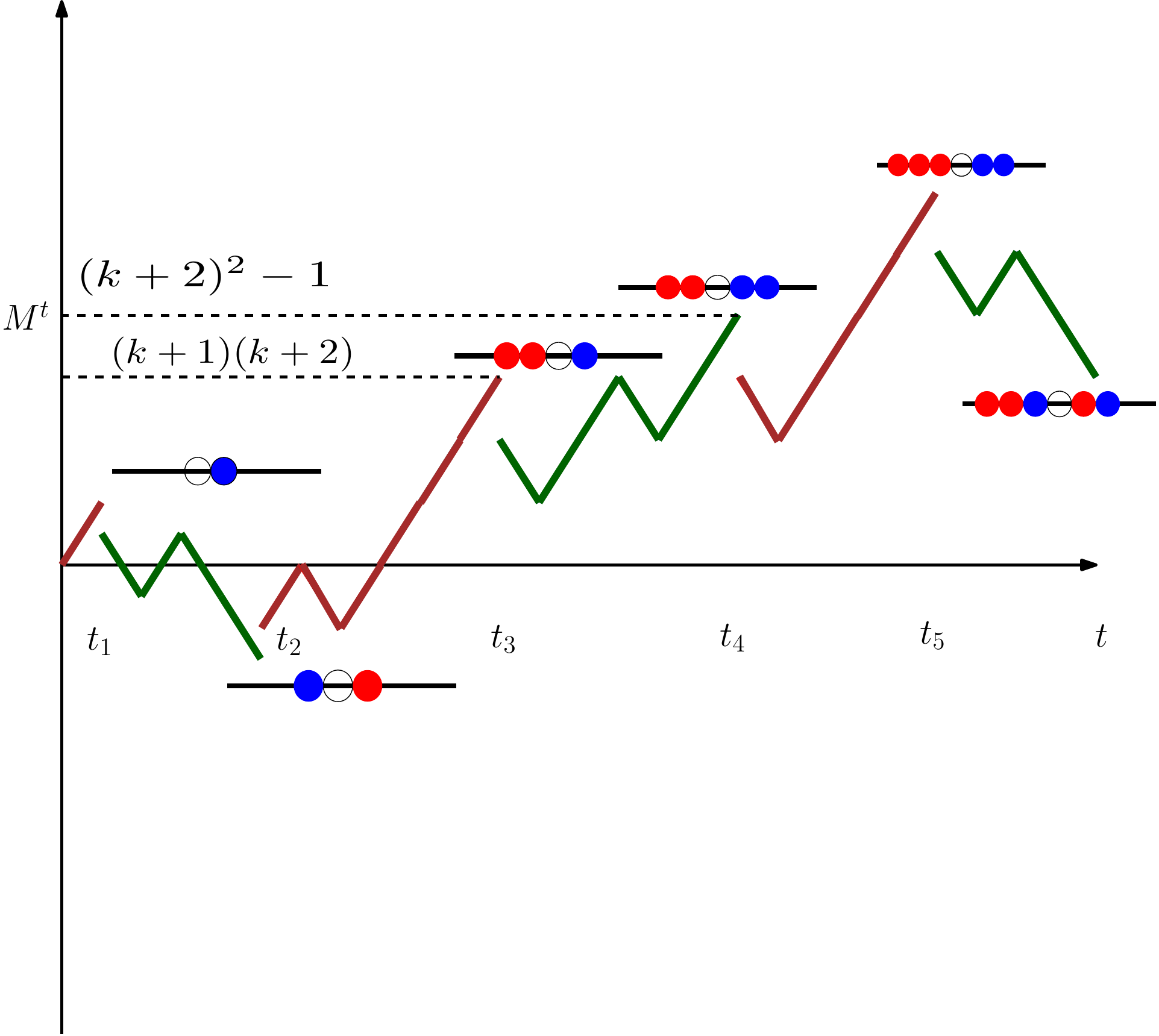}
   \caption{ The graph of $M^t$ in terms of the microstep $t$. The brown (odd) segments denote the random walks $\scE_0,\scE_1,\ldots$, which are  paths of step-size $2$ and $\scE_k$ up to a factor of $-1$,  starts from $k(k+1)$ and stops on hitting $\pm (k+1)(k+2).$ Similarly the green (even) segments  denote the random walks $\scF_0,\scF_1,\ldots$ which are  paths of step-size $2$ and $\scF_k$ up to a change in sign (depending on the endpoint of $\scE_k$) starts from $(k+1)^2$ and stops on hitting $\pm ((k+2)^2-1)$. At the  microsteps corresponding to the end of the individual odd or even segments, where new sites are occupied, there are jumps in $M^t$ from a brown segment to a green segment or vice-versa.
 As mentioned before $\scE_{k}$ captures the transition of the explored territory from $(k,k)$ to $(k+1,k)$ or $(k,k+1),$
while $\scF_k$ captures the transition from any of the latter to $(k+1,k+1).$
 We illustrate the various configurations corresponding to $M^t$ at times $t_1<t_2<\ldots< t_5 <t$   corresponding to the last microstep of every brown or green excursion. At all but the last one, a new site is  explored. Whenever two adjacent green and brown random walk paths  are stopped on hitting the same side of the $x-$ axis, the color on either side of the origin does not change. However if they are stopped on hitting opposite sides, the colors switch.
The last configuration corresponds to a partial green random walk path  and is the  configuration corresponding to $M^t$.}
    \label{fig:Mtfig}
\end{figure}

To prove Theorems \ref{informal} and \ref{mainruns}, we first prove versions of those theorems at the microstep time scale, based on the following observations:

$\bf{1}.$ We first prove certain useful combinatorial  properties of this process: as already stated before for e.g., the number of sites explored on each side of the origin can at most differ by $1$. Thus the set of explored sites on both sides of the origin increases from $(k,k+1)$ or $(k+1,k)$ to $(k+1,k+1)$, and then from $(k+1,k+1)$ to $(k+2,k+1)$ or $(k+1,k+2)$, and so on (the first coordinate represents the number of explored sites in $-\N$ and the second coordinate represents the number of explored sites in $\N$).

$\bf{2}.$ Now the key observation is that the microsteps corresponding to exploring new sites are related to hitting times of different level sets for the absolute value process $|M^t|$.
Let $\tau_{k,k}^{k,k+1}$ be the stopping time of a symmetric random walk of step-size $2$ starting from $k(k+1)$ till it hits $\pm (k+1)(k+2)$ and let $\tau_{k,k+1}^{k+1,k+1}$ be the stopping time of a symmetric random walk  of step-size $2$ starting from $(k+1)^2$ till it hits $\pm ((k+2)^2-1).$
It turns out (see Lemma \ref{l:Mrandwalk}), that the law of $M^t$ while going from state $(k,k)$ to either $(k+1,k)$ or $(k,k+1)$ is that of a symmetric random walk $\scE_k$ of step-size $2$ whose absolute value starts from $k(k+1)$ till it hits $\pm(k+1)(k+2).$ Similarly the law of $M^t$ while going from state $(k+1,k)$ or $(k,k+1)$ to $(k+1,k+1)$ is that of a symmetric random walk $\scF_k,$ of step-size $2$ starting from $(k+1)^2$ and stops on hitting $\pm((k+2)^2-1).$ 
Thus $\tau_{k,k}^{k,k+1}$ and $\tau_{k,k+1}^{k+1,k+1}$ are the corresponding hitting times.
Standard random walk facts imply that  $$\E \tau_{k,k-1}^{k,k}=k^3+k^2, \mbox{ and }\E \tau_{k,k}^{k,k+1}=(k+1)^3.$$ 
Thus, 
$\frac{\E A_{k,k}}{k^4}\rightarrow \frac{1}{2}$  as  $k\rightarrow \infty$ where $A_{k,k}$ is the total number of microsteps taken to reach the state $(k,k).$ This suggests we should scale the number of explored sites by $t^{1/4}$ to obtain a nontrivial limit.

$\bf{3}.$ From the above discussion it follows that at microtime $t$ the number of sites explored on each side is roughly $k$ if the maximum value attained by the process $(M^s)_{0\le s\le t}$ is approximately $k^2.$ We might suspect  that the number of explored sites, properly normalized, in the limit should  behave like the square-root of the maximum of the  absolute value of Brownian motion. 
However certain combinatorial constraints force it to be not quite the above but the  square-root of the following related quantity: 
\begin{equation}\label{opt323}
\sup_{a>0}\{T_a+T'_a \le 1\}
\end{equation}
 where $T_a$ and $T'_a$ are hitting times of $a$ for $|\clB(\cdot)|$ and $|\clB'(\cdot)|$ where $\clB(\cdot)$ and $\clB'(\cdot)$  are independent standard Brownian motions.   Note that without the $T'_a$ term, the above quantity would be exactly the maximum of the absolute value of Brownian motion run up to time $1$.
 
To see why \eqref{opt323} appears, note that since the ending point of $\scE_k$ is the starting point of $\scE_{k+1},$ one can concatenate the random walk paths $\{\scE_k\}_{k\ge1}$ to get an honest random walk path $\scE=(\scE_0; \scE_1;\scE_2;\ldots)$ of step size 2.  However the ending point of $\scF_k$ is not quite the starting point of $\scF_{k+1},$ but only differs by one in absolute value which has negligible contribution. Thus by suitable translation followed by concatenation of the segments $\scF_k$ one can obtain an independent random walk path $\scF$ (see \eqref{e:defF'} for formal constructions). 
Now the total number of microsteps is the number of steps travelled by $\scE$ plus the steps travelled by  $\scF$. 
 However by definition, the maximum of $\scE$ and maximum of $\scF$ are same up to a negligible error (both of them are close to  $k^2$, if $k$ sites have been explored on either side of the origin up to microstep $t$).
Thus normalizing by $t$ and applying Donsker's theorem, we get that $\frac{k^2}{\sqrt t}$ should converge to \eqref{opt323} up to certain deterministic multiplicative factors (see Theorem \ref{t:weakconvmicro} below for a precise statement.)

$\bf{4}.$ 
Given $t,$ as mentioned before, the state of the explored sites is $(k,k)$ if $$\sum_{j=0}^{k-1}\left[\tau_{j,j}^{j,j+1}+\tau_{j,j+1}^{j+1,j+1}\right]\le t< \sum_{j=0}^{k-1}\left[\tau_{j,j}^{j,j+1}+\tau_{j,j+1}^{j+1,j+1}\right]+\tau_{k,k}^{k,k+1}.$$ Similarly the state is $(k,k+1),$ if 
	\begin{equation} \label{e.asymmetric} \sum_{j=0}^{k-1}\left[\tau_{j,j}^{j,j+1}+\tau_{j,j+1}^{j+1,j+1}\right] +\tau_{k,k}^{k,k+1} \le t< \sum_{j=0}^{k}\left[\tau_{j,j}^{j,j+1}+\tau_{j,j+1}^{j+1,j+1}\right]. \end{equation}
Let us assume the latter case for the purposes of exposition. Thus the total number of microsteps $t$ is the total number of steps taken by the excursions $\scE_0,\ldots,\scE_k$ and $\scF_0,\ldots,\scF_{k-1}$ and a part of the excursion $\scF_k.$
Now by Brownian scaling  $\sum_{i=0}^{n}\tau_{j,j}^{j,j+1}$ should converge to $tT_{\scX_1}$ and $\sum_{i=0}^{n-1}\tau_{j,j+1}^{j+1,j+1}$ should converge to $tT'_{\scX_1}$ where $\scX_1$ is the argmax in  \eqref{opt323} ($k^2$ should be thought of as $\sqrt t \scX_1$). 

Recalling the notations from Figure \ref{fig:Ball2def1}, let $S^t$, $S^t_{\cE}$ and $S^t_{\cW}$ be the analogues of $S(n)$, $S_{\cE}(n)$ and $S_{\cW}(n)$ after $t$ microsteps, and similarly let $\cE^t(j),\cW^t(j)$ be the analogues of $\cE(n,j), \cW(n,j)$ (see Section \ref{fdn} for precise definitions).

By our assumption \eqref{e.asymmetric},
 the process $M^t$ is essentially the same as the concatenated path 
$$\scE_0;\scF_0;\scE_1;\scF_1;\scE_2;\scF_2;\ldots,\scE_k$$ along with a part of $\scF_{k}$ run from microstep $\sum_{j=0}^{k-1}\left[\tau_{j,j}^{j,j+1}+\tau_{j,j+1}^{j+1,j+1}\right] +\tau_{k,k}^{k,k+1}$.
The extremal runs $\cE^t(1),\cE^t(2), \ldots$ are now determined by this final part of $\scF_{k}$   which can be thought of as the random walk path $\scF=\scF_1;\scF_2;\ldots$ run from time $t T'_{\scX_1}$ to time $t(1-T_{\scX_1})$ (see Figure \ref{fig:Mtfig}).

It turns out that $\sqrt t \scX_1$ is an approximate local maximum for the path $\scF$ (and an exact local maximum after passing to Brownian motion). The subsequent alternating minima and maxima of $\scF$ are  related to the extremal runs for e.g.  the next global minima $\scM_1$ of $\scF_n$ in the time interval $[tT'_{\scX_1}, t(1-T_{\scX_1})],$ (attained at say, $t \scY_1$), is related to  $\cE^{t}(1)$ and subsequently $\scU_1$ (the global maxima in the time interval $[t \scY_1, t(1-T_{\scX_1}) ]$) is related to  $\cE^{t}(2)$ and so on (see Figure \ref{brow} for an illustration.) 

We give a short example to illustrate why this must be the case: Imagine the first time when the configuration is monochromatic up to distance $k$ 
with red on the right and blue on the left of the origin. As mentioned above $M^t$ reaches a value of  about $k^2$ at this point. Now suppose before any site at distance $k+1$ is explored,  the interval $[-\frac{2k}{3},\frac{2k}{3}],$ becomes monochromatic with blue on the right and red on the left. Thus at this point there are only two runs of red and blue on the right hand side as well as on the left hand side of the origin with $E^t(1)=\frac{k}{3}$ and $E^t(2)=\frac{2k}{3}.$
Also notice that at this point $M^t$ reaches the value of about $k^2-\frac{8k^2}{9}$ and is the current minimum after reaching the maximum $k^2$. Thus the difference between the maximum and the next minimum is related to the $E^t(1).$ One can continue this argument to describe the remaining runs.

In Section~\ref{s:scaling} we formalize the above to prove the following theorems:
\begin{theorem}\label{t:weakconvmicro} In the above setup, 
$\frac{S_\cE^t}{t^{1/4}}\overset{d}{\Rightarrow} \sqrt{2 \scX_1},$
where $\scX_1$ is the same as in Theorem \ref{informal}.
\end{theorem}
We prove later in Lemma \ref{l:comb} that 
$|S_\cE^t-S_\cW^t|\leq 1$ and hence the scaling limit of the support is twice the right side in the above theorem.
Recalling $\scX_1,\scX_2,\ldots,$ from Theorem \ref{mainruns}, the corresponding microstep version of the latter is the following: 
\begin{theorem}\label{runs56}
\begin{equation}\label{e:3}
\frac{(\cE^t(1),\cE^t(2),\ldots,\cE^t(k))}{t^{1/4}}\overset{d}{\Rightarrow} \sqrt{2}\left(\sqrt{\scX_1}-\sqrt{\frac{\scX_2}{2}},\sqrt{\frac{\scX_2}{2}}-\sqrt{\frac{\scX_3}{2}},\ldots,\sqrt{\frac{\scX_k}{2}}-\sqrt{\frac{\scX_{k+1}}{2}}\right).
\end{equation}
\end{theorem}
Note that  the two theorems just stated are in terms of the total number of microsteps $t$. 
To deduce our main results, it remains to compare the total number of microsteps to to the number of particles emitted. This is one of the key results in the paper. 
 Starting from the initial empty configuration, let $V(n)$ be the number of microsteps required for the first $n$ particles to settle down.   Theorems \ref{informal} and \ref{mainruns} follow in a relatively straightforward manner from Theorems \ref{t:weakconvmicro}, \ref{runs56} respectively and the  following  comparison result, where $\overset{\P}{\rightarrow}$ denotes convergence in probability.
 
\begin{theorem} {\em (Comparison of time scales)} \label{t:mainparticle}With the above notations,
$\frac{V(n)}{n}\overset{\P}{\rightarrow} \alpha,$
where $\alpha= (\frac{1}{2}-\sum_{j=1}^\infty \frac{1}{j(j+1)^2(j+2)})^{-1}.$ \end{theorem}

Below we outline the main ingredients involved in the proof of Theorem~\ref{t:mainparticle}.
Observe that, whenever a particle of some color, say red, emits from the origin, there is always a particle of opposite color, in this case blue, sitting at position $1$ or $-1$ (since the previous blue particle has settled somewhere, either it settled adjacent to the origin, or there was a run of blue particles adjacent to origin). 
Note that Theorem \ref{t:mainparticle} says that the number of particles emitted is comparable to the total number of steps taken by the particles. 
Now let us suppose that for some odd $n$, the length of the  run of blue vertices adjacent to the origin is $L$. Thus the $n^{th}$  particle emitted which is blue, stops after hitting the nearest red particle which is at distance $1$ on one side (by the above comment) and at distance $L+1$ on the other. 
It is a simple random walk fact that the expected number of steps taken by the blue particle is $L+1.$ 
Thus the key ingredient to proving Theorem \ref{t:mainparticle} is to show that on an average, $L$ is not too large. At a high level this means that new sites get explored only very rarely and most of the time particles kill particles of opposite color near the origin.

To make this formal,  we compare the original model with a killed version of the model, which is a renewal process. The killed model is as follows. We fix some integer $L$. Then we run the competitive erosion dynamics,  but whenever a particle jumps outside $[-L,L]$, we kill it, and again start by emitting a particle from the origin with the appropriate color depending on the parity of the round. Using certain combinatorial arguments we show that whenever a particle jumps outside $[-L,L]$,  the configuration in $[-L,L]$ is necessarily monochromatic on  both sides of the origin with opposite color, (see Section \ref{rene654}). This shows that the process is renewed every time a particle is killed. 
Another important property of the killed process is that under a natural coupling, both the original model and the killed model agree on the interval $[-L,L].$

The only thing left to show is that, this killed model well approximates the ratio of particles to microsteps of the original model for large $L$. 
By standard renewal process theory, the ratio of the particles to microsteps in this killed model is close to the ratio of the corresponding expectations. The latter is shown to converge to a constant as $L\rightarrow \infty$, by a suitable recursion relation between the process killed at $L$ and that killed at $L-1$. 
Now note that a discrepancy between the original model and killed model occurs only when a particle jumps out of the interval $[-L, L]$ which implies that one of the runs adjacent to the origin is at least $L.$  
We finally show that  in the original model the proportion of microsteps  when the length of the runs adjacent to the origin is at least $L$ is at most $O(\frac{1}{L})$ and hence taking $L$ large enough, the particle to microsteps ratio in the  killed model approximates that of the  original model to arbitrary precision.

The proof of this last  fact proceeds by showing that the proportion of microsteps  when the length of one of the runs adjacent to the origin is \textbf{exactly} $\ell$ is at most $O(\frac{1}{\ell^2})$ and hence the fraction of microsteps when one of the runs adjacent to the origin is \textbf{at least} $L$ is obtained by adding $O(\frac{1}{\ell^2})$ over $\ell$ from $L$ to $\infty$.
To prove the $O(\frac{1}{\ell^2})$ bound, observe that when the length of the run adjacent to the origin of the same color as the particle currently emitted (say of color red) is exactly  $\ell$,  the particle has a $\frac{1}{\ell+2}$ probability of hitting the blue particle at distance $\ell+1$ before hitting the blue particle adjacent to the origin, increasing the run length from  $\ell$ to $\ell+1.$ Thus on an average about $\ell$  particles have to be emitted to achieve this. Moreover each of these particles stop on hitting the boundary of the interval $[-1,\ell+1 ] $ and hence the expected number of microsteps taken in this process is about $\ell$ for each particle and hence is about $\ell^2$ in total. 
Now once the run length has exceeded $\ell+1,$ the only way the run can be $\ell$ again, is if a new blue sub-run of length $\ell$ is formed overriding the red run of length at least $\ell+1.$ 
It is not hard to verify that this will change the value of the martingale like process $M^t$ by $O(\ell^2)$ and hence it should roughly take $O(\ell^4)$ steps to achieve this. 

Thus for every at most $\ell^2$ steps spent when the run length is exactly $\ell$ there are at least $\ell^4$ steps where the run length is not $\ell.$ Thus the fraction of microsteps when the run length is exactly $\ell$ is approximately $\frac{1}{\ell^2}.$

\subsection{Organization of the paper} The paper is organized as follows. In the next section we develop the relevant notations and formal definitions used throughout the paper and also formally define the objects used in the statement of the main theroems. Some combinatorial observations about the model have been put together in Section \ref{s:combinatorics}. In Section \ref{s:scaling}, we prove Theorems \ref{t:weakconvmicro} and \ref{runs56} analyzing the potential function $M^t$. In Section \ref{s:particles}, we analyze the relation between the number of elapsed microsteps elapsed and the number of particles emitted, and prove Theorem \ref{t:mainparticle} and as a consequence finish the proofs of Theorems \ref{informal} and \ref{mainruns}. 

\subsection*{Acknowledgements} This work was initiated when SG and LL were visiting Microsoft Research; they thank Yuval Peres for his hospitality. 
SG thanks Jim Pitman for many helpful conversations and for pointing out reference \cite{bianeyor}. He also acknowledges the support by a Miller Research Fellowship. 
LL was supported by NSF \href{http://www.nsf.gov/awardsearch/showAward?AWD_ID=1455272}{DMS-1455272} and a Sloan Fellowship. 
SS was partially supported by Lo\`eve Fellowship.

\section{Formal definitions and other notations}\label{fdn} In this section, we shall introduce the model formally and the relevant notations to be used throughout the paper. Following the convention of the previous sections we will denote the red color by $\mathrm{R}$, blue color by $\mathrm{B}$ and a site with no color, by $0$. For any $\sigma \in \{0,\cR,\cB\}^{\Z\setminus \{0\}}$, and any $x\in \Z\setminus \{0\}$, let $\sigma(x)\in \{0,\cR,\cB\}$ denote the state of $\sigma$ at site $x$. Thus the competitive erosion can be thought of as a Markov chain on the following state space: 
\begin{eqnarray}\label{e:Omega}
\Omega=\left\{\sigma \in \{0,\cR,\cB\}^{\Z\setminus \{0\}}: \exists N_1,N_2\in \N\cup\{0\} \mbox{ such that } \sigma(x)=0 \quad \forall x\not\in [-N_1,N_2], \notag \right. \\
\left.\sigma(x)\neq 0 \quad \forall x \in [-N_1,N_2]\setminus\{0\}\vphantom{\{0,\cR,\cB\}^{\Z\setminus \{0\}}}\right\},
\end{eqnarray}
which is the set of all colorings of $\Z\setminus \{0\}$ with finitely many sites adjacent to the origin colored by red or blue. We define Competitive erosion as the Markov chain $X(n)$ on $\Omega$ with $X(0)(x)=0$ $\forall x \in \Z\setminus\{0\}.$

Now, consider a sequence $\{Y^{n}\}_{n\in \N}$ of independent random walks on $\Z$ starting from $0$, i.e., $Y^n(0)=0$, and $Y^n(i)$ denotes the position of the random walk $Y^n$ at time $i$. For each $n\in \N$, given the configuration $X(n-1)$, we define $X(n)$ as follows. Consider the set $U_\cB(n)$ of sites that were empty or occupied by red particles at time $n-1$, i.e.,  $U_\cB(n):= X(n-1)^{-1}\{0,\cR\},$ and similarly  $U_\cR(n):= X(n-1)^{-1}\{0,\cB\}.$
Then let 
$\tau_\cB(n):=\inf\{i\in \N:Y^{n}(i)\in U_\cB(n)\}$,
and 
$\tau_\cR(n):=\inf\{i\in \N:Y^{n}(i)\in U_\cR(n)\}.$
Now, for $n$ odd, define
\[X(n)(x)=X(n-1)(x) \quad \forall x\neq Y^{n}(\tau_\cB(n)),\quad  \mbox{and}\quad 
X(n)(Y^{n}(\tau_\cB(n)))=\cB .\]
Similarly, for $n$ even,
\[X(n)(x)=X(n-1)(x) \quad \forall x\neq Y^{n}(\tau_\cR(n))\quad\mbox{and} \quad
X(n)(Y^{n}(\tau_\cR(n)))=\cR .\]

\begin{remark}\label{r:explorenewsite}
If $X(n-1)(Y^n(\tau_\cB(n)))=0$, we say that the blue particle explores/occupies an unoccupied site, and if $X(n-1)(Y^n(\tau_\cB(n)))=\cR$ we say that the blue particle ``kills" a red particle and occupies its position. Also, if $X(n-1)(Y^n(\tau_\cB(n)))=0$, then clearly, $X(n-1)(x)=\cB$ for all $x\in [ 1,Y^n(\tau_\cB(n))-1]$ or $[ -Y^n(\tau_\cB(n))+1,-1 ]$ depending on  whether $Y^n(\tau_\cB(n))>0$ or $Y^n(\tau_\cB(n))<0$ respectively.
Similar comment applies to red particles as well. 
\end{remark}

Moreover,  as highlighted in Section \ref{s:outline}  the notion of `microsteps' where we count the steps of the individual random walks $\{Y^n\}$, will be useful. Thus it will be notationally convenient to define the following lifting of the Markov chain $X(n)$ on $\Omega$ discussed above, which will allow us to encode the microsteps as well.  Let $\Omega'\subset \Omega \times \Z \times \{\cR,\cB\}$, where $\Omega$ is defined in \eqref{e:Omega}, be
\[\Omega':=\{(\sigma,x,s):\sigma\in \Omega, x\in \{z:\sigma(z)\neq 0\}\cup \{0\}, s\in \{\cR,\cB\}\}.\]

Now, define the Markov Chain $Z^t=(Z_1^t,Z_2^t,Z_3^t)$ (where $Z^t_1$ denotes the configuration at microstep $t$, $Z_2^t$ denotes the currently moving particle, and $Z_3^t$ denotes the color of the currently moving particle) as follows. First, let
\[Z_1^0(x)=0 \quad \forall x\in \Z\setminus \{0\}, \quad Z_2^0=0, \quad Z_3^0=\cB.\] 
 Also, for any $n\in \N$, let
\begin{align}\label{hit78}
\tau^{n}&:=\tau_\cB(n) \text{ if } n \text{ is odd, and},\\
\nonumber
\tau^{n}&:=\tau_\cR(n) \text{ if } n \text{ is even}.
\end{align}

Then, for $\sum_{i=1}^{n-1}\tau^i\leq  t < \sum_{i=1}^{n}\tau^i$, define
\begin{equation}\label{e:defZ}
Z_1^t=X(n-1), \quad Z_2^t=Y^{n}\left(t-\sum_{i=1}^{n-1}\tau^i\right),
\end{equation}
and 
$Z_3^t=\cB$  if $n$ is odd, and $Z_3^t=\cR$ if $n$ is even.
Clearly, from definition, $\sum_{i=1}^n\tau^i$ denotes the index of the microstep when the $n$-th particle has settled, and
\[Z_1^{\sum_{i=1}^n\tau^i}=X(n).\]

In this notation the function $M^t$ in \eqref{martingale35} has the following description: 
\begin{equation}\label{e:defM}
M^t=\left[\sum_{x\in \Z} x\ind\left(Z_1^t(x)=\cB\right)-\sum_{x\in \Z} x\ind\left(Z_1^t(x)=\cR\right)\right]+\left[2Z_2^t\ind\left(Z_3^t=\cB\right)-2Z_2^t\ind\left(Z_3^t=\cR\right)\right].
\end{equation}

Given the above notation, we define formally the following quantities appearing in Theorems \ref{informal}, \ref{mainruns}, \ref{t:weakconvmicro}, and \ref{runs56}.  
Define the following set of random variables. Let
\begin{align}\label{e:defSE}
S_{\cE}^t&=S_{\cE}^t(1)=\sup\{x\in \N: Z_1^t(x)\neq 0\},\\
\label{e:defSW}
S_{\cW}^t&=S_{\cW}^t(1)=\sup\{x\in \N: Z_1^t(-x)\neq 0\}.
\end{align}

In these and the following definitions, whenever the relevant set is empty, we define the supremum to be $0$. 
Hence, $S_{\cE}^{t}, S_{\cW}^{t}$ denote the total number of sites occupied on either side of the origin at microstep $t$, and $S^t:=S_{\cE}^{t}+ S_{\cW}^{t}$ denotes the support at microstep $t$.

Once we have defined $S_{\cE}^t(i), S_{\cW}^t(i)$ for some $i\in \N$, let
\begin{align}\label{e:defSEi}
S_{\cE}^t(i+1)&=\sup\{x\in [ 1,S_{\cE}^t(i)]: Z_1^t(x)\neq Z_1^t(S_{\cE}^t(i))\}\\
\label{e:defSWi}
S_{\cW}^t(i+1)&=\sup\{x\in [ 1,S_{\cW}^t(i)]: Z_1^t(-x)\neq Z_1^t(-S_{\cW}^t(i))\}.
\end{align}
Also, we define for each $i\in \N$,
\begin{equation}\label{e:defEW}
\cE^{t}(i)=S_{\cE}^t(i)-S_{\cE}^t(i+1); \quad \cW^t(i)=S_{\cW}^t(i)-S_{\cW}^t(i+1).
\end{equation}
Observe that $\cW^{t}(i),\cE^{t}(i)$ denote the lengths of the $i$-th monochromatic run counted from the two ends at the left (west) and right (east) sides of the origin at microstep $t$. Also
\begin{equation}\label{sup45}
S_{\cE}^{t}(i)=\sum_{j\geq i} \cE^{t}(j), \text{ and } S_{\cW}^{t}(i)=\sum_{j\geq i} \cW^{t}(j).
\end{equation}

Let $\cR^t,\cB^t$ denote the number of red and blue particles at microstep $t$, i.e.
\begin{equation}\label{e:defRt}
\cR^t=\sum_{x\in \Z} \ind (Z^t_1(x)=\cR); \quad \cB^t=\sum_{x\in \Z} \ind (Z^t_1(x)=\cB).
\end{equation}

Also, if $\sum_{i=1}^n \tau^i$ denotes the microstep when the $n$-th particle has settled, then define, for each $k\in \N$, 
\begin{equation}\label{e:defSEpart}
S_{\cE}(n,k)=S_{\cE}^{\sum_{i=1}^n\tau^i}(k); \quad S_{\cW}(n,k)=S_{\cW}^{\sum_{i=1}^n\tau^i}(k),
\end{equation}
and
\begin{equation}\label{e:defRpart}
\cR(n)=\cR^{\sum_{i=1}^n\tau^i}; \quad \cB(n)=\cB^{\sum_{i=1}^n\tau^i},
\end{equation}
and
\begin{equation}\label{e:defEWpart}
\cE(n,i):=S_{\cE}(n,i)-S_{\cE}(n,i+1), \quad \cW(n,i):=S_{\cW}(n,i)-S_{\cW}(n,i+1),
\end{equation}
as the corresponding quantities at usual step $n$ for the Markov chain $X(n)$. Similarly the support after $n$ particle have settled down is 
$
S(n):=S^{\sum_{i=1}^n\tau^i}.
$

We also provide the formal definitions of $\scM_i, \scU_i$ appearing in the statement of Theorem \ref{mainruns}. Recall $\clB, \clB', \scX_1$ from Theorem \ref{informal}. Working with the same assumption as stated right after \eqref{alm464},  that  $\scX_1$  is the maximum of $\clB'(s)$ for $s\in [0,1-T_{\scX_1}]$, standard facts imply that it  is almost surely attained uniquely at a point in the open interval $(0,1-T_{\scX_1}).$ 
Moreover considering the excursion below level $\scX_1$ of $\clB'$ on the interval $[T'_{\scX_1}, 1-T_{\scX_1}]$ 
we define, 
 \begin{equation}\label{formal56de}
 \scM_1 :=\min \{\clB'(s):T'_{\scX_1}\leq s\leq 1-T_{\scX_1}\},
 \end{equation}
  and let the unique time between $T'_{\scX_1}$ and $1-T_{\scX_1}$ at which this value $\scM_1$ is attained be $\scY_1$. Let 
  \[\scU_1:=\max\{\clB'(s): \scY_1\leq s\leq 1-T_{\scX_1}\},\] 
 and let the unique time\footnote{Using standard arguments (see for e.g. the chapter on Brownian excursion in \cite{revuz}) one can show that almost surely $\scZ_{k-1}<\scY_k<1-T_{\scX_1}$  and similarly $\scY_{k}<\scZ_k<1-T_{\scX_1}.$  Uniqueness of $\scY_k$, and $\scZ_{k}$ follows, from a similar argument as in Theorem $2.11$ of \cite{MP}. 
} at which this value is obtained by $\clB'$ be $\scZ_1$. 
 In general, once we have defined $\scM_k,\scU_k,\scY_k,\scZ_k$, we define
 \[\scM_{k+1}:=\min \{\clB'(s):\scZ_k\leq s\leq 1-T_{\scX_1}\},\]
 and $\scY_{k+1}$ as the unique time between $\scZ_k$ and $1-T_{\scX_1}$ at which the value $\scM_{k+1}$ is attained. 
and similarly
 \[\scU_{k+1}:=\max \{\clB'(s):\scY_{k+1}\leq s\leq 1-T_{\scX_1}\},\]
 and $\scZ_{k+1}$ as the unique time between $\scY_{k+1}$ and $1-T_{\scX_1}$ at which the value $\scU_{k+1}$ is attained.
Define 
\begin{equation}\label{e:defXi}
\scX_2=\scX_1-\scM_1, \scX_3=\scU_1-\scM_1,\scX_4=\scU_1-\scM_2, \ldots.
\end{equation}
(If $\clB'(T'_{\scX_1})<0$, replace $\clB'$ by $-\clB'$, $\scM_i,\scU_i,\scX_i$ by $-\scM_i,-\scU_i,-\scX_i$ respectively).  Note that the above discussion implies that $\scX_i> \scX_{i+1}$ almost surely.

\section{Some combinatorial observations}\label{s:combinatorics}
In this section, we put together some combinatorial observations that will be used throughout the paper.  However the proofs sometimes are a bit tedious and skipping the proofs in this section will not affect readability of the future sections.  
Recall the definitions of $S_{\cE}(n,k),S_{\cW}(n,k),\cR(n),\cB(n)$ from \eqref{e:defSEpart} and \eqref{e:defRpart}.
 
\begin{lemma}\label{l:comb} For all $n\in \N$,
$|S_{\cE}(n)-S_{\cW}(n)|\leq 1.$
 Also, for $n$ odd, 
$\cB(n)-\cR(n)\geq 0,$
 and analogously for $n$ even, $\cR(n)-\cB(n)\geq 0.$ 
 (Recall that blue particles are emitted at odd steps).
\end{lemma}
\begin{proof}This lemma follows by observing that $\cB(n)-\cR(n)$ changes by either $1$ or $2$ at every step, and it changes by $1$ only when an unoccupied site is occupied by the emitted particle. Similarly, $S_{\cE}(n)-S_{\cW}(n)$ also changes by at most $1$, and only when an unoccupied site is explored. Moreover, whenever an unoccupied site, say on the right side of the origin, is occupied by a particle, say blue, all the sites on the right side of the origin must have contained only blue particles (see Remark \ref{r:explorenewsite}). These observations, together with induction, proves the lemma. The following is the induction hypothesis: 
 \begin{align*}
 |S_{\cE}(m)-S_{\cW}(m)|&\leq 1 \quad  \mbox{ for all } m\leq n, \text{ and,}\\
\cB(m)-\cR(m)& \geq 0 \quad \mbox{for all odd } m\leq n, \text{ and,}\\
\cR(m)-\cB(m)& \geq 0 \quad \mbox{for all even } m\leq n.
\end{align*}
We prove that the statements hold for $m=n+1$. Without loss of generality, we assume that $(n+1)$ is odd (so that at the $(n-1), n$ and $n+1$-th steps  blue, red and blue particles are emitted respectively).  First we show that 
\begin{equation}\label{e:toprove}
\cB(n+1)-\cR(n+1)\geq 0.
\end{equation}
To this end, first note from the mechanism of the erosion model, at every odd step $m$ when a blue particle is emitted, either it kills a red particle and occupies its position, whence 
\[\cB(m)-\cR(m)=(\cB(m-1)-\cR(m-1))+2,\] 
 or the blue particle occupies an empty site, in which case 
 \[\cB(m)-\cR(m)=(\cB(m-1)-\cR(m-1))+1.\]
  An analogous observation can be made for the even step when a red particle emits. Since the difference between the number of red and blue particles can change by at most $2$ at every step, the only way in which \eqref{e:toprove} can fail to happen given the induction hypothesis, is to have  the following:
\[\cB(n-1)-\cR(n-1)=0, \quad \cR(n)-\cB(n)=2, \quad \cB(n+1)-\cR(n+1)=-1.\]
Thus, assume $\cB(n-1)=\cR(n-1)=k$
for some $k\in \N$. Then 
\[\cB(n)=k-1,\cR(n)=k+1, \mbox{ and } \cB(n+1)=k,\cR(n+1)=k+1.\]
 The values at steps $n$ and $n+1$ together imply that a new site is explored by a blue particle at $(n+1)$-th step, which in turn implies that at $n$-th step, one side of the origin, say the right/east side, consisted only of blue particles (see Remark \ref{r:explorenewsite}). Since  $\cB(n)=k-1$, this implies 
 \[S_{\cE}(n)\leq \cB(n)\leq k-1.\] 
 Since the total number of particles at step $n$ was 
 $S_{\cE}(n)+S_{\cW}(n)=\cB(n)+\cR(n)=2k,$
we have $S_{\cW}(n)\geq k+1,$
and hence $S_{\cW}(n)-S_{\cE}(n)\geq 2,$ which contradicts the induction hypothesis. This proves \eqref{e:toprove}.
The only thing left to show is that 
\begin{equation}\label{e:toprove2}
|S_{\cE}(n+1)-S_{\cW}(n+1)|\leq 1.
\end{equation}
We argue in a similar fashion as above. At any step $m$ after a particle settles, either exactly one of $S_{\cE}(m)$ or $S_{\cW}(m)$ increases by  $1$, or both of them remain the same. Hence, there is nothing to prove if $S_{\cW}(n)=S_{\cE}(n)$. The only other alternative allowed by the induction hypothesis is 
\[|S_{\cW}(n)-S_{\cE}(n)|=1.\] 
Without loss of generality, assume 
\begin{equation}\label{e:assp}
S_{\cE}(n)=S_{\cW}(n)-1.
\end{equation}
From this, the only way \eqref{e:toprove2} will not hold is if a new site is explored by the blue particle at step $n+1$, and the new explored site is on the left side of the origin. This ensures that the left side has all blue particles at step $n$. But then \[\cR(n)-\cB(n)\leq S_{\cE}(n)-S_{\cW}(n)<0,\] 
from \eqref{e:assp}, which contradicts the induction hypothesis as $n$ is even. This completes the induction step and proves the lemma.
\end{proof}

The following lemma follows directly from Lemma \ref{l:comb}.
\begin{lemma}\label{l:r-b} With the above definitions, for all $n\in \N$, 
\begin{equation}\label{e:r-b}
|\cR(n)-\cB(n)|\leq 2.
\end{equation}
\end{lemma}
\begin{proof} Without loss of generality, assume that $n$ is odd. If \eqref{e:r-b} does not hold, then, because of Lemma \ref{l:comb}, 
$\cB(n)\geq \cR+3.$
Then, at step $n+1$, since $\cR(n+1)-\cB(n+1)$ can increase by at most $2$, hence, 
\[\cR(n+1)-\cB(n+1)<0,\] 
contradicting Lemma \ref{l:comb} as $n+1$ is even.
\end{proof}

\subsection{Configuration at steps of occupation of new sites}\label{rene654} In this subsection, we show that the configuration of colors, at the end of a round when the emitted particle has settled in an unoccupied site, looks monochromatic on either side of the origin. Recall that this observation makes the killed process described in Section \ref{s:outline} as a renewal process. 
 
Without loss of generality assume that the step $n$ is odd, so that the emitted particle is blue, and it settles at an unoccupied site, say on the right side of the origin. Then clearly, the right side of the origin had a string of only blue particles at step $n-1$ (see Remark \ref{r:explorenewsite}). What was the configuration at that step $n-1$ on the left side of the origin? We claim that, even on the left side, at step $n-1$, there was a string of only red particles. This is the content of the next lemma.  
Recall the definitions of $\tau^n$ from \eqref{hit78} and $S_{\cE}(n):=S_{\cE}(n,1),S_{\cW}(n):=S_{\cW}(n,1)$ from \eqref{e:defSEpart}.
\begin{lemma}\label{l:explorconfig} If at a step $n$, a new site is occupied by the emitted particle, then at step $n-1$, both sides of the origin were monochromatic and of opposite color, i.e., if $X(n-1)(Y^n(\tau^n))=0$, then 
\begin{align*}
X(n-1)(x) &=X(n-1)(y) \quad \mbox{ for all } x,y\in [ 1, S_{\cE}(n-1) ],\\
X(n-1)(x) &=X(n-1)(y) \quad \mbox{ for all } x,y\in [  -S_{\cW}(n-1),-1 ],\end{align*}
and 
\[X(n-1)(x)\neq X(n-1)(y) \quad \mbox{ for all } x\in [ 1, S_{\cE}(n-1) ], y\in [  -S_{\cW}(n-1),-1 ].\]
\end{lemma}
Clearly from the above lemma, at step $n$ as well, both sides of the origin remain monochromatic and of opposite color. 
The proof of the above is a direct application of Lemma \ref{l:comb}.
\begin{proof}
 Assume without loss of generality that $n$ is odd and a new site is explored on the right side at step $n$, i.e., $Y^n(\tau^n)>0$. Hence,
 \[X(n-1)(x)=\cB \quad \mbox {for all } x\in [ 1,S_{\cE}(n-1) ].\]
 Since,
 \[S_{\cE}(n)-S_{\cW}(n)=S_{\cE}(n-1)-S_{\cW}(n-1)+1,\]
 hence because of Lemma \ref{l:comb}, at step $n-1$, there are the following two possibilities: 
 \begin{itemize}
 \item Case $(1)$: $S_{\cE}(n-1)=S_{\cW}(n-1)$
 \item Case $(2)$: $S_{\cE}(n-1)=S_{\cW}(n-1)-1$
 \end{itemize}

For Case $(1)$: Let $S_{\cE}(n-1)=S_{\cW}(n-1)=k$ for some $k\in \{0,1,2,\ldots\}$. Since $n-1$ is even, by Lemma \ref{l:comb},
 $\cR(n-1)\geq \cB(n-1).$
Since  $\cB(n-1)\geq S_{\cE}(n-1)=k,$ and the total number of occupied sites
 \[\cB(n-1)+\cR(n-1)=S_{\cE}(n-1)+S_{\cW}(n-1)=2k,\] 
 it forces the left side to have all its $k$ particles red.

For Case $(2)$: Let, 
\[S_{\cE}(n-1)=k,\,\, S_{\cW}(n-1)=k+1 \quad \mbox{for some } k\in \{0,1,2,\ldots\}.\] 
If the left side contains at least one blue particle, then 
$\cB(n-1)\geq S_{\cE}(n-1)+1=k+1,$
 and 
$\cR(n-1)\leq S_{\cW}(n-1)-1=k.$
 Hence $\cB(n-1)>\cR(n-1)$, which contradicts Lemma \ref{l:comb} as $n-1$ is even.
\end{proof}

\subsection{Formation of layers}
A different and useful way of looking at the configuration of the various runs at a particular step, is to look at how the whole process develops as layers one on top of the other. 
 Fix $n\in \N$, and let
\[\mbox{NR}_\cE(n):=\sum_{x\in [ 1,S_{\cE}(n)-1]}\ind(X(n)(x)\neq X(n)(x+1))+1\]
denote the total number of runs on the right side of the origin at step $n$, and
\[\mbox{NR}_\cW(n):=\sum_{x\in [ S_{\cW}(n),-2]}\ind(X(n)(x)\neq X(n)(x+1))+1\]
denote the total number of runs on the left side of the origin at step $n$. Define,
\begin{equation}\label{e:defnmax}
\mathfrak{M}(n):=\mathfrak{M}(n,1):=\max\{k\in [ 1,n ]: \mbox{NR}_\cE(k)\leq 1, \mbox{NR}_\cW(k)\leq 1\},
\end{equation}
and let
\begin{equation}\label{e:deflayer1}
(L_\cE(n,1), L_\cW(n,1)):=(S_{\cE}(\mathfrak{M}(n)), S_{\cW}(\mathfrak{M}(n))).
\end{equation}
Thus after the $n^{th}$ particle has settled down, $\mathfrak{M}(n)$ is the most recent time when there was exactly one run on each side of the origin; and the those runs on each side comprise the first layer.

As the number of explored sites can only increase, and by Lemma \ref{l:explorconfig}, at the steps of exploration of new sites, the two sides are monochromatic,  it follows that $L_\cE(n,1)=S_{\cE}(n)$ and $L_\cW(n,1)=S_{\cW}(n)$ \footnote{Note however that  $\mathfrak{M}(n)$ may not be the last step where the maximum monochromatic run length is achieved individually on any particular side of the origin. For example, if $(L_\cE(n,1),L_\cW(n,1))=(k,k)$, and again, for some $r\in (\mathfrak{M}(n),n]$, one has $X(r)([-k,-1])=(\cB,\cR,\ldots,\cR)$ and $X(r)([1,k])=(\cB,\ldots,\cB)$, then $\mathfrak{M}(n)$ is not the last step where the maximum monochromatic run length is achieved on the right side of the origin).}.  
We consider the configuration after step $\mathfrak{M}(n,1)$ till step $n$. Clearly, no new site is explored. Consider all the steps between $\mathfrak{M}(n,1)$ and $n$ when there are at most two runs on either side of the origin, and let $\mathfrak{M}(n,2)$ denote the last step among these, i.e.,
\[\mathfrak{M}(n,2):=\max\{k\in [ \mathfrak{M}(n,1),n]: \mbox{NR}_\cE(k)\leq 2, \mbox{NR}_\cE(k)\leq 2\},\]
 and let  
 \begin{equation}\label{e:deflayer2}
 (L_\cE(n,2), L_\cW(n,2)):=(S_{\cE}\large(\mathfrak{M}(n,2),\mbox{NR}_\cE(\mathfrak{M}(n,2))\large), S_{\cW}\large(\mathfrak{M}(n,2), \mbox{NR}_\cW(n,2)\large)),
 \end{equation}
where $S_{\cE}(n,i), S_{\cW}(n,i)$ are as defined in \eqref{e:defSEpart}.
We define $(L_\cE(n,i),L_\cW(n,i))$ for $i\in \{3,4,\ldots\}$ successively in a similar fashion.

\begin{lemma}\label{l:layers} We have that $L_\cE(n,i)$ and $L_\cW(n,i) $ are non increasing in $i$ for each $n$ and 
\[|L_\cE(n,i)-L_\cW(n,i)|\leq 1\] 
for all $i$. Also the pairs of layers $(L_\cE(n,i),L_\cE(n,i+1)),$  $(L_\cW(n,i),L_\cW(n,i+1))$ and  $(L_\cE(n,i),L_\cW(n,i))$ for all $i$ are of opposite colors. 
\end{lemma}
\begin{proof} We only present a sketch of the proof omitting the details. Fix any $i\in \N$ and consider the $i$-th layer $(L_\cW(n,i),L_\cE(n,i))$. After step $\mathfrak{M}(n,i-1)$, whenever a particle kills another particle of opposite color from previous layer $(L_\cW(n,i-1),L_\cE(n,i-1))$, pretending that as an  exploration of a new site allows us to use the inductive arguments in the proofs of Lemmas \ref{l:comb} and \ref{l:r-b} and Lemma \ref{l:explorconfig}.  Thus considering only the particles emitted after $\mathfrak{M}(n,i-1)$ we recover the statements of these lemmas for the $i$-th layer and this completes the proof.  
\end{proof}
\begin{definition}  For any $j\in \N$, we define the modified run lengths (counted from the ends) which will be useful later.
Let 
\begin{equation}\label{e:defmodrun}
\cE_m(n,j):=L_\cE(n,j)-L_\cE(n,j+1), \quad \cW_m(n,j)=L_\cW(n,j)-L_\cW(n,j+1).
\end{equation}
\end{definition}
\begin{remark}\label{r:modrun}
Observe that the non-zero elements of the modified run lengths are exactly equal to the usual run lengths defined earlier in \eqref{e:defEWpart} and in the correct order. The only difference is that $\cE_m(n,j)=0$ can occur for some $j$. For such a $j$, $\cW_m(n,j)=1$ because of Lemma \ref{l:layers}. Similarly $\cW_m(n,j)=0,\cE_m(n,j)=1 $ can also occur. Also for any $j$, the runs corresponding to $\cE_m(n,j)$ and $\cW_m(n,j)$ are of opposite colors (allowing the possibility that one of them can be of length $0$). Hence, if one knows the run lengths on one side of the origin, one gets the run lengths within $\pm 1$ of the other side and their colors.
\end{remark}

Also, for any $t\in \N$, such that $\sum_{i=1}^{n-1}\tau^i\leq t<\sum_{i=1}^{n}\tau^i$, where $\tau^i$ are defined in \eqref{hit78}, let
\begin{equation}\label{e:defmaxmicro}
\mathfrak{M}^t:=\sum_{i=1}^{\mathfrak{M}(n)-1}\tau^i, \quad \mathfrak{M}^t(j):=\sum_{i=1}^{\mathfrak{M}(n,j)-1}\tau^i,
\end{equation}
and
\begin{equation}\label{e:deflayermicro}
(L_\cE^t(j),L_\cW^t(j)):=(L_\cE(n-1,j),L_\cW(n-1,j)),
\end{equation}
and the modified run lengths
\begin{equation}\label{e:defmodrunmicro}
\cE_m^t(j):=L_\cE^t(j)-L_\cE^t(j+1), \quad \cW_m^t(j)=L_\cW^t(j)-L_\cW^t(j+1).
\end{equation}
The above defined modified run lengths would be convenient for the proofs of Theorems \ref{mainruns} and \ref{runs56}.
Note that the length of the layers are non-increasing and on the event that they are strictly decreasing, the modified run lengths $\cE_m^t(\cdot)$  and the original run lengths $\cE^t(\cdot)$  are the same (This is shown to occur with high probability in Lemma \ref{l:modrunsame}).

\section{Scaling limit in microstep time scale} \label{s:scaling}
In this section we prove Theorems \ref{t:weakconvmicro} and \ref{runs56}. However to get started we need a bit of notation. 
For any $k\in \N$, define,
\begin{align}\label{e:Ann}
A_{k,k}&:=\inf\{n\in \N: (S_{\cW}(n),S_{\cE}(n))=(k,k)\},\\
\label{e:Ann+1}
A_{k,k+1}&:=\inf\{n\in \N: (S_{\cW}(n),S_{\cE}(n))=(k,k+1) \mbox { or }(S_{\cW}(n),S_{\cE}(n))=(k+1,k) \},
\end{align}
and let
\begin{equation}\label{e:AMnn}
A^M_{k,k}:=\sum_{i=1}^{A_{k,k}} \tau^i,\quad A^M_{k,k+1}:=\sum_{i=1}^{A_{k,k+1}} \tau^i
\end{equation}
be the microsteps corresponding to the step $A_{k,k},A_{k,k+1}$ respectively ($\tau^i$ was defined in \eqref{hit78}).

By Lemma \ref{l:comb}, it follows that, $A_{0,1}<A_{1,1}<\ldots<A_{k-1,k-1}<A_{k-1,k}<A_{k,k}<\ldots$
are precisely the steps where new sites are explored and the support $S(n)$ increases (see Remark \ref{r:explorenewsite}). Also, by Lemma \ref{l:explorconfig}, at any of these steps $A_{k,k}$ or $A_{k,k+1}$, both sides of the origin are monochromatic and of opposite color. As already stated in Section \ref{s:outline},  a simple but key observation in the paper is that $M^{t}$ behaves like a random walk between certain times. This is the content of the next lemma (see also Figure \ref{fig:Mtfig} for an illustration).
\begin{lemma}\label{l:Mrandwalk} Fix any $k\in \N\cup\{0\}$. Let $\tau_{k,k}^{k,k+1}$ be the stopping time of a symmetric random walk $\tilde{E}_k(j)$ of step-size $2$ starting from $k(k+1)$ till it hits $\pm (k+1)(k+2)$. Then, if $M^{A_{k,k}^M}\geq 0$, then,
\[\{M^t\}_{t\in [ A_{k,k}^M,A_{k,k+1}^M-1]}\overset{d}{=}\{\tilde{E}_k(j)\}_{j\in [ 0,\tau_{k,k}^{k,k+1}-1]} ,\]
and if $M^{A_{k,k}^M}\leq 0$, then,
\[\{M^t\}_{t\in [ A_{k,k}^M,A_{k,k+1}^M-1]}\overset{d}{=}\{-\tilde{E}_k(j)\}_{j\in [ 0,\tau_{k,k}^{k,k+1}-1]}.\]

Similarly, for $k\in \N\cup\{0\}$, let $\tau_{k,k+1}^{k+1,k+1}$ be the stopping time of a symmetric random walk $\tilde{F}_k(j)$ of step-size $2$ starting from $(k+1)^2$ till it hits $\pm ((k+2)^2-1)$. Then, if $M^{A_{k,k+1}^M}\geq 0$, then,
\[\{M^t\}_{t\in [ A_{k,k+1}^M,A_{k+1,k+1}^M-1]}\overset{d}{=}\{\tilde{F}_k(j)\}_{j\in [ 0,\tau_{k,k+1}^{k+1,k+1}-1]} ,\]
and if $M^{A_{k,k+1}^M}\leq 0$, then,
\[\{M^t\}_{t\in [ A_{k,k+1}^M,A_{k+1,k+1}^M-1]}\overset{d}{=}\{-\tilde{F}_k(j)\}_{j\in [ 0,\tau_{k,k+1}^{k+1,k+1}-1]}.\]
Above $\overset{d}{=}$ denotes equality in distribution.
\end{lemma}

\begin{proof}The proof follows from the observation that $M^t$ reaches an absolute value $(k+1)(k+2)$ for the first time when $(S^t_{\cW},S^{\cE})=(k,k+1) \text{ or } (k+1,k)$. (Note from \eqref{martingale35} that the first time a site is explored, it contributed twice the weight). However since in the next round the new site explored only contributes $(k+1)$ and not $2(k+1)$, the absolute value of the process $M^t$ can be thought to have  an instantaneous jump down from $(k+1)(k+2)$ to $(k+1)(k+2)-(k+1)=(k+1)^2.$ 
Also notice that $M^t$ reaches an absolute value $(k+2)^2-1$ for  the first time when $(S^t_{\cW},S^t_{\cE})=(k+1,k+1).$ (in this case $M^t$ instantaneously jumps down to $(k+2)^2-1-(k+1)=(k+1)(k+2).$ )
The above, along with the observation that $M^t$ is a random walk of step size $2$ till a new site gets explored completes the proof. 
\end{proof}

For notational brevity in the sequel,  we will denote 
 the random walk $\{\tilde{E}_k(j)\}_{j\in [ 0,\tau_{k,k}^{k,k+1}]}$ or $\{-\tilde{E}_k(j)\}_{j\in [ 0,\tau_{k,k}^{k,k+1}]}$ for the first and second case respectively by $\scE_k$ and similarly, denote the random walk $\{\tilde{F}_k(j)\}_{j\in [ 0,\tau_{k,k+1}^{k+1,k+1}]}$ or $\{-\tilde{F}_k(j)\}_{j\in [ 0,\tau_{k,k+1}^{k+1,k+1}]}$ for the third and fourth case respectively by $\scF_k$.

\subsection{Proof of Theorem \ref{t:weakconvmicro}} \label{po1}  As outlined in Section \ref{s:outline}, the main idea of the proof is to use  Lemma \ref{l:Mrandwalk}. By joining the alternate segments of $M^t$, we get two independent random walks, each of which hits values close to $\pm k^2,$ if and only if the $k$-th site is explored. The proof then follows by an application of Donsker's theorem.
Recall the definitions of $S_{\cE}^t, S_{\cW}^t, S^t, A_{k,k}$ from Section \ref{fdn} and \eqref{e:Ann}. 
Recall that, if $S_{\cE}^t=k$, then $k-1\leq S_{\cW}^t \leq k+1$ by Lemma \ref{l:comb}, and moreover by Lemma \ref{l:Mrandwalk}, 
\[k(k+1)\leq \cZ_t:=\max_{s\leq t}|M^s|\leq (k+1)(k+2).\]
So that,\begin{equation}\label{e:conmaxM}
S_{\cE}^t\leq \sqrt{\cZ_t}\leq S_{\cE}^t+2.
\end{equation}
Thus to prove Theorem \ref{t:weakconvmicro}, it suffices to show the weak convergence of
$\frac{\cZ_t}{t^{1/2}}$. Also for any random walk path $\cS$ stopped at  time $\tau$, let $N(\cS)=\tau$ (the length of the random walk path), and let $\cS(0)$ and $\cS(N(\cS))$ denote the starting and ending points of the random walk path $\cS$. 
Often we will need to translate a random walk path $\cS$ by a number $s$, and we will denote the translated path by $\mathcal{S}-\{s\}$ which is clearly a random walk path started at 
$\cS(0)-s$.

Consider  $\scE_i$ and $\scF_i,$  the alternate segments of random walks of step size $2$ contained in $M^t$ as defined above. Then by definition, for any $k\in \N$, 
\begin{equation}\label{e:Z1}
(k+1)^2-1\leq \cZ_t < (k+1)(k+2) \quad \mbox{ iff } \sum_{0\le i\leq k-1}(N(\scE_i)+N(\scF_i))\leq t<\sum_{0\le i\leq k-1}(N(\scE_i)+N(\scF_i))+N(\scE_k).
\end{equation} 
Similarly, 
\begin{equation}\label{e:Z2}
(k+1)(k+2)\leq \cZ_t<(k+2)^2-1 \quad \mbox{ iff } \sum_{0\le i\leq k-1}(N(\scE_i)+N(\scF_i))+N(\scE_k)\leq t<\sum_{0\le i\leq k}(N(\scE_i)+N(\scF_i)).
\end{equation}  
 Also by definition, the starting and ending points of the random walk segments have the following properties: for all $i\geq 1$, 
\[|\scE_i(0)|=|\scE_{i-1}(N(\scE_{i-1}))|; \mbox{ and } |\scF_i(0)|=|\scF_{i-1}(N(\scF_{i-1}))|+1.\]
 Now, we define a random walk path by joining the segments of $\scE_i$ `end-to-end'. More precisely, we define $\scE_0'=\scE_0$, and for all $i\geq 1$, define 
 \begin{equation}\label{e:defE'}
 \scE_i'=\scE_i\quad \mbox{ if } \scE_i(0)=\scE'_{i-1}(N(\scE'_{i-1})), \quad \mbox{ else } \scE_i'=-\scE_i.
 \end{equation}
  Then the concatenated walk $\scE''_k:=(\scE'_0;\scE'_1;\scE'_2;\ldots;\scE'_k)$ is a random walk of step-size $2$ starting from $0$ till it hits $\pm (k+1)(k+2)$, (see Figure \ref{fig:operation} (a)).
\begin{figure}[h]
    \centering
    \includegraphics[width=0.65\textwidth]{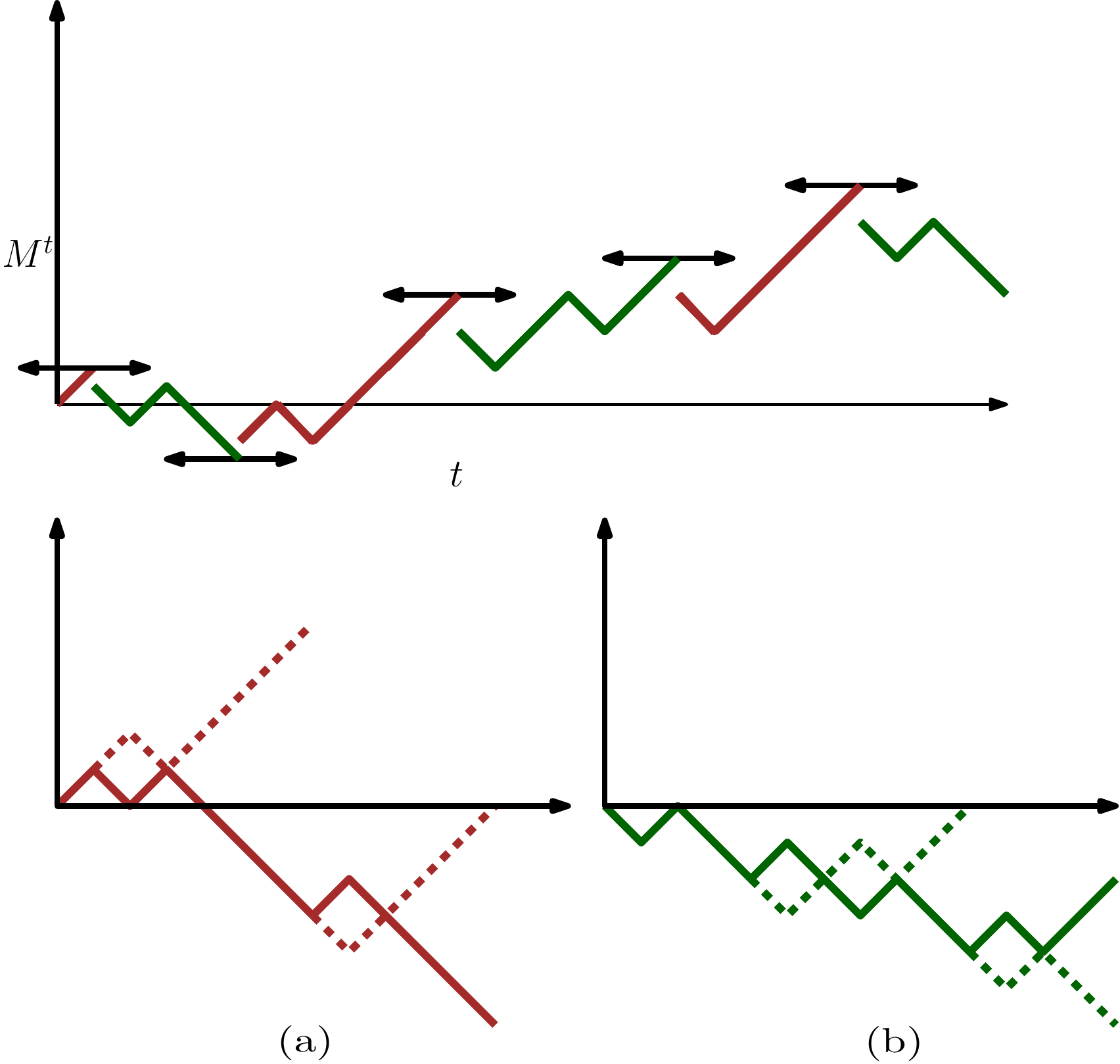}
   \caption{Illustrating the concatenation operations to obtain the random walk paths $\scE_k''$ and $\scF_k''.$ From Lemma \ref{l:Mrandwalk} it follows that the starting point of any brown segment has the same absolute value as the ending point of the previous green segment. In (a) we  join them `end-to-end' to form a random walk path, we inductively attach a new brown segment to the already formed path by concatenating the new segment or concatenating after reflecting to ensure that the end point of the path already formed agrees with the starting point of the newly added segment. The dotted curves indicate the actual segment when it was reflected and concatenated. For consecutive green segments as mentioned before there is a discrepancy of $\pm1$ between the absolute values of ending point and starting point. Thus in (b) we translate them by $\pm1$ to get rid of the discrepancy and then perform a reflection-concatenation operation as with the brown segments.}
    \label{fig:operation}
\end{figure}

Constructing a random walk path by joining the segments of $\scF_i$ is slightly more involved, since $|\scF_i(0)|\neq|\scF_{i-1}(N(\scF_{i-1}))|$. In this case we perform the operations \textbf{reflection-translation-concatenation} to join the segments `end-to-end'. Formally, we do the following. Let $\scF_0':=\scF_0-\{\scF_0(0)\}$ (translating the path to have the starting point at $0$). Also for $i\geq 1$, if 
\[\scF''_{i-1}=(\scF_0';\scF_1';\ldots;\scF_{i-1}')\]
 is the concatenated walk, then we define $\scF'_i$ as follows: If 
 \begin{equation}\label{e:defF'}
 \scF_i(0)=\scF'_{i-1}(N(\scF'_{i-1})))\pm1, \mbox{ then } \scF'_i=\scF_i-\{\scF_i(0)-\scF''_{i-1}(N(\scF''_{i-1}))\}.
 \end{equation}
 
 Otherwise, $\scF'_i=\scF^R_i-\{\scF^R_i(0)-\scF''_{i-1}(N(\scF''_{i-1}))\},$ where $\scF^R_i=-\scF_i$ is the reflection of $\scF_i$ across the $X$-axis, (see Figure \ref{fig:operation} (b)).

Then the concatenated walk $\scF''_k=(\scF'_0;\scF'_1;\scF'_2;\ldots;\scF'_k)$ is a random walk of step size $2$ starting from $0$ and since each segment $\scF_i$ is shifted by $1$, the  endpoint of $\scF_k$ and the endpoint of $\scF''_k$ in absolute values differ by at most $k+1$, i.e.,
\begin{equation}\label{e:diffF}
\bigl||\scF''_k(N(\scF''_k))|-((k+2)^2-1)\bigr|\leq k+1.
\end{equation}
{Also the two random walks $\scE''_k:=(\scE_0';\scE_1';\ldots;\scE'_k)$ and $\scF''_k:=(\scF_0';\scF_1';\ldots;\scF'_k)$ are independent}. To see this, observe that, given the endpoints of each of the segments $\scE_i,\scF_i$ (to be precise, only whether the endpoints are positive or negative is important, as their absolute values are fixed by definition), the segments are independent. 

Since, now both $\scE''_k,\scF''_k$ start from $0$, and the segments have been joined `end-to-end', $\scE''_k,\scF''_k$ are independent. We extend these random walk segments to two independent symmetric random walks $\scE'', \scF''$ starting from $0$ of step-size $2$, such that the path
$\scE''_k$ is the initial segment of length $N(\scE''_k)$  of the path $\scE''$ and the obvious corresponding statement holds  for $\scF''$ as well. 

For every $a>0$, let $\tau^{\scE''}_a$ (resp. $\tau^{\scF''}_a$ ) be the number of steps required for the random walk $\scE''$ (resp. $\scF''$) to hit $\pm a$. Then, we claim that
\begin{equation}\label{e:maxzonv}
\frac{\cZ_t-\sup_a\{\tau^{\scE''}_a+\tau^{\scF''}_a\leq t\}}{t^{1/2}}\overset{\P}{\rightarrow} 0,
\end{equation}
where $\overset{\P}{\rightarrow}$ denotes convergence in probability. 
It is easy to see how Theorem \ref{t:weakconvmicro} follows from this. Let $\clB, \clB'$ be two independent Brownian motions on $[0,1]$. After standard interpolation, consider the random walk paths $\left\{\frac{\scE''(rt)}{2t^{1/2}}\right\}_{r\in [0,1]},\left\{\frac{\scF''(rt)}{2t^{1/2}}\right\}_{r\in [0,1]}$ as elements of $C[0,1]$ (space of continuous functions on $[0,1]$ equipped with the topology of uniform convergence). Then, by Donsker's theorem, \begin{align}\label{brow23}
\left(\bigl\{\frac{\scE''(rt)}{2t^{1/2}}\bigr\}_{r\in [0,1]},  \bigl\{\frac{\scF''(rt)}{2t^{1/2}}\bigr\}_{r\in [0,1]}\right)\overset{d}{\Rightarrow} \Bigl(\{\clB(r)\}_{r\in [0,1]},\{\clB'(r)\}_{r\in [0,1]}\Bigr),
\end{align}
where $\overset{d}{\Rightarrow}$ denotes convergence in distribution.
In Lemma \ref{l:contmap}, it is shown that the function
\begin{equation}\label{contmap2}
h(f,g)=\sup_a\{\tau^f_a+\tau^g_a\leq 1\}
\end{equation}
is continuous for functions $f,g\in C[0,1]$. Hence, by continuous mapping,
\begin{equation}\label{e:Donsker}
\frac{\sup_{a}\{\tau_{a}^{\scE''}+\tau_{a}^{\scF''}\leq t\}}{2t^{1/2}}=\sup_{a}\left\{\tau_{a}^{\frac{\scE''(t\cdot)}{2t^{1/2}}}+\tau_{a}^{\frac{\scF''(t\cdot)}{2t^{1/2}}}\leq 1\right\}\overset{d}{\Rightarrow} \sup_a\{T_{a}+T'_a\leq 1\},
\end{equation}
where $T_a$ and $T_a'$ are the hitting times of $\pm a$ for the two independent Brownian motions $\clB(\cdot), \clB'(\cdot).$ (Equivalently $T_a$ and $T_a'$ are the hitting times of $a>0$ for two independent reflected standard Brownian motions).

Hence, the only thing left to prove is \eqref{e:maxzonv}.
Note that  for any microstep $t,$ by definition, we have 
 \[(k+1)^2-1\leq \cZ_t<(k+1)(k+2) \quad \mbox{or}\quad (k+1)(k+2)\leq \cZ_t<(k+2)^2-1,\]
for some $k\in \N$. We first assume the former case. Then, by \eqref{e:Z1},\eqref{e:Z2}, \eqref{e:diffF}, 
\begin{eqnarray*}
&&\sum_{0\le i\leq k-1}(N(\scE_i)+N(\scF_i))\leq t<\sum_{0\le i\leq k-1}(N(\scE_i)+N(\scF_i))+N(\scE_k)\\
&\mbox{ implies } &\tau^{\scE''}_{k(k+1)}+\tau^{\scF''}_{(k+1)^2-1-k}\leq t<\tau^{\scE''}_{(k+1)(k+2)}+\tau^{\scF''}_{(k+1)^2-1+k},\\
&\mbox{ which implies } &k^2+k\leq \sup_{a}\{\tau_{a}^{\scE''}+\tau_{a}^{\scF''}\leq t\}\leq k^2+3k+2,
\end{eqnarray*}
where the first implication uses the fact that  $\sum_{0\le i\leq k-1}(N(\scE_i))=N((\scE'_0;\scE'_1;\ldots;\scE'_k))$ and a similar fact for $\scF_i$'s., Hence, using \eqref{e:Z1},
\[|\cZ_t-\sup_{a}\{\tau_{a}^{\scE''}+\tau_{a}^{\scF''}\leq t\}|\leq 2k+2\leq 2\sqrt{\sup_{a}\{\tau_{a}^{\scE''}+\tau_{a}^{\scF''}\leq t\}}+2.\] 
which is a tight random variable at scale $t^{1/4}$ by \eqref{e:Donsker} and hence divided by $t^{1/2}$ converges to zero in probability.
A similar calculation using \eqref{e:Z2} would imply the same, when $(k+1)(k+2)\leq \cZ_t<(k+2)^2-1$. 
Hence \eqref{e:maxzonv} follows.

The following short lemma provides the necessary argument for the application of continuous mapping to the function in \eqref{contmap2} which in turn implied \eqref{e:Donsker} from Donsker's theorem.

\begin{lemma}\label{l:contmap} If $f_n,g_n,f,g \in C[0,1]$ such that $f_n\rightarrow f, g_n\rightarrow g$, where the convergence is in the sup-norm ($\|\cdot\|_{\infty}$), then 
\[\sup_a\{\tau^{f_n}_a+\tau^{g_n}_a\leq 1\}\rightarrow \sup_a\{\tau^f_a+\tau^g_a\leq 1\}.\]
\end{lemma}
\begin{proof}Let $u_n=\|f_n-f\|_{\infty}$ and  $v_n=\|g_n-g\|_{\infty}$. Then by hypothesis $u_n,v_n\rightarrow 0$. Let $z:=\sup_a\{\tau^f_a+\tau^g_a\leq 1\},$ and $z_n:=\sup_a\{\tau^{f_n}_a+\tau^{g_n}_a\leq 1\}$. Moreover also let $t_1=\tau^f_z$, and  $t_2=\tau^g_z$.  
Clearly,
\[z-u_n=|f(t_1)|-u_n\leq |f_n(t_1)|,  \text{ and, } z-v_n=|g(t_1)|-v_n\leq |g_n(t_1)|.\]
Thus, if $r_n=z-\max(u_n,v_n)$, then
$\tau^{f_n}_{r_n}+\tau^{g_n}_{r_n}\leq t_1+t_2\leq 1,$
and hence,
$z_n\geq r_n=z-\max(u_n,v_n).$
implying that $\liminf z_n\geq z$.

For the other inequality, assume that $\limsup z_n>z$. By going to a subsequence (we use the same notation for subsequence), this implies, there exists some $\varepsilon>0$ such that $z_n\geq z+\varepsilon$.  Then, let $c_n:=\tau^{f_n}_{z+\varepsilon},d_n:=\tau^{g_n}_{z+\varepsilon}$. Then $c_n+d_n\leq 1$. Also, for any large $n$ such that $u_n<\varepsilon/2, v_n<\varepsilon/2$, we have
\[z+\frac{\varepsilon}{2}\leq
z+\varepsilon-u_n\le|f_n(c_n)|-u_n\leq |f(c_n)|, \quad z+\frac{\varepsilon}{2}\leq z+\varepsilon-v_n\le |g_n(c_n)|-g_n\leq |g(c_n)|.\]
This implies, $z=\sup_a\{\tau^f_a+\tau^g_a\leq 1\}\geq z+\frac{\varepsilon}{2}$, thus arriving at a contradiction. 
\end{proof}

\subsection{Weak convergence of terminal run lengths}
In this subsection, we prove Theorem \ref{runs56}.

The proof will actually follow in a straightforward way from the following: 
\begin{theorem} \label{t:extremerunwk}
Recall the definitions $L_\cE^t(1)=S_{\cE}^t, L_\cE^t(2),\ldots$ from \eqref{e:deflayermicro} and $\scX_1,\scX_2,\ldots$ from  \eqref{e:defXi}. Then, for any fixed $k>1$,
\[\frac{(L^t_\cE(1),L^t_\cE(2),\ldots,L^t_\cE(k))}{t^{1/4}}\overset{d}{\Rightarrow} \sqrt{2}\left(\sqrt{\scX_1},\sqrt{\frac{\scX_2}{2}},\ldots,\sqrt{\frac{\scX_{k}}{2}}\right).\]
\end{theorem}
We postpone the proof of Theorem \ref{t:extremerunwk} and first see how Theorem \ref{runs56} follows directly from this.

\begin{proof}[Proof of Theorem \ref{runs56} ]  Recall the modified run lengths counted from the ends: $\cE_m^t(j)$ for $j\geq 1$ from  \eqref{e:defmodrunmicro}.
Note that by continuous mapping, as a straightforward consequence of Theorem \ref{t:extremerunwk} we have, \begin{equation}\label{e:eq2}
\frac{(\cE_m^t(1), \cE_m^t(2), \ldots, \cE_m^t(k))}{t^{1/4}}\overset{d}{\Rightarrow} \sqrt{2}\left(\sqrt{\scX_1}-\sqrt{\frac{\scX_2}{2}},\sqrt{\frac{\scX_2}{2}}-\sqrt{\frac{\scX_3}{2}},\ldots,\sqrt{\frac{\scX_k}{2}}-\sqrt{\frac{\scX_{k+1}}{2}}\right).
\end{equation}
Using the above, Theorem \ref{runs56} follows from the next lemma.
\end{proof}
\begin{lemma}\label{l:modrunsame}Recall $\cE_m^t(j), \cE^t(j)$ from \eqref{e:defEW} and \eqref{e:defmodrunmicro}. Then, for any fixed $k\geq 1$,
\[\P\left((\cE_m^t(1),\cE_m^t(2),\ldots,\cE_m^t(k))\neq (\cE^t(1),\cE^t(2),\ldots,\cE^t(k))\right)\rightarrow 0.\]

\end{lemma}
\begin{proof}From the discussion after \eqref{formal56de}, it follows that, all the random variables in the RHS of \eqref{e:eq2} are strictly positive almost surely. By Remark  \ref{r:modrun}, it follows that, 
\begin{eqnarray*}
&&\P\left((\cE_m^t(1),\cE_m^t(2),\ldots,\cE_m^t(k))\neq (\cE^t(1),\cE^t(2),\ldots,\cE^t(k))\right)\\
&\leq & \P(\cE_m^t(j)=0 \mbox{ for some } j\in \{1,2,\ldots,k\})\\
&\leq & \sum_{j=1}^k \P(\cE_m^t(j)=0)\\
&\leq & \sum_{j=1}^k \P\left(\frac{\cE_m^t(j)}{t^{1/4}}\leq 0\right) \rightarrow  0,
\end{eqnarray*}
where the last conclusion follows from \eqref{e:eq2}.
\end{proof}

We now proceed towards proving Theorem \ref{t:extremerunwk}. 
 As already mentioned in \eqref{alm464} we have the following:
\begin{lemma}\label{l:localmax} Let $\clB,\clB'$ and $\scX_1$ be as in the statement of Theorem \ref{informal}. Then almost surely, exactly one of the following occurs:
 \begin{align*}
\scA&:=\left\{\scX_1=\max \{|\clB'(s)|:s\in [0,1-T_{\scX_1}]\}\right\}\\
\scA'&:=\left\{\scX_1=\max \{|\clB(s)|:s\in [0,1-T'_{\scX_1}]\}\right\}.
\end{align*}
Hence, by symmetry, $\P(\scA)=\P(\scA')=\frac{1}{2}$. 
\end{lemma} 
\begin{remark}\label{endissue} It will be useful later to observe that as a straightforward consequence of  continuity properties of distribution of Brownian motion, almost surely $|\clB'(T'_{\scX_1})|\neq |\clB'(1-T_{\scX_1})|$ and similarly $|\clB(T_{\scX_1})|\neq |\clB(1-T'_{\scX_1})|.$
\end{remark}

We will also need the following lemma which states a refinement of the weak convergence result in \eqref{brow23}, conditioned on  $\scA$ or $\scA'.$ For this purpose, we will need the following `discrete' versions of $\scA$ and $\scA'$. For any fixed $t\in \N$, let $\mathfrak{A}_t$ denote the event that the vertical line $X=t$ intersects the graph of $M^t$ at $\scF_n$ for some $n$ (see Figure \ref{fig:Mtfig}), i.e., $(S_{\cW}^t,S_{\cE}^t)=(n,n+1)$ or $(n+1,n)$ for some $n$.
Thus $\mathfrak{A}^c_t$ is the event that the vertical line $X=t$ intersects the graph of $M^t$ at $\scE_n$ for some $n$, i.e., $(S_{\cW}^t,S_{\cE}^t)=(n,n)$ for some $n$.  Also let $\clB,\clB'$ be two independent Brownian motions as above and recall the random walks $\scE'',\scF''$ defined in \eqref{e:defE'} and
\eqref{e:defF'}.
\begin{lemma}\label{l:max} Let $\mu_t^{\mathfrak{A}}$ denote the conditional distribution of $\left(\left\{\frac{\scE''(rt)}{2t^{1/2}}\right\}_{r\in [0,1]},\left\{\frac{\scF''(rt)}{2t^{1/2}}\right\}_{r\in [0,1]}\right)$ given $\mathfrak{A}_t$, and $\mu^{(\scA)}$ denote the conditional distribution of $\left(\clB, \clB'\right)$ given $\scA$. Then,
\[\mu_t^\mathfrak{A}\overset{d}{\Rightarrow} \mu^{\scA}.\]

Also, by symmetry, if  $\mu_t^{\mathfrak{A}^c}$ denotes the conditional distribution of $\left(\left\{\frac{\scE''(rt)}{2t^{1/2}}\right\}_{r\in [0,1]},\left\{\frac{\scF''(rt)}{2t^{1/2}}\right\}_{r\in [0,1]}\right)$ given $\mathfrak{A}_t^c$, and $\mu^{\scA'}$ denote the conditional distribution of $\left(\clB, \clB \right)$ given $\scA'$. Then,
\[\mu_t^{\mathfrak{A}^c}\overset{d}{\Rightarrow} \mu^{\scA'}.\]
\end{lemma}

Note that the set $\scA$ has probability half and hence the above conditional distributions can be defined in a straightforward way. 
Before proving the above lemmas, we complete the proof of  Theorem \ref{t:extremerunwk}.

\begin{proof}[Proof of Theorem \ref{t:extremerunwk}] We only consider the case $k=2$. The proof of the general case is obtained by repeating similar arguments and is omitted. 
First assume that $\mathfrak{A}_t$ occurs, and without loss of generality that $(S_{\cW}^t,S_{\cE}^t)=(n_1+1,n_1)$ for some $n_1$. 
More generally let $L_\cE^t(i)=n_i$ (by Lemma \ref{l:layers}, this implies $n_i-1\leq L_\cW^t(i)\leq n_i+1$).

Also, recalling $\mathfrak{M}^t$ from \eqref{e:defmaxmicro}, we assume that $S_{\cE}^{\mathfrak{M}^t}=(\cB,\ldots,\cB)$, and  $S_{\cW}^{\mathfrak{M}^t}=(\cR,\ldots,\cR)$ (the other cases will be similar), and hence $M^{\mathfrak{M}^t}>0$. Thus, by the arguments in the proof of Lemma \ref{l:Mrandwalk},
\[(n_1+1)(n_1+2)\leq \max_{s\leq t}M^s\leq (n_1+2)^2-1.\]

Since $L_\cE^t(2)=n_2$, and $n_2-1\leq L_\cW^t(2)\leq n_2+1$, it follows by observing the value of $M^s$ when the second layer got formed, that 
\[(n_1+1)^2-2(n_2+1)^2\leq  \min_{\mathfrak{M}^t\leq s\leq t}M^s\leq (n_1+1)^2-2n_2^2.\]
Recall that $\tau^{\scE''}_a, \tau^{\scF''}_a$ are the (first) hitting times of $\pm a$ for $\scE'',\scF''$ and further let $\tilde{\tau}^{\scF''}_{(a+1)^2}$  be  the last time $\scF''$ hits $\pm (a+1)^2$ before time $t-\tau^{\scE''}_{(a+1)(a+2)}$. Then, the above statements along with the definition of $\scF''$ imply that
\begin{equation}\label{e:L2}
(n_1+1)^2-2(n_2+1)^2\leq \min_{\tilde{\tau}^{\scF''}_{(n_1+1)^2}\leq s\leq t-\tau^{\scE''}_{(n_1+1)(n_1+2)}} \scF''(s)\leq (n_1+1)^2-2n_2^2.
\end{equation}

Now, define,
\begin{equation}\label{proxy1}
\mathfrak{L}_t:=\left\{
\begin{array}{ll}
\max\limits_{s\leq t} M^s-\left(\min\limits_{\tilde{\tau}^{\scF''}_{(n_1+1)^2}\leq s\leq t-\tau^{\scE''}_{(n_1+1)(n_1+2)}} \scF''(s)\right) &\mbox{ when } M^{\mathfrak{M}^t}=\scF''(\tilde{\tau}^{\scF''}_{(n_1+1)^2})=(n_1+1)^2 ,\\
-\min\limits_{s\leq t} M^s+\left(\max\limits_{\tilde{\tau}^{\scF''}_{(n_1+1)^2}\leq s\leq t-\tau^{\scE''}_{(n_1+1)(n_1+2)}} \scF''(s)\right) &\mbox{ when } M^{\mathfrak{M}^t}=\scF''(\tilde{\tau}^{\scF''}_{(n_1+1)^2})=-(n_1+1)^2.
\end{array}
\right.
\end{equation}
 Hence, following similar arguments as in the proof of Theorem \ref{t:weakconvmicro}, it follows from \eqref{e:L2}, that for every $\varepsilon>0$, 
\begin{equation}\label{e:inprob}
\P\left(\left|\frac{L_E^t(2)-\sqrt{\frac{\mathfrak{L}_t}{2}}}{t^{1/4}}\right|>\varepsilon\Bigg|\mathfrak{A}_t\right){\rightarrow}0.
\end{equation}
Thus we will use $\sqrt{\frac{\mathfrak{L}_t}{2}}$ as a proxy for 
$L_E^t(2)$ since it satisfies nice continuity properties which will be convenient in proving weak convergence results. 

A similar calculation follows when the event $\mathfrak{A}_t^c$ occurs instead. Let $\mathfrak{L}_t'$ be the analogous definition of $\mathfrak{L}_t$, when the event $\mathfrak{A}_t^c$ occurs, instead of $\mathfrak{A}_t$ in the definition of $\mathfrak{L}_t$.

We now claim that the distribution  $\frac{\mathfrak{L}_t}{2t^{1/2}}$ conditional on the event $\mathfrak{A}_t$
converges  to the distribution of $\scX_2(\clB')$ conditional on the event $\scA,$ where $\scX_2(\clB')$ is as defined in \eqref{e:defXi} for standard Brownian motions $\clB,\clB'$.
 This follows from  Lemma \ref{l:max} and continuous mapping, once we establish the convergence of  
 \begin{align}\label{twoterm}
 \frac{\max_{s\leq t}|M^s|}{2t^{1/2}} \text{ and }  \frac{\left(\min\limits_{\tilde{\tau}^{\scF''}_{(n_1+1)^2}\leq s\leq t-\tau^{\scE''}_{(n_1+1)(n_1+2)}} \scF''(s)\right)}{2t^{1/2}},
 \end{align}
to their Brownian counterparts.
Lemma \ref{l:contmap} takes care of the first term.  The arguments for the second term are presented later (see  Lemmas \ref{l:cont1}, \ref{l:continuity} and the discussion preceding them).
Moreover, given the above, by symmetry, the distribution  $\frac{\mathfrak{L}_t'}{2t^{1/2}}$ conditional on the event $\mathfrak{A}_t'$
converges  to the distribution of $\scX_2(\clB)$ conditional on the event $\scA',$ where $\scX_2(\clB)$ is as defined in  \eqref{e:defXi} by replacing $\clB'$ by $\clB$.

It follows easily from \eqref{e:inprob} and the above that, for any $b>0$,
\begin{eqnarray*}
&&\lim_t\P\left(\frac{L^t_\cE(2)}{t^{1/4}}\leq b\right)
=\lim_t\P\left(\left[\frac{L^t_\cE(2)}{t^{1/4}}\leq b\right]\Big| \mathfrak{A}_t\right)\P(\mathfrak{A}_t)+\lim_t\P\left(\left[\frac{L^t_\cE(2)}{t^{1/4}}\leq b\right]\Big| \mathfrak{A}_t^c\right)\P(\mathfrak{A}_t^c)\\
&=&\lim_t\P\left(\left[\sqrt{\frac{\mathfrak{L}_t}{2t^{1/2}}}\leq b\right]\Big| \mathfrak{A}_t\right)\P(\mathfrak{A}_t)+\lim_t\P\left(\left[\sqrt{\frac{\mathfrak{L}'_t}{2t^{1/2}}}\leq b\right]\Big| \mathfrak{A}_t^c\right)\P(\mathfrak{A}_t^c)\\
&=& \P\left(\left[\sqrt{\scX_2(\clB')}\leq b\right]\Big| \scA\right)\P(\scA)+\P\left(\left[\sqrt{\scX_2(\clB)}\leq b\right]\Big| \scA'\right)\P(\scA')\\
&=& P\left(\left[\sqrt{\scX_2(\clB')}\leq b\right]\Big| \scA\right)(\P(\scA)+\P(\scA'))
=\P\left(\left[\sqrt{\scX_2(\clB')}\leq b\right]\big|\scA\right)=\P\left(\sqrt{\scX_2}\leq b\right),
\end{eqnarray*}
where the last line follows by using symmetry. 

\end{proof}
 We now complete the proofs of Lemmas \ref{l:localmax} and \ref{l:max}.

\begin{proof}[Proof of Lemma \ref{l:localmax}]
Since $\clB,\clB'$ are two independent Brownian motions, their respective sets of local extrema (any point which is a local maxima or a local minima) are  disjoint with probability one.  To see this note that the set of local extrema for Brownian motion is a countable set, and the fact that any fixed point is a local maxima with probability zero. The proof now follows by conditioning on $\clB,$ and showing that the probability that some $t$ which is a local extrema of $\clB$ is also a local extrema of $\clB'$ is zero, followed by union bounding over all local extrema of $\clB.$

Moreover, for at least one of $|\clB|,|\clB'|$, $T_{\scX_1}$ or $T'_{\scX_1}$ must be a local maxima. This follows because otherwise, the fact that $T_{\scX_1}+T'_{\scX_1}<1$ almost surely, would  contradict the maximality of $\scX_1$.
Also observe that  if $\scX_1$ is the local maxima of $|\clB'|$, so that $T_{\scX_1}$ is not a point of local maxima of $|\clB|$, then $\scX_1$ is the maximum value of $|\clB'|$ till time $1-T_{\scX_1}$. To see this, observe that if there exists some $z>\scX_1$ such $T'_z< 1-T_{\scX_1}$, then the fact that for any $t>T_{\scX_1}$, there is a point $t>t_0>T_{\scX_1}$ such that $|\clB(t_0)|>\scX_1$ (since it is not a local extrema), would imply again that 
$\sup_{a}\{T_a+T'_a\leq 1\}>\scX_1$.
\end{proof}

\begin{proof}[Proof of Lemma \ref{l:max}] We only prove the first claim and the second one follows by symmetry. Using the Portmanteau Theorem, it is enough to show that, for any closed set $\mathfrak{G}\subseteq C[0,1]\times C[0,1]$,
\begin{equation}\label{e:closedset}
\limsup_t \P\left(\left(\left\{\frac{\scE''(rt)}{2t^{1/2}}\right\}_{r\in [0,1]},\left\{\frac{\scF''(rt)}{2t^{1/2}}\right\}_{r\in [0,1]}\right)\in \mathfrak{G}\Big|\mathfrak{A}_t\right)\leq \P\left(\left(\{\clB(r)\},\{\clB'(r)\}\right)\in \mathfrak{G}\Big|\scA\right).
\end{equation}
Now on the event $\mathfrak{A}_t$ since $(S_{\cW}^t,S_{\cE}^t)=(m,m+1)$ or $(m+1,m)$ for some $m\in \N$,  by Lemma \ref{l:Mrandwalk},
\[(m+1)(m+2)\leq \max_{s\leq t}|M^s|\leq (m+2)^2-1.\]
Also, if $\mathfrak{A}_t$ happens, then $\scF''(t-\tau^{\scE''}_{(m+1)(m+2)})$ is a part of the segment $\scF'_m$ (the segments $\scE'_k,\,\scF'_k$ are defined in \eqref{e:defF'}). Thus,
\[(m+1)^2\leq \max_{0\leq s\leq t-\tau^{\scE''}_{(m+1)(m+2)}}|\scF''(s)|\leq (m+2)^2-1.\]
Thus, on the event $\mathfrak{A}_t$,
$\displaystyle{\left|\max_{0\leq s\leq t-\tau^{\scE''}_{(m+1)(m+2)}}|\scF''(s)|-\max_{s\leq t}|M^s|\right|\leq 2(m+1)\leq 2\sqrt{\max_{s\leq t}|M^s|}+2,}$ and $\sqrt{\max_{s\leq t}|M^s|}$ is a tight random variable at scale $t^{1/4}$ since by Theorem \ref{t:weakconvmicro} and \eqref{e:conmaxM}, one has $\frac{\max_{s\leq t}|M^s|}{2t^{1/2}}\Rightarrow \scX_1$. Hence, for any fixed $\eta>0$, for all large enough $t$,
\begin{equation}\label{e:o(1)}
\mathfrak{A}_t\subseteq \left\{\left|\frac{\max_{s\leq t}|M^s|}{2t^{1/2}}-\frac{\max\limits_{0\leq s\leq t-\tau^{\scE''}_{(m+1)(m+2)}}|\scF''(s)|}{2t^{1/2}}\right|\leq \eta\right\}=:\mathfrak{C}_t.
\end{equation}

Also by Donsker's theorem,
\begin{align}\label{e:wkconv}
\left(\left\{\frac{\scE''(rt)}{2t^{1/2}}\right\}_{r\in [0,1]},\left\{\frac{\scF''(rt)}{2t^{1/2}}\right\}_{r\in [0,1]},\frac{\max_{s\leq t}|M^s|}{2t^{1/2}},\frac{\max\limits_{0\leq s\leq t-\tau^{\scE''}_{(m+1)(m+2)}}|\scF''(s)|}{2t^{1/2}} \right)\\
\nonumber
\overset{d}{\Rightarrow}\left(\clB,\clB',\scX_1,\max_{0\leq s\leq 1-T_{\scX_1}}|\clB'(s)| \right),
\end{align}
where $\clB,\clB'$ are independent Brownian motions. The convergence of the third term in \eqref{e:wkconv} follows from Theorem \ref{t:weakconvmicro} and \eqref{e:conmaxM}. The required continuity arguments for the convergence of the fourth term in \eqref{e:wkconv} is provided in Lemma \ref{l:cont1} (see the discussion preceding it). Hence, by a simple continuous mapping,
\begin{align}\label{e:weakconv2}
\mathfrak{X}_t:=\left(\left\{\frac{\scE''(rt)}{2t^{1/2}}\right\}_{r\in [0,1]},\left\{\frac{\scF''(rt)}{2t^{1/2}}\right\}_{r\in [0,1]},\frac{\max_{s\leq t}|M^s|}{2t^{1/2}}-\frac{\max\limits_{0\leq s\leq t-\tau^{\scE''}_{(m+1)(m+2)}}|\scF''(s)|}{2t^{1/2}} \right)\\
\nonumber
\overset{d}{\Rightarrow}\left(\clB,\clB',\scX_1-\max_{0\leq s\leq 1-T_{\scX_1}}|\clB'(s)| \right)=:\mathfrak{X}.
\end{align}
 Hence, if $\mathfrak{G}\subseteq C[0,1]\times C[0,1]$ is any closed set, and
\begin{align*}
\mathfrak{J}_t &:=\left\{\left(\left\{\frac{\scE''(rt)}{2t^{1/2}}\right\}_{r\in [0,1]},\left\{\frac{\scF''(rt)}{2t^{1/2}}\right\}_{r\in [0,1]}\right)\in \mathfrak{G}\right\},\\
\mathfrak{J}&:=\left\{\left(\{\clB(r)\}_{r\in [0,1]},\{\clB'(r)\}_{r\in [0,1]}\right)\in \mathfrak{G}\right\},
\end{align*} then, 
\begin{align*}
\limsup_t \P(\mathfrak{J}_t\cap \mathfrak{A}_t)&\leq  \limsup_t \P(\mathfrak{J}_t\cap \mathfrak{C}_t)=\limsup_t \P\left(\mathfrak{X}_t\in (\mathfrak{G}\times[-\eta,\eta])\right),\\
&\leq \P(\mathfrak{X}\in (\mathfrak{G}\times[-\eta,\eta]))=\P\left(\mathfrak{J}\cap \left\{\left|\scX_1-\max_{0\leq s\leq 1-T_{\scX_1}}|\clB'(s)|\right|\leq \eta\right\}\right),
\end{align*}
where the inequality in the second line follows because of \eqref{e:weakconv2} and the fact that $\mathfrak{G}\times[-\eta,\eta]$ is a closed set. By letting $\eta\rightarrow 0$, one has $\limsup_t\P(\mathfrak{J}_t\cap \mathfrak{A}_t)\leq \P(\mathfrak{J}\cap \scA)$.

Moreover, by taking $\mathfrak{G}=C[0,1]^2$, one has $\limsup_t\P(\mathfrak{A}_t)\leq \P(\scA)$. Further, replacing $\scE''$ by $\scF''$, and using Lemma \ref{l:localmax}, we have,
$\limsup_t\P(\mathfrak{A}^c_t)\leq \P(\scA')$. Since $\scA'$ is the complement of the event $\scA$ by Lemma \ref{l:localmax}, this gives that $\P(\mathfrak{A}_t)\rightarrow \P(\scA)$. 

Thus from the above we get that $\limsup_t\P(\mathfrak{J}_t\cap \mathfrak{A}_t)\leq \P(\mathfrak{J}\cap \scA)$ and 
$\P(\mathfrak{A}_t)\rightarrow \P(\scA)$. 
Hence \eqref{e:closedset} follows.
\end{proof}

The only things left to prove are the necessary continuity arguments used in the proof of Lemma \ref{l:max} and Theorem \ref{t:extremerunwk}: i.e. justifying the continuity of the second term in \eqref{twoterm} and continuity of the fourth term in \eqref{e:wkconv}.
Recall that for any function $f\in C[0,1]$ and $a>0,$  $\tau_a^f$ denotes the first time $|f|$ hits $a$. 
Consider a sequence of functions $f_n,g_n,f,g\in C[0,1]$ such that $f_n\rightarrow f,$ and  $g_n\rightarrow g$ where the convergence is in sup-norm ($\|\cdot\|_\infty$), and $z$ is as in \eqref{e:defmaxhit} and two sequences $z_n$, and $z_n'$ converging to $z$. If  $f,g$ satisfy certain conditions, stated in the hypothesis of Lemma \ref{l:cont1},which Brownian motion paths almost surely do, the latter implies that that $\tau^{f_n}_{z_n}\rightarrow \tau^{f}_{z}$. This in turn implies that,
\[\max_{0\leq s\leq 1-\tau^{f_n}_{z_n}}|g_n(s)|\rightarrow \max_{0\leq s\leq 1-\tau^{f}_{z}}|g(s)|.\]
This, takes care of the convergence of the fourth term in \eqref{e:wkconv}.

Further, if $\tilde{t}_n'$, as in Lemma \ref{l:continuity}, denotes the last time $z_n'$ is hit by $|g_n|$, where $|g|$ attains a local maxima at $\tau_z^g$, then Lemma \ref{l:continuity} shows that $\tilde{t}_n'\rightarrow \tau_z^g$. This, together with Lemma \ref{l:cont1} ($\tau^{f_n}_{z_n}\rightarrow \tau^{f}_{z}$), in turn imply,
\[\min_{\tilde{t}_n'\leq s\leq 1-\tau^{f_n}_{z_n}}g_n(s)\rightarrow\min_{\tau_z^g\leq s\leq 1-\tau^{f}_{z}}g(s),\]
which is the required continuity of the second term in \eqref{twoterm}. 
 We now formally state and prove the lemmas used in the above discussion.

\begin{lemma}\label{l:cont1}Let $f_n,g_n,f,g\in C[0,1]$ be such that $f_n\rightarrow f, g_n\rightarrow g$ where the convergence is in sup-norm ($\|\cdot\|_\infty$).  Let
\begin{equation}\label{e:defmaxhit}
z=\sup_a\{\tau_a^f+\tau_a^g\leq 1\}.
\end{equation} 
Let $z_n\rightarrow z, z'_n\rightarrow z$. Let $t_n=\tau^{f_n}_{z_n}, t=\tau^f_z, t_n'=\tau^{g_n}_{z_n'}, t'=\tau^g_z$. Also assume $t_n+t_n'\leq 1$, and one of the following two cases occur (by Lemma \ref{l:localmax}, independent Brownian motions satisfy this property a.s.).
\begin{itemize}
\item[1.] $|f|$ attains a local maxima at $t$, and $g$ is such that for all $s>t'$, there exists $s>t_0>t'$ such that $|g(t_0)|>z$.

\item[2.] $|g|$ attains a local maxima at $t'$, and $f$ is such that for all $s>t$, there exists $s>t_0>t$ such that $|f(t_0)|>z$.
\end{itemize}
Moreover assume that in any open interval, $|f|,|g|$ attain their maximums at at most one point (A straightforward adaption of the argument in  \cite{MP} yields that  Brownian motion satisfies this a.s.) and that $|f(1-\tau^g_a)|\neq a$ and $|g(1-\tau^f_a)|\neq a$ (By Remark \ref{endissue} this  holds almost surely for independent Brownian motions by standard continuity arguments.)
Then $t_n\rightarrow t, t_n'\rightarrow t'$.
\end{lemma}

\begin{proof} By symmetry, it is enough to show that $t_n'\rightarrow t'$. Since $v_n:=\|g_n-g\|_\infty\rightarrow 0$, hence,
\[z_n'-v_n=|g_n(t_n')|-v_n\leq |g(t_n')|.\]
If $\liminf t_n'<t'$, then going to a subsequence, there exists some $\varepsilon>0$ such that $t_{n_k}'<t'-\varepsilon$ for all $n_k$. Since $|g(s)|<z$ for all $z<t'$ by definition, hence, because of continuity of $g$,
\[\sup_{0\leq s\leq t'-\varepsilon}|g(s)|=:z_0=z-\delta,\]
for some $\delta>0$. Thus,
\[z_{n_k}'-v_{n_k}\leq |g(t_{n_k}')|\leq \sup_{0\leq s\leq t'-\varepsilon}|g(s)|=z-\delta.\] 
This contradicts that $z_n'$ converges to $z.$

Now we prove the other inequality, namely that $\limsup t_n'\le t'.$ Consider Case $1.$. If $\limsup t_n'>t'$, then going to a subsequence, there exists $\varepsilon>0$ such that $t'_{n_k}\geq t'+\varepsilon$ for all $n_k$. Because of the assumption on $g$, there exists some $t'< t_0\leq t'+\frac{\varepsilon}{2}$ such that $|g(t_0)|=z+\delta$ for some $\delta>0$. Get $n_k$ large enough such that
\[|g_{n_k}(t_0)|\geq |g(t_0)|-\frac{\delta}{2}=z+\frac{\delta}{2}.\]
Also choose $n_k$ large enough such that $z'_{n_k}\leq z+\frac{\delta}{4}$. But, $t'_{n_k}\geq t'+\varepsilon\geq t_0+\frac{\e}{2}$, and the continuity of $g_n$ contradicts the definition of $t'_{n_k}=\tau^{g_{n_k}}_{z_{n_k}'}$.

We now consider Case $2$. The conditions on $f,g$ and the definition of $z$ imply that (see Lemma \ref{l:localmax})
$\sup_{0\leq s\leq 1-t}|g(s)|=z.$
Assume that $\limsup t_n'>t'$. Then, since $0\leq t_n'\leq 1$ (since by assumption, $t_n+t_n'\leq 1$), there exists $\varepsilon>0$ and a subsequence such that $t_{n_k}' \rightarrow t'+\varepsilon$. Since $t_n'\leq 1-t_n$ and moreover the arguments in Case 1 imply that $t_n\rightarrow t$. Thus, $t'+\varepsilon\leq 1-t$. Further,
\[z_{n_k}'-v_{n_k}=|g_{n_k}(t_{n_k}')|-v_{n_k}\leq |g(t_{n_k}')|\leq |g_{n_k}(t_{n_k}')|+v_{n_k}=z_{n_k}'+v_{n_k},\]
and the continuity of $g$ imply that $|g(t'+\varepsilon)|=z$. Thus there exist two points $t',t'+\varepsilon\in [0,1-t]$ such that
\[|g(t')|=|g(t'+\varepsilon)|=z=\sup_{0\leq s\leq 1-t}|g(s)|.\]
 Now first of all by hypothesis $t'+\e$ is strictly less than $1-t.$ However this implies two maxima in the open interval $(0,1-t)$ which then 
contradicts the other assumption on $g$.

This contradicts the assumption on $g$. 
\end{proof}

\begin{lemma}\label{l:continuity}Assume the conditions in Lemma \ref{l:cont1} and assume that Case $2.$ holds. Let 
\[\tilde t_n'=\sup\{s\leq 1-t_n:|g_n(s)|=z_n'\}\]
be the last time before $1-t_n$ such that $g_n$ attains the value $\pm z_n'$. Then $\tilde t'_n \rightarrow t'$.
\end{lemma}

The proof is similar to the last part of the proof of Lemma \ref{l:cont1}. We briefly outline it here.
\begin{proof}Since $\tilde t_n'\in [0,1]$, hence enough to show that every converging subsequence of $\{\tilde t_n'\}$ converges to $t'$. Let $\tilde t_{n_k}'\rightarrow t_0$. Since $t_{n_k}'\leq 1-t_{n_k}$, and $t_{n_k}\rightarrow t$ by Lemma \ref{l:cont1}, hence $t_0\leq 1-t$. Further,
\[z_{n_k}'-v_{n_k}=|g_{n_k}(\tilde t_{n_k}')|-v_{n_k}\leq |g(\tilde t_{n_k}')|\leq |g_{n_k}(\tilde t_{n_k}')|+v_{n_k}=z_{n_k}'+v_{n_k},\]
and the continuity of $g$ imply that $|g(t_0)|=z$. If $t_0\neq t'$, then there exist two points $t_0,t'\in [0,1-t]$ such that
\[|g(t_0)|=|g(t')|=z=\sup_{0\leq s\leq 1-t}|g(s)|.\] The arguments in the proof of the previous lemma now go through verbatim, 
contradicting the assumptions on $g$.
\end{proof}

\section{Comparison of particle and microstep time scales}\label{s:particles}  

In this section we prove Theorem \ref{t:mainparticle}. However we first complete the proofs of Theorems \ref{informal} and \ref{mainruns} assuming the former.  Recall the definition of $V(n)$ from the statement of Theorem \ref{t:mainparticle}.
\begin{proof}[Proof of Theorem \ref{informal}]
By definition, $S_{\cE}(n)=S_{\cE}^{V(n)}$. Hence, because of Theorem \ref{t:mainparticle}, it is enough to show that
\[\frac{S_{\cE}^{V(n)}}{(V(n))^{1/4}}\overset{d}{\Rightarrow} G,\]
where $G:=\sqrt{2}\sqrt{\sup\{a\geq 0:T_a+T'_a\leq 1\}}$. To this end, fix any $\varepsilon>0$. By Theorem \ref{t:mainparticle}, 
\begin{equation}\label{sandwich}
\alpha-\varepsilon<\frac{V(n)}{n}<\alpha+\varepsilon \quad \mbox{ with probability } 1-\delta_n,
\end{equation}
where $\delta_n\rightarrow 0$ as $n\rightarrow \infty$. Since $S_{\cE}^t$ is non-decreasing in $t$, this implies,
\[\P\left[\frac{S_{\cE}^{\floor{n(\alpha-\varepsilon)}}}{\left(n(\alpha-\varepsilon)\right)^{1/4}}\left(\frac{\alpha-\varepsilon}{\alpha+\varepsilon}\right)^{1/4}\leq \frac{S_{\cE}^{V(n)}}{(V(n))^{1/4}}\leq \frac{S_{\cE}^{\ceil{n(\alpha+\varepsilon)}}}{\left(n(\alpha+\varepsilon)\right)^{1/4}}\left(\frac{\alpha+\varepsilon}{\alpha-\varepsilon}\right)^{1/4}\right]\geq 1-\delta_n.\]
Thus, for any $c\in \R$, using Theorem \ref{t:weakconvmicro} one has,
\begin{eqnarray*}
\P\left(\frac{S_{\cE}^{V(n)}}{(V(n))^{1/4}}\leq c\right)&\leq & \delta_n+\P\left(\frac{S_{\cE}^{\floor{n(\alpha-\varepsilon)}}}{\left(n(\alpha-\varepsilon)\right)^{1/4}}\left(\frac{\alpha-\varepsilon}{\alpha+\varepsilon}\right)^{1/4}\leq c\right)\\
&=& \delta_n +\P\left(\frac{S_{\cE}^{\floor{n(\alpha-\varepsilon)}}}{\left(n(\alpha-\varepsilon)\right)^{1/4}}\leq c\left(\frac{\alpha+\varepsilon}{\alpha-\varepsilon}\right)^{1/4}\right)\rightarrow \P\left(G\leq  c\left(\frac{\alpha+\varepsilon}{\alpha-\varepsilon}\right)^{1/4}\right).
\end{eqnarray*}
Similarly,
\begin{eqnarray*}
\P\left(\frac{S_{\cE}^{V(n)}}{(V(n))^{1/4}}\leq c\right)&\geq & \P\left(\frac{S_{\cE}^{\ceil{n(\alpha+\varepsilon)}}}{\left(n(\alpha+\varepsilon)\right)^{1/4}}\left(\frac{\alpha+\varepsilon}{\alpha-\varepsilon}\right)^{1/4}\leq c\right)-\delta_n \rightarrow \P\left(G\leq  c\left(\frac{\alpha-\varepsilon}{\alpha+\varepsilon}\right)^{1/4}\right).
\end{eqnarray*}
Letting $\varepsilon \rightarrow 0$ and using the continuity of the distribution function of $G$ one has the result.
\end{proof}
\begin{proof}[Proof of Theorem \ref{mainruns}] The arguments are a combination of the ones appearing in the previous proof along with those appearing in the proof of Theorem \ref{t:extremerunwk}. Hence we just sketch the main steps and as in the proof of Theorem \ref{t:extremerunwk} we only consider the case ${\rm E}(n,1)$, since the arguments for ${\rm E}(n,2),{\rm E}(n,3),\ldots,$ are similar. 

Now recall \eqref{proxy1},
\begin{equation*}
\mathfrak{L}_t:=\left\{
\begin{array}{ll}
\max\limits_{s\leq t} M^s-\left(\min\limits_{\tilde{\tau}^{\scF''}_{(n_1+1)^2}\leq s\leq t-\tau^{\scE''}_{(n_1+1)(n_1+2)}} \scF''(s)\right) &\mbox{ when } M^{\mathfrak{M}^t}=\scF''(\tilde{\tau}^{\scF''}_{(n_1+1)^2})=(n_1+1)^2 ,\\
-\min\limits_{s\leq t} M^s+\left(\max\limits_{\tilde{\tau}^{\scF''}_{(n_1+1)^2}\leq s\leq t-\tau^{\scE''}_{(n_1+1)(n_1+2)}} \scF''(s)\right) &\mbox{ when } M^{\mathfrak{M}^t}=\scF''(\tilde{\tau}^{\scF''}_{(n_1+1)^2})=-(n_1+1)^2.
\end{array}
\right.
\end{equation*}
as well as the terms in \eqref{twoterm},
\begin{align*}
 \frac{\max_{s\leq t}|M^s|}{2t^{1/2}} \text{ and }  \frac{\left(\min\limits_{\tilde{\tau}^{\scF''}_{(n_1+1)^2}\leq s\leq t-\tau^{\scE''}_{(n_1+1)(n_1+2)}} \scF''(s)\right)}{2t^{1/2}},
 \end{align*}
Recall that $\scE''$ and $\scF''$ properly scaled converge to Brownian motions $\clB$ and $\clB'$ respectively. Thus using the properties of $\clB, \clB'$ stated before during the discussion around \eqref{formal56de}, it follows that for any $\delta$, for all small enough $\e$ for all large enough $t$ with probability at least $1-\delta$ the following holds:
When the first case in the above expression of $\mathfrak{L}_t$ holds, then
simultaneously for all $t_*\in [t(1-\e),t(1+\e)],$
\begin{align*}
\max_{s\leq t_*}|M^s|&=\max_{s\leq t}|M^s|\\
\min\limits_{\tilde{\tau}^{\scF''}_{(n_1+1)^2}\leq s\leq t_*-\tau^{\scE''}_{(n_1+1)(n_1+2)}} \scF''(s)&=\min\limits_{\tilde{\tau}^{\scF''}_{(n_1+1)^2}\leq s\leq t-\tau^{\scE''}_{(n_1+1)(n_1+2)}} \scF''(s)
\end{align*}
The obvious corresponding statement holds for the second case in the above definition of $\mathfrak{L}_t.$

This shows that for any $\delta$, for all small enough $\e$ for all large enough $t,$ 
$\mathfrak{L}_{t_*}=\mathfrak{L}_{t}$ for all $t_*\in [t(1-\e),t(1+\e)]$ with probability at least $1-\delta.$ 
This allows us to use the sandwiching statement in \eqref{sandwich} and carry out the proof as in the proof of Theorem \ref{informal}; the only difference being that in the latter we relied on \eqref{sandwich} and Theorem \ref{t:weakconvmicro}, whereas here we use Theorem \ref{t:extremerunwk} instead which was used in the proof of Theorem \ref{runs56}.
\end{proof}

We now dive in to the proof of Theorem \ref{t:mainparticle} which spans over the next three subsections.

\subsection{Proportion of time when there are long monochromatic runs }\label{sec980}
Recall the discussion from Section \ref{s:outline} about the strategy to compare the number of microsteps and the total number of particles emitted.  Recall from \eqref{hit78}, that $\tau^{i}$ is the number of microsteps taken by the $i^{th}$ particle.
To this end we have the following definition. 
\begin{definition}Let $n$ be odd and without loss of generality assume that $X(n-1)(-1)={\rm{R}}.$  For some positive integer $L$ we call $n$ or the $n^{th}$ particle as $L-good$ if $X(n-1)(y)={\rm{B}}$  for all $1\le y\le L-1 $ and $X(n-1)(L)\neq {\rm{B}}$. Similarly we define an even $n$ to be $L-good$  by switching ${\rm{R}}$ and ${\rm{B}}.$ Thus  the $n^{th}$ particle is said to be $L-$good  if the  run adjacent to the origin of the same color as that of the emitted particle  has length $L-1$ when the particle was emitted. We also call a microstep $t$ as $L-good$ if  $$\sum_{j=1}^{n-1}\tau^j< t \le \sum_{j=1}^{n}\tau^j,$$ 
and $n$ is $L-good,$ i.e., the microstep was taken by a particle which was $L-good.$  
\end{definition}
We now prove bounds on the fraction of $L-good$ microsteps for large $L.$ However it will be convenient to break the analysis into two similar parts where in one part we consider the case when the run adjacent to the origin of size $L-1$ is on the positive axis and in the other we consider the negative axis.
To this end we call the  $n^{th}$ particle as $({\rm{E}},L)-good$ if the run adjacent to the origin of the same color as that of the particle is of length $L-1$ and it is on the positive axis. Similarly we call the particle as $({\rm{W}},L)-good$ if the relevant run is on the negative axis. Similarly a microstep $t$ inherits the same terminology from the corresponding particle associated to it.

Fix $L\in \N$. For any microstep $t,$ let 
 \begin{equation}\label{gooddef978}
G(L,t)=\sum_{j=1}^{t}\mathbf{1}(j^{th} \mbox{ microstep is } L-good ),
\end{equation}
 be the number of $L-good$ microsteps up to $t$.  
For $L \in \N,$ let $T^{(0)}_{{\rm{E}}, L}=0$ and $T^{(1)}_{{\rm{E}}, L}, T^{(2)}_{{\rm{E}}, L}, \ldots$ be the sequence of microsteps with the following properties:
\begin{itemize} 
\item $T^{(j)}_{{\rm{E}}, L}$ is $({\rm{E}}, L)-good$ for any $j.$
\item For any $j,$ there exists $t$ such that $T^{(j)}_{{\rm{E}}, L}< t <T^{(j+1)}_{{\rm{E}}, L}$  such that $t$ is $({\rm{E}},L+1)-good.$
\item For any $j$,  $T_{{\rm{E}}, L}^{(j+1)}$ is the earliest microstep with the above properties.
\end{itemize}
In words, consider the first time a microstep becomes $({\rm{E}},L)-good$. Typically the next few microsteps are either $(W,j)-good$ for some $j$ or  $({\rm{E}},\ell)-good$ for some $\ell\le L$ and eventually a microstep becomes $({\rm{E}},\ell)-good$ for some $\ell\ge L+1$. Then after some time a microstep would become $({\rm{E}},L)-good$ again for the first time. These are the times $T^{(j)}_{{\rm{E}},L},$ when the process returns to the state of being $({\rm{E}},L)-good$ after becoming $({\rm{E}},\ell)-good$ for some $\ell\ge L+1$. 
Let $\cF_{{\rm{E}},L}^{(i)}=T^{(i)}_{{\rm{E}},L}-T^{(i-1)}_{{\rm{E}},L}$ (the waiting times between consecutive $T^{(i)}_{{\rm{E}},L}$'s ) and let $$\cF^{(i)}_{{\rm{E}},L,L}=\sum_{T^{(j)}_{{\rm{E}},L}< t <T^{(j+1)}_{{\rm{E}},L}}\mathbf{1}(t \mbox{ is } ({\rm{E}},L)-good),$$ (the number of steps in $\cF_{{\rm{E}},L}^{(i)}$ that are themselves $({\rm{E}},L)-good$). 
Also, let $\ccW^{(i)}_{{\rm{E}},L}:=T^{(i+1)}_{{\rm{E}},L}-t_*$ where $t_*=\inf\{t: T^{(j)}_{{\rm{E}},L}< t <T^{(j+1)}_{{\rm{E}},L},\,\,\,  \mbox{and } t \mbox{ is } ({\rm{E}},\ell)-good \mbox{ for some } \ell\ge L+1\}$ i.e., the number of microsteps needed to get an $({\rm{E}},L)-good$ microstep, after a certain microstep has been $({\rm{E}},\ell)-good$ for some $\ell\ge L+1$.
The following two lemmas concerning the $\cF^{(i)}_{{\rm{E}},L,L}$'s and $\ccW^{(i)}_{{\rm{E}},L}$'s would be crucial. Now it would be obvious from the proofs and obvious symmetry of the situation that these results hold for  $\cF^{(i)}_{{\rm{W}},L,L}$'s and $\ccW^{(i)}_{{\rm{W}},L}$ as well where the latter are the obvious analogues obtained by replacing $\rm{E}$ by $\rm{W}.$ Hence in the following for brevity we will suppress the $\rm{E}$ subscript.  
It will also be convenient to keep in mind the consequence of the proof of Lemma \ref{l:layers}, that whenever a particle is $L-good$ then both the intervals $[-L+1,-1]$ and $[1,L-1]$ are monochromatic with opposite colors. 

\begin{lemma} \label{rw12}
There exists $c>0,$  such that for all large enough $r,$ for any $L\in \N,$ conditionally on the past, for any $i,$ the random variable, $\cF^{(i)}_{L,L}$ satisfies $\P(\cF^{(i)}_{L,L}\ge L^2 r)\le e^{-cr}.$
\end{lemma}
\begin{proof} To prove the lemma,  we will show that, conditionally on the past, $\cF^{(i)}_{L,L}$, is dominated by the number of steps taken by a symmetric random walk on $\Z$ started at $0$, reflected at $-1$  to reach $L.$ 
The lemma now follows using standard random walk hitting time estimates. 
Now notice that  a particle is $({\rm{E},}L)-good$ iff the particle is stopped on hitting $\{-1,L\}.$  Thus $\cF^{(i)}_{L,L}$ counts all such microsteps
 before $L$ is hit if it is at all hit. (Note that it might happen that the number of $L-good$ particles emitted before the interval $[1,L]$ becomes monochromatic could be as small as one on the event that a layer of opposite color grows to make $[1,L-1]$ as the same color as $L$.) If  a $({\rm{E},}L)-good$ particle hits $-1$ instead  the next 
 $({\rm{E},}L)-good$ if any, starts at $0$ which can be thought of as the previous particle reflecting at $-1$ to end at $0.$ Thus we are done.
\end{proof}

\begin{lemma}\label{rw13} Given $L\in \N,$ for each $i$,   $\ccW^{(i)}_{L}$ dominates the hitting time of $\{\pm \frac{L^2}{4}\}$ for a simple random walk on $\Z$ started from the origin.
\end{lemma}
\begin{proof}Let $\tilde t_i$ be the last $L+j-$good time in the interval $T^{(i)}_{L}$ and $T^{(i+1)}_{L}$  for some $j>0$ on any side of the $x-axis.$  Thus by definition for all $\tilde t_i <t<T^{(i+1)}_{L}$, the configuration outside $[-L,L]$ does not change.  Moreover by Remark \ref{r:modrun}, at $\tilde t_i$ both $[-L,-1]$ and $[1,L]$ are monochromatic of opposite colors and at $T^{(i+1)}_{L}$ the intervals $[-L+1,-1]$ and $[1,L-1]$ are monochromatic with different colors than at $\tilde t_i$. Thus $|M^{\tilde t_i}-M^{T^{(i+1)}_{L}}|\ge L^2$ and hence by triangle inequality 
$$\max(|M^{\tilde t_i}-M^{T^{(i)}_{L}}|,|M^{T^{(i)}_{L}}-M^{T^{(i+1)}_{L}}|)\ge \frac{L^2}{2}$$ and hence we are done as $M^s$ is a simple random walk of step size $2.$
\end{proof}

For any $k,$ let $N_1(t):=N_1(t,k)$ be such that $T^{(N_1(t))}_{{\rm E},k}<t \le T^{(N_1(t))+1}_{{\rm E}, k}.$  Similarly define $N_2(t)$ but replacing $\rm{E}$ by $\rm{W}.$
Then by definition,
\begin{equation}\label{boundfrac}
\frac{G(k,t)}{t}\leq \frac{\sum_{i=1}^{N_1(t)} \cF^{(i)}_{{\rm E},k,k}}{T^{(1)}_{{\rm E},k}+\sum_{i=1}^{N(t)-1} \ccW^{(i)}_{{\rm E},k}}+\frac{\sum_{i=1}^{N_2(t)} \cF^{(i)}_{{\rm W},k,k}}{T^{(1)}_{{\rm W},k}+\sum_{i=1}^{N(t)-1} \ccW^{(i)}_{{\rm W},k}},
\end{equation}
since the first term is an upper bound on the fraction of $({\rm{E}},k)-good,$ microsteps and the second term is a bound on $({\rm{W}},k)-good,$ microsteps.

\subsection{Coupling with a killed renewal process}  
We now use \eqref{boundfrac} and Lemmas \ref{rw12} and \ref{rw13} to bound $\frac{G(L,t)}{t}.$ 

\begin{proposition}\label{bound23}
Fix any $\delta>0.$ For any $L,$ there exists $\alpha>0$ such that for all large $t,$ with probability at least $1-e^{-ct^{\alpha}},$ for all $L\le k\le t,$
$$\frac{G(k,t)}{t}\le \frac{1}{k^{2-\delta}}.$$ 
\end{proposition}
For any fixed $k,$ it turns out  that $\frac{G(k,t)}{t}\le \frac{C}{k^{2}},$ with high probability. However for our purposes, we would need the above bound, uniformly over $k,$ and hence as a result we pay the arbitrarily small $\delta$ in the exponent.  
Before proving the above bound, we show how to prove Theorem \ref{t:mainparticle} using it. 
Fix $L\in \N$. As outlined in Section \ref{s:outline}, at this point we  consider the following killed version  $X^{\rm{killed}}(\cdot):=X_{L}^{\rm{killed}}(\cdot)$ of the erosion process. 
The setup is the following:
\begin{itemize}
\item It is a version of actual erosion process, but we now restrict our attention only to the interval $[ -L,L ].$
\item The starting configuration is either $[\underbrace{{\rm{R}},\ldots {\rm{R}}}_{L \text{ times}}, 0, \underbrace{{\rm{B}},\ldots {\rm{B}}}_{L \text{ times}}]$ or $[\underbrace{{\rm{B}},\ldots {\rm{B}}}_{L \text{ times}}, 0, \underbrace{{\rm{R}},\ldots {\rm{R}}}_{L \text{ times}}].$
\item A particle either red or blue is emitted at the origin, and then subsequently the color of the particle is alternated till a particle exits the interval $[ -L,L ],$ at which point the process is killed. 
\end{itemize}
Let $\ccR^{(L)}$ and $\ccQ^{(L)}$ be respectively the number of particles emitted and microsteps taken in this process.
By obvious symmetry, the laws of  $\ccR=\ccR^{(L)}$ and $\ccQ=\ccQ^{(L)}$ do not depend on whether initially   $[1,L]$ is colored red or blue or whether the initial particle is red or blue.
However it will be convenient for us to denote by $X^{\rm{killed}}_{a_1,a_2}(\cdot),$ the law of the process where the initial configuration has color $a_1\in \{\rm{R},\rm{B}\}$ on $[1,L]$  and the starting particle has color $a_2\in \{\rm{R},\rm{B}\}.$ 
The next lemma which follows directly from the discussion in the proof of Lemma \ref{l:layers} and Remark \ref{r:modrun}, states that when the above process is killed, the configuration is still monochromatic on each side and hence looks like the configuration at time zero.
\begin{lemma}\label{killmon}$X^{\rm{killed}}_{a_1,a_2}(\ccR^{(L)})$ is monochromatic on each side of the origin with opposite colors.
\end{lemma}

Thus $\ccR^{(L)}$ is distributed as $A_{L,L+1}-A_{L,L}$ where the latter were defined in \eqref{e:Ann}.
By Lemma \ref{l:Mrandwalk},  $\ccQ^{(L)}$ is distributed as the number of steps of a random walk path of step size $2$ starting from $L(L+1)$, stopped when it either increases by $2(L+1)$ or decreases by $-2L(L+1)-2(L+1)=-2(L+1)^2$, 
and  hence $\E(\ccQ^{(L)})=(L+1)^3$. 
\begin{lemma}\label{c:valueofalpha} For $\ccR^{(L)}, \ccQ^{(L)}$ as defined above,
\begin{equation}\label{ratio123}
\lim_{L\rightarrow \infty}\frac{\E(\ccR^{(L)})}{\E(\ccQ^{(L)})}=\frac{1}{\alpha},
\end{equation}
where $\alpha$ appears in Theorem \ref{t:mainparticle}.
\end{lemma}
The proof of this is based on a recursion and the formal details are presented in Section \ref{rec879}.
However first we finish the proof of Theorem \ref{t:mainparticle}. We will rely on the following coupling between the actual process and the killed process. To formally state the coupling recall from \eqref{e:Ann}, that $A_{L,L}$ is the number of particles emitted when $X(\cdot)$ reaches a configuration which is monochromatic on either side of the origin on the intervals  $[-L, -1]$ and $[ 1, L].$ 
The next lemma couples the original process with a sequence of killed processes with different initial configurations, and different colors of the initially emitted particles. Recall from the statement of Theorem \ref{t:mainparticle}, that $V_n$ is the total number of microsteps taken by the first $n$ particles in the process $\{X(i)\}_{i=1}^n.$ Also recall the notation $G(L,t)$ from Proposition \ref{bound23}.
\begin{proposition}\label{coupling12}
There exists a coupling of the processes $\{X(A_{L,L}+n)\}_{n\ge 1}$ and a sequence of process $ {\left\{X^{{\rm{killed}},(i)}_{a_i,b_i}(\cdot)\right\}}_{i=1}^{m_n}$ where $X^{{\rm{killed}},(i)}_{a_i,b_i}(\cdot)$ is an independent copy of the process $X^{{\rm{killed}}}_{a_i,b_i}$ where $a_i, b_i$ are functions of the process  $\{X(A_{L,L}+n)\}_{n\ge 1}$ and $m_n$ is a non-decreasing random sequence such that the following holds: 
\begin{enumerate}
\item  $|\sP_n-n|\le \sum_{\ell=L+1}^{\infty}G(\ell,V_n)$ for all $n$ where $\sP_n=\sum_{i=1}^{m_n}\ccR^{(L)}_i$ where $\ccR^{(L)}_i$ is the total number of particles emitted during the $i^{th}$ process $X^{{\rm{killed}},(i)}_{a_i,b_i}(\cdot).$
\item  Moreover,
$$ 
\sT_{n} \le V_n \le \sT_{n}+\sum_{\ell=L+1}^{\infty}G(\ell,V_n).
$$ where $\sT_n=\sum_{i=1}^{m_n}\ccQ^{(L)}_i$ where $\ccQ^{(L)}_i$ is the total number of microsteps taken during the $i^{th}$ process $X^{{\rm{killed}},(i)}_{a_i,b_i}.$
\end{enumerate}

\end{proposition}

Before describing the above coupling, we show how to quickly finish the proof of Theorem \ref{t:mainparticle} using the above and Proposition \ref{bound23}.

\begin{proof}[Proof of Theorem \ref{t:mainparticle}] 
Observe that, using (2) above,
$$\frac{\sT_{n}}{n}\le \frac{V_n}{n}\le\frac{\sT_{n}+\sum_{\ell=L+1}^{\infty}G(\ell,V_n)}{n}\,. $$
Moreover, notice that with probability at least $1-e^{-n^c}$ for some $c>0,$
$$n\le V_n \le n^{5}\,,$$ where the first inequality is deterministic and the second inequality follows from standard random walk estimates (by Lemma \ref{tailmart45} stated later, for $\delta=1/4$, with probability at least $1-e^{-n^c}$, $\sup_{t\leq n^5}|M^t|\geq n^{5/2-2\delta}$, and deterministically, if $\sup_{t\leq n^5}|M^t|\geq n^{5/2-2\delta}$ then $S^{n^5}\geq n^{5/4-\delta}\ge n$ and hence $V_n\le n^5$).
Fixing $L>0$, using the above and Proposition \ref{bound23}, followed by an union bound over $n\le t\le n^5 $ it follows that with probability at least $1-e^{-n^c}$ for some $c>0,$ we have
$\sum_{\ell=L+1}^{\infty}G(\ell,V_n)\le \frac{V_n}{\sqrt L}$ and thus by (2) above,
$$V_{n}(1-\frac{1}{\sqrt{L}})\le \sT_{n}\le V_n .$$ 

Moreover by the law of large numbers, it follows that $\frac{\sP_n}{\sT_{n}} $ converges almost surely to $\frac{\E(\ccR^{(L)})}{\E(\ccQ^{(L)})}$. Now using Lemma \ref{c:valueofalpha}, for any $\e>0$ we can choose $L$ large enough so that $\frac{\E(\ccR^{(L)})}{\E(\ccQ^{(L)})}\ge \frac{1}{\alpha}-\e$ and hence  for all large $n$ with probability going to $1,$ we have  
$\sP_n\ge V_{n}(1-\frac{1}{\sqrt{L}})(\frac{1}{\alpha}-\e)$.
Thus using (1) above and the preceding discussion, $$\sP_n(1-\frac{C}{\sqrt L})\le n\le \sP_n(1+\frac{C}{\sqrt L}),$$ for some $C=C(\alpha)$. Thus 
we get $$\frac{\sT_{n}}{\sP_n}(1-\frac{C'}{\sqrt L})\le \frac{V_n}{n}\le\frac{\sT_{n}}{\sP_n}(1+\frac{C'}{\sqrt L})$$
for some constant $C'=C'(\alpha)$, and hence 
Theorem \ref{t:mainparticle} follows. 
\end{proof}

We now prove Proposition \ref{bound23}.
\begin{proof}[Proof of Proposition \ref{bound23}]  Recall the two terms on the RHS in \eqref{boundfrac}. We will provide bounds only for the first term and omit the completely symmetric details for the second term. Also for notational brevity we will drop the $\rm{E}$ in the notation. Thus we will bound 
$$ \frac{\sum_{i=1}^{N_1(t)} \cF^{(i)}_{k,k}}{T^{(1)}_{k}+\sum_{i=1}^{N_1(t)-1} \ccW^{(i)}_{k}}.$$

The proof considers two cases $k \le t^{\gamma}$ and $k >t^{\gamma}$ for some to-be-later specified value of $\gamma$.
The first part of the proof shows that for any $k \le t^{\gamma},$ 
$N_1(t,k)$  is large. 
In this case, for our purposes we can afford to use the following rather crude bound:
 \begin{equation}\label{lboc34}
 \P(N_1(t,k)\geq t^{1/4-\delta})\geq 1-e^{-t^c}
 \end{equation} which is proved in Lemma \ref{lboccur}.
Note that  by Lemma \ref{rw12},  the terms in the numerator in \eqref{boundfrac} are sub-exponential variables at scale $k^2$.  Since $N_1(t,k)$ is large for $k\le t^{\gamma},$ 
this allows us to use concentration results to bound the numerator. 
Also notice that by Lemma \ref{rw13}, the terms in the denominator, dominates a sub-exponential variable at scale $k^4$.
This shows that the ratio in \eqref{boundfrac} is bounded  approximately by $\frac{1}{k^2}$.
 Formally we use,
 \[\P\left(\frac{\sum_{i=1}^{N_1(t,k)}\cF^{(i)}_{k,k}}{N_1(t,k)}\geq 2\E(\cF^{(1)}_{k,k}))\right)\leq e^{-t^{c}}+\sum_{m=t^{1/4-\delta}}^{t} \P\left(\frac{\sum_{i=1}^{m}\cF^{(i)}_{k,k}}{m}\geq 2\E(\cF^{(1)}_{k})\right).\]
 Similarly
 \[\P\left(\frac{\sum_{i=1}^{N_1(t,k)}\ccW^{(i),*}_{k}}{N_1(t,k)}\leq \frac{1}{2}\E(\ccW^{(1),*}_{k}))\right)\leq e^{-t^{c}}+\sum_{m=t^{1/4-\delta}}^{t} \P\left(\frac{\sum_{i=1}^{m}\ccW^{(i),*}_{k}}{m}\geq 2\E(\ccW^{(1),*}_{k})\right).\]
 where $\ccW^{(i),*}_{k}$ are i.i.d. copies of hitting time of $\pm\frac{k^2}{4},$ for a standard random walk on $\Z$ started at the origin.
Thus by simple union bound and the following estimate, the probabilities on the LHS in the above two expressions are both at most $e^{-t^c}$ for some $c>0$. 
There exists a universal $c>0$ such that for any $k,m>0,$
\begin{align}\label{con1}
\P\left(\frac{\sum_{i=1}^{m}\cF^{(i)}_{k,k}}{m}\geq 2\E(\cF^{(1)}_{k,k}))\right)&\leq e^{-cm},\\
\P\left(\frac{\sum_{i=1}^{m}\ccW^{(i),*}_{k}}{m}\leq \frac{1}{2}\E(\ccW^{(1),*}_{k})\right)&\leq e^{-c'm}.
\end{align}
The above follows from standard concentration of sub-exponential variables \cite{J14} and Lemmas \ref{rw12} and \ref{rw13}.
Now as $\E(\cF^{(i)}_k)=O(k^2)$ and $\E(\ccW^{(i,*)}_k)=\Theta(k^4)$, we have,
\begin{eqnarray*}
\P\left(\frac{G(t,k)}{t}\geq \frac{C}{k^2}\right)&\leq & \P\left(\frac{\sum_{i=1}^{N_1(t,k)}\cF^{(i)}_{k,k}}{\sum_{i=1}^{N_1(t,k)-1}\ccW^{(i,*)}_k}\geq \frac{C}{k^2}\right)\\
&\leq & \P\left(\frac{\sum_{i=1}^{N_1(t,k)}\cF^{(i)}_{k,k}}{N_1(t,k)}\geq 2\E(\cF^{(1)}_{k,k})\right)+\P\left(\frac{\sum_{i=1}^{N_1(t,k)}\ccW^{(i,*)}_k}{N_1(t,k)}\leq \frac{1}{2}\E(\ccW^{(1,*)}_k)\right)\\
&\leq & e^{-t^{c}}.
\end{eqnarray*}

Now fix $ t^\alpha\leq k \leq t$. In this regime we would not argue largeness of $N_1(t,k)$ but use the fact each of the entries in the denominator of \eqref{boundfrac} is large compared to the corresponding term in the numerator and this would suffice to show that the ratio is small even if the number of terms in the sum $N_1(t,k)$, is small.
Formally from \eqref{boundfrac} and  union bound we have,
\begin{align*}
\P\left(\frac{G(t,k)}{t}\geq \frac{C}{k^{2-2\delta}}\right)& \le \sum_{i=1}^{t}\P(\cF^{(i)}_{k,k}\ge k^{2+\delta})+\sum_{i=1}^{t}\P(\ccW^{(i,*)}_k\le k^{4-\delta})+\P(T^{(1)}_{K}\le k^{4-\delta}),\\
&\le e^{-t^{c}},
\end{align*}
where the last inequality follows from exponential tails of $\cF^{(i)}_{k,k}$ and $\cW^{(i,*)}_{k}$ at scales $k^2$ and $k^4$ respectively. 
\end{proof}
We now finish the proof of \eqref{lboc34}. We will start with the following lemma.
\begin{lemma}\label{tailmart45} For all small $\delta>0$ there exists $c>0,$ such that for all large enough $T,$
$$\P\left(\sup_{t\leq T}|M^t|\geq T^{1/2-2\delta}\right)\geq 1-e^{-T^{c}}.$$
\end{lemma}
\begin{proof} Recall the concatenated walks $(\scE'_1;\scE'_2;\ldots)$ and $(\scF'_1;\scF'_2;\ldots)$ from \eqref{e:diffF}.  Also recall \eqref{e:Z1} and \eqref{e:Z2} and without loss of generality let us assume that the former holds for some $k.$ Now by \eqref{e:Z1}, either $\sum_{i=1}^{k}N(\scE_i)\ge \frac{T}{2}$ or $\sum_{i=1}^{k-1}N(\scF_i)\ge \frac{T}{2}$.
Now as already mentioned in the discussion right after \eqref{e:Z1}, the concatenated walk  $(\scE'_1;\scE'_2;\ldots;\scE'_k)$ is a random walk of step-size 2 run starting from 0 till it hits $\{\pm (k+1)(k+2)\}$ and has run for $\sum_{i=1}^{k}N(\scE_i)$ steps.
On the other hand by \eqref{e:diffF}, $(\scF'_1;\scF'_2;\ldots; \scF'_{k-1})$ is a random walk of step size 2 which hits $\{\pm((k+2)^2-1-(k+1))\},$ and has run for $\sum_{i=1}^{k}N(\scF'_i)$ steps.
Now the result follows from the following straightforward random walk estimate: for a standard random walk $\{X_s\}_{s\ge 1}$ on $\Z,$ and any large enough $m>0$
$$\P\left(\sup_{s\le m}|{X_s}|\ge m^{1/2-\delta}\right)\ge 1-e^{-m^{c}}$$ for some $c=c(\delta)>0.$ whose proof follows by observing that there exists a universal constant $c>0$ such that uniformly from any point in the interval $[-m^{1/2-\delta},m^{1/2-\delta}]$ the chance to exit the interval in the next $m^{1-2\delta}$ is $c>0$ independent of $m$ and $\delta$. (see \cite{lpw} for more details.)
\end{proof}

\begin{lemma}\label{lboccur}Fix $k\leq T^\gamma$ where $\gamma$ is some sufficiently small positive constant, one has 
\[\P(N_1(T,k)\geq T^{1/4-\delta})\leq 1-e^{-T^c}.\]
where $c=c(\gamma)>0.$
The same result holds for $N_2(T,k).$
\end{lemma}

\begin{proof} 
The proof is based on the fact that after $T$ microsteps have been taken, the number of sites to be explored is approximately $T^{1/4}.$ Now for any $j$ large enough, we will show that there is a significant chance of a particle being $(\rm{E},k)-good$ among the particles emitted between the times that the number of sites explored went from $j$ to $j+1.$
To this end note that  by \eqref{e:Z1} and \eqref{e:Z2}, deterministically if $\sup_{t\leq T}|M^t|\geq T^{1/2-2\delta}$, then $S_{\cE}^T\geq \frac{T^{1/4-\delta}}{2}.$ Thus  
\[\P\left(S_{\cE}^T\geq \frac{T^{1/4-\delta}}{2}\right)\geq \P\left(\sup_{t\leq T}|M^t|\geq T^{1/2-2\delta}\right)\geq 1-e^{-T^{c}}.\] 
 where the last inequality follows from the previous lemma. Recall the notations $A_{i,i}$ and $A_{i,i+1}$ from \eqref{e:Ann}. Let $Z_i$ denote the event that there is a particle with index between, $A_{i,i}$ and $A_{i,i+1}$, which is $({\rm{E}},k)-good.$  Clearly then 
\[N_1(T,k)\geq \sum_{i=\frac{cT^{1/4-\delta}}{2}}^{S_{\cE}^T-1} \ind(Z_i).\]
Recall from Lemma \ref{l:Mrandwalk}, that the microsteps between $A_{i,i}$ and $A_{i,i+1}$  correspond to a random walk segment $\scE_i$ that goes from $i(i+1)$ or $-i(i+1)$ to $\pm (i+1)(i+2)$. Without loss of generality, assume that the value of $M$ starts from $i(i+1)$ i.e., $[1,i]$ is colored blue and $[-i,-1]$ is colored red. 
One can verify that $Z_i$ occurs if the value of $M$ hits $i(i+1)-2k(k+1)$ before hitting $(i+1)(i+2)$ since at this point the interval $[-k,-1]$  is monochromatic colored blue and $[1,k]$ is colored red which implies that there must have been a particle  that had been $({\rm{E}},k)-good.$

By standard Gambler's ruin computations,  
\[\P(Z_i)\ge \frac{2(i+1)}{k(k+1)+2(i+1)}\geq \frac{1}{2},\]
by choosing $\gamma$ sufficiently small so that $1/4-\delta>2\gamma$.
Also the events $Z_i$ are clearly independent. Hence, 
\begin{eqnarray*}
\P(N_1(T,k)\leq T^{1/4-2\delta})&\leq & e^{-T^{4c}}+\P(N_1(T,k)\leq T^{1/4-2\delta}, S^T_\cE\geq T^{1/4-\delta})\\
&\leq &e^{-T^c}+\sum_{m=T^{1/4-2\delta}}^T\P\left(\sum_{i=\frac{T^{1/4-2\delta}}{2}}^m \ind(Z_i)\leq T^{1/4-3\delta} \right)\\
&\leq & e^{-T^{c}}.
\end{eqnarray*}
\end{proof}

The only thing left is the proof of Proposition \ref{coupling12}.
\begin{proof}[Proof of Proposition \ref{coupling12}] Recall the notation $\ccR^{(L)}_i$ from the statement of the proposition. 
The coupling is rather natural and simple to describe: Let us start with the $A_{L,L}+1^{th}$ particle. By obvious symmetry and using the same random walks, one can exactly couple $X(A_{L,L}+\cdot)$ and $X^{{\rm{killed}},(1)}_{a_1,b_1}(\cdot)$ where $a_1$ and $b_1$ are determined by the configuration $X(A_{L,L})$ and the parity of $A_{L,L}$. Note that under this  coupling, the two processes stay exact up to killing of the latter process.
After the latter process is killed,  the particle which exited $[-L,L]$ still continues to move in the former process. 

Now one of two things can happen:
\begin{itemize}
\item
The particle settles outside $[-L,L].$   Note that so far exactly $\ccR^{(L)}_1$ many particles have been emitted in both the processes. 
Now by Lemma \ref{killmon}, the configuration $X(\ccR^{(L)}_1)$ is still monochromatic on $[-L,-1]$ and $[1,L]$, with opposite colors. 
Thus we can again use a similar coupling as above to exactly couple
$X(A_{L,L}+\ccR^{(L)}_1+\cdot)$ and $X^{{\rm{killed}},(2)}_{a_2,b_2}(\cdot)$ for an appropriate choice of $a_2,b_2.$

\item Note that it might also happen that the particle which exited $[-L,L],$  eventually returns to the origin.
From this point onwards we can couple the microsteps in the process $X(\cdot)$ with $X^{{\rm{killed}},(2)}_{a_2,b_2}(\cdot)$ for an appropriate choice of $a_2$ and $b_2$ in the natural way till the latter process gets killed.
\end{itemize}
We continue as above to build the coupling for the entire process $\{X(A_{L,L}+n)\}_{n\ge 1}$. 
Note that due to occurrences of the second case above, total number of particles emitted in both the processes begin to differ since in the former process the particle that returns to the origin continued to move while in the latter process a new particle is emitted at the origin to couple with the former particle.

But by definition the number of particles over counted in the latter is clearly upper bound by the number of microsteps taken in the former process that are $\ell-good,$ for some $\ell\ge L$ and hence the first bound in the statement of Proposition \ref{coupling12} follows. 
The  second statement about the difference in microsteps in the two processes follows by a similar argument. We omit the details.
\end{proof}

\subsection{Convergence of expectation:}\label{rec879}
In this subsection we prove Lemma \ref{c:valueofalpha}.
Let, 
\begin{equation}\label{e:defwn}
w_k:=\E(\ccQ^{(k)})=\E(A_{k,k+1}-A_{k,k}), 
\end{equation}
where $A_{k,k}, A_{k,k+1}$ are as defined in \eqref{e:Ann} and \eqref{e:Ann+1}. 
The proof of the  lemma relies on the following  recursive relation between $w_k$ and $w_{k-1}$.

\begin{proposition}\label{rec56} With the above definitions, $w_0=1$ and
\begin{equation}\label{e:recursion}
w_k=w_{k-1}\frac{(k+1)(k+2)}{k^2}-\frac{1}{k}
\end{equation}
for all $k\geq 1$.
\end{proposition}
Before proving the above  we now finish the proof of Lemma \ref{c:valueofalpha}.
\begin{proof}[Proof of Lemma \ref{c:valueofalpha}]
For $k\in \N$, let 
$c_{k-1}=\frac{(k+1)(k+2)}{k^2}.$
Solving this recursion, we get
$w_k=w_0\prod_{i=0}^{k-1} c_i-\sum_{j=1}^k\frac{\prod_{i=j}^{k-1}c_i}{j}.$
Also, it is easy to see that,
$\prod_{i=j}^{k-1} c_i=\frac{(j+2)(j+3)^2\ldots (k+1)^2(k+2)}{(j+1)^2(j+2)^2\ldots k^2}=\frac{(k+1)^2(k+2)}{(j+1)^2(j+2)}.$
Hence,
$\frac{w_k}{(k+1)^2(k+2)}=\frac{w_0}{2}-\sum_{j=1}^k\frac{1}{j(j+1)^2(j+2)},$
and thus,  $\frac{w_k}{k^3}\rightarrow \frac{1}{2}-\sum_{j=1}^\infty\frac{1}{j(j+1)^2(j+2)},$ as $k$ goes to infinity.

\end{proof}

We finish with the proof of Proposition \ref{rec56}.
\begin{proof}[Proof of Proposition \ref{rec56}] By Lemma \ref{l:explorconfig}, at steps $A_{k,k},A_{k,k+1}$, both sides of the origin are monochromatic and of opposite color. Assume without loss of generality that $A_{k,k}$ is odd, so that at step $A_{k,k}+1$, a red particle is emitted. Also, assume that 
\[X(A_{k,k})([ -k, -1])=(\cR,\cR,\ldots,\cR), \quad X(A_{k,k})([ 1, k])=(\cB,\cB,\ldots,\cB), \quad X(A_{k,k})(z)=0 \mbox { o.w.}.\]
 That is, after the $A_{k,k}^{\mbox{\tiny{th}}}$ round, the monochromatic run of length $k$ on the positive axis is blue. 
 Clearly, from the assumptions, and Lemma \ref{l:Mrandwalk},
\[M^{A^M_{k,k}}=k(k+1),\]
where $M^t$ is as defined in \eqref{e:defM}, and $\{M^t:t\in [A^M_{k,k}, A^M_{k,k+1}\}$ is a random walk of step-size $2$ from $k(k+1)$ till it hits $\pm (k+1)(k+2)$. 
The main observation leading to the recursion in Proportion \ref{rec56}, is that, between times $A^M_{k,k}$ and $A^M_{k,k+1}$ if we look at the intermediate time when a new particle reaches  $-k$ or $k,$  then the average  number of particles emitted from $A^M_{k,k}$ till so far is $w_{k-1}.$ 
This is denoted by the following:
Let 
 \[T:=\inf \{t\geq A_{k,k}^M: M^t\in \{k(k+3),-(k-1)k\}\}.\]
From $Z^{A_{k,k}^M}$, the Markov chain $Z^{T}$, as defined in \eqref{e:defZ}, can be in either of the two following states
\begin{itemize}
\item State 1: Corresponding to the case  $M^T=k(k+3),$ which occurs with probability $1-\frac{1}{k+1}$ (when a red particle reaches $-k$)
 \item State 2: Corresponding to the case  $M^T=-(k-1)k,$
which occurs with probability $\frac{1}{k+1}$ (when a red particle reaches $k$)
\end{itemize}
Let $T^P$ be the number of particles emitted till the microstep $T$. 
The fact that $T^P-A_{k,k}=w_{k-1}$  is evident from the following discussion about the increments of the process $M^s.$ Without loss of generality, let us assume that $M^{A_{k-1,k-1}^M}=(k-1)k$ (instead of $-(k-1)k$). Now note that as the process goes from state $(k-1,k-1)$ to $(k-1,k)$ or $(k,k-1),$ the value of $M^t$ reaches value $\pm(k(k+1))$ and then immediately at $A_{k-1,k}^M$ it becomes $\pm k^2.$ Let us denote this value `instantaneously' before $A_{k-1,k}^M$ as $M^{A_{k-1,k}^M,-}.$
 When the latter is $k(k+1)$ then note that 

 \[M^{A_{k-1,k}^M,-}-M^{A_{k-1,k-1}^M}=k(k+1)-(k-1)k=2k=k(k+3)-k(k+1)=M^T-M^{A_{k,k}^M},\]
 if $T$ is as in State 1. When $M^{A_{k-1,k}^M,-}=-k(k+1)$,
 \[M^{A_{k-1,k}^M,-}-M^{A_{k-1,k-1}^M}=-k(k+1)-(k-1)k=M^T-M^{A_{k,k}^M},\]
if $T$ is as in State $2$. 
 
 We consider the following  two cases separately. 

\textbf{State 1:} Since at step $A_{k,k}+1$, a red particle is emitted, it is easy to see that, at $T$, the only configuration possible corresponding to the value of $M^T$ is,
\[Z_1^T(x)=\cR \quad \mbox {for all } x\in [ -k,-1 ], \quad Z_1^T(x)=\cB \quad \mbox {for all } x\in [ 1,k ], \quad Z_1^T(x)=0 \quad \mbox{o.w.},\]
and 
\[Z_2^T=-k, \quad Z_3^T=\cR.\]
Note that one can naturally couple  the steps of the configuration from $A_{k,k}^M$ to $T$, and the steps of the configuration between $A_{k-1,k-1}^M$ and $A_{k-1,k}^M$ (with a red particle emitted at step $A_{k-1,k-1}+1$), where
\begin{equation}\label{e:an-1.1}
X(A_{k-1,k-1})(z)=\cR \mbox{ for all } z\in [ -(k-1),-1 ];\quad  X(A_{k-1,k-1})(z)=\cB \mbox{ for all } z\in [ 1,k-1 ]; 
\end{equation}
and
\begin{equation}
\label{e:bn-1.1}
X(A_{k-1,k})(z)=\cR \mbox{ for all } z\in [ -k,-1 ];\quad  X(A_{k-1,k})(z)=\cB \mbox{ for all } z\in [ 1,k-1 ];
\end{equation}
(by using the same random walks  $Y^i$ for the two processes to be equal), and hence the number of particles emitted in the two cases have the same distribution.

\textbf{State 1.1}: Note that at State 1, there is an extra red particle at $-k.$ Hence there are two possibilities. Let $Y'(t)=Y^{T^P}\left(t+T-\sum_{i=1}^{T^P-1} \tau^i\right)$ denote the steps of the random walk performed after microstep $T$ by this additional red particle starting at  $Y'(0)=-k$. If $$T_1:=\inf\{j\geq 1: Y'(j)\in \{-(k+1),1\} \},$$
then either 
$Y'(T_1)=-(k+1),$
which happens with probability $1-\frac{1}{k+2}$.
In this case, we reach the configuration with explored territory of the form $(k+1,k)$ and hence we reach $A_{k,k+1}^M$ from $T$ without emitting any new particle. Otherwise 
$Y'(T_1)=1,$
which happens with probability $\frac{1}{k+2}$. This gives rise to the configuration
\begin{equation}\label{e:configz}
Z_1^{T_1}([-k,-1])=(\cR,\cR,\ldots,\cR), \quad Z_1^{T_1}([1,k])=(\cR,\underbrace{\cB,\cB,\ldots,\cB}_{k-1}), \quad Z_1^{T_1}(z)=0 \mbox{ o.w. }, Z_2^{T_1}=0.
\end{equation}
Let $z_k$ be the expected number of particles emitted starting from the above configuration in \eqref{e:configz} till one reaches the configuration in $A_{k,k+1}$. We claim, 
\begin{equation}\label{e:defz}
w_k=\left(1-\frac{1}{k+2}\right)(1+z_k)+\frac{1}{k+2}\times 1.
\end{equation}
This is easy to see, as starting from the configuration in $A_{k,k}$, the emitted red particle either ultimately sits at $1$ ( with probability $1-\frac{1}{k+2}$) yielding the configuration in \eqref{e:configz}, or at $-(k+1)$ (with probability $\frac{1}{k+2}$) in which case we the explored territory is $(k+1,k)$ and hence in the latter case only one particle was emitted while in the former $(1+z_k)$ many particles are emitted on average. 

Thus above we have related the number of new particles emitted after State 1 has been reached to $w_k$. Below we discuss what happens if instead we are in State 2.

\textbf{State $2$:} Since at step $A_{k,k}+1$, a red particle is emitted, one can check that in this case, at $T$, the configuration corresponding to the value of $M^T$ is,
\begin{equation}\label{confstep2}
Z_1^T([-k,-1])=(\cR,\underbrace{\cB,\cB,\ldots,\cB}_{k-1}); Z_1^T([1,k])=(\cR,\cR,\ldots,\cR), Z_1^T(x)=0 \mbox{ o.w. }, Z_2^T=0
\end{equation}
As in the previous case one observes that the steps of the configuration from $A_{k,k}^M$ to $T$ can be coupled naturally with the steps of the configuration  from $A_{k-1,k-1}^M$ to $A_{k-1,k}^M$ (with a red particle emitted at step $A_{k-1,k-1}+1$), where
\begin{equation}\label{e:an-1.2}
X(A_{k-1,k-1}([-(k-1),-1])=(\cR,\cR,\ldots,\cR); X(A_{k-1,k-1})([1,k-1])=(\cB,\cB,\ldots,\cB), 
\end{equation}
and
\begin{equation}\label{e:bn-1.2}
X(A_{k-1,k}([-(k-1),-1])=(\cB,\cB,\ldots,\cB); X(A_{k-1,k})([1,k])=(\cR,\cR,\ldots,\cR), X(A_{k-1,k})(z)=0 \mbox{ o.w. } .
\end{equation}

Note that \eqref{e:bn-1.1} and \eqref{e:bn-1.2} are the only two possible configurations at $A_{k-1,k}$ to be reached from the configuration at $A_{k-1,k-1}$ by starting with a red particle being emitted at step $A_{k-1,k-1}+1$, as the new site explored must be through red particle. This formalizes the claim that the number of particles emitted up to $T$ can be related to $w_{k-1},$ i.e.,
\[\E(T^P-A_{k,k})=w_{k-1}\,.\]

\textbf{State 2.1:}  However if we start with configuration in State 2 (described in \eqref{confstep2}), where $M^T=-(k-1)k,$ since $Z_2^T=0$, a new particle emits at this step, and this particle is necessarily blue. This is because 
$\cR(T^P)-\cB(T^P)=\cR^T-\cB^T> 0,$  so $T^P$ is even by Lemma \ref{l:comb}, hence  $Z_2^T=\cB$. Let $T_2$ be the first time the random walk $\{M^t: t\geq T\}$ hits $\pm k(k+1)$. As before, there are two possibilities: either  $M^{T_2}=-k(k+1)$, corresponding to the configuration
\begin{equation}\label{e:monn}
Z_1^{T_2}([-k,-1])=(\cB,\cB,\ldots,\cB);Z_1^{T_2}([1,k])=(\cR,\cR,\ldots,\cR),\,Z_1^{T_2}(x)=0 \mbox{ o.w. }, Z_2^{T_2}=0,
\end{equation}
 which in turn corresponds to the journey from the configuration
\begin{equation}\label{eq32} \sigma([-(k-1),-1])=(\cB,\cB,\ldots,\cB); \sigma([1,k-1])=(\cR,\cR,\ldots,\cR); \sigma(x)=0 \mbox{ o.w. },
\end{equation}
 to
 \[\sigma([-k,-1])=(\cB,\cB,\ldots,\cB); \sigma([1,k-1])=(\cR,\cR,\ldots,\cR); \sigma(x)=0 \mbox{ o.w. },\]
 (recall that $Z_1^T(k)=\cR$, so this gives the configuration in \eqref{e:monn});
  
the other possibility is $M^{T_2}=k(k+1)$, corresponding to 
\begin{equation}\label{e:monn2}
Z_1^{T_2}([-k,-1])=(\cR,\cR,\ldots,\cR);Z_1^{T_2}([1,k])=(\cB,\cB,\ldots,\cB), Z_1^{T_2}(x)=0 \mbox{ o.w. }, Z_2^{T_2}=0,
\end{equation}
 which in turn can be coupled with the steps from the configuration
 \begin{equation}\label{eq33}
\sigma([-(k-1),-1])=(\cB,\cB,\ldots,\cB); \sigma([1,k-1])=(\cR,\cR,\ldots,\cR); \sigma(x)=0 \mbox{ o.w. },
 \end{equation}
 to
 \[\sigma([-(k-1),-1])=(\cR,\cR,\ldots,\cR); \sigma([1,k])=(\cB,\cB,\ldots,\cB); \sigma(x)=0 \mbox{ o.w. },\]
(note that the new territory explored must be through blue particle). Again as before the above two cases in \eqref{eq32} and \eqref{eq33} are exactly the ones involved in the journey from $A_{k-1,k-1}$ to $A_{k-1,k}$ and hence in this stage the average number of particles is $w_{k-1}$.

\textbf{State 2.1.1:} Both the end configurations in  State 2.1 (\eqref{e:monn} and \eqref{e:monn2}) correspond to the original configuration at $A_{k,k}$ we started from. So it takes further $w_k$ expected number of particles from here to go to the configuration at $A_{k,k+1}$.
Thus, bringing all this together, 
\[w_k=w_{k-1}+\left(1-\frac{1}{k+1}\right)\frac{1}{k+2}z_k+\frac{1}{k+1}(w_{k-1}+w_k),\]
where by \eqref{e:defz}, 
$z_k=\frac{k+2}{k+1}\left(w_k-\frac{1}{k+2}\right)-1.$
Simplifying we get \eqref{e:recursion}.
\end{proof}

\section{Variants of competitive erosion}

Competitive erosion is quite sensitive to changes in the model definition.  In this concluding section we discuss several variants.

\subsection{Random color sequence}

In the model we studied, the color of the new particle alternates deterministically between red and blue (Figure~\ref{f.random}a). A different behavior emerges if instead the color of the new particle is random, red or blue with probability $1/2$ independent of the past (Figure~\ref{f.random}b).

\begin{figure}[h]
    \centering
    \begin{tabular}{cc}
    (a) & \includegraphics[width=.7\textwidth]{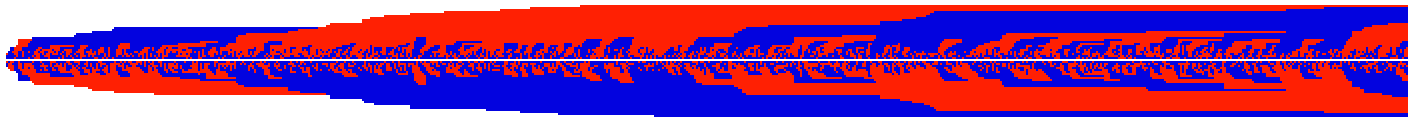}  \\
    (b) & \includegraphics[width=.7\textwidth]{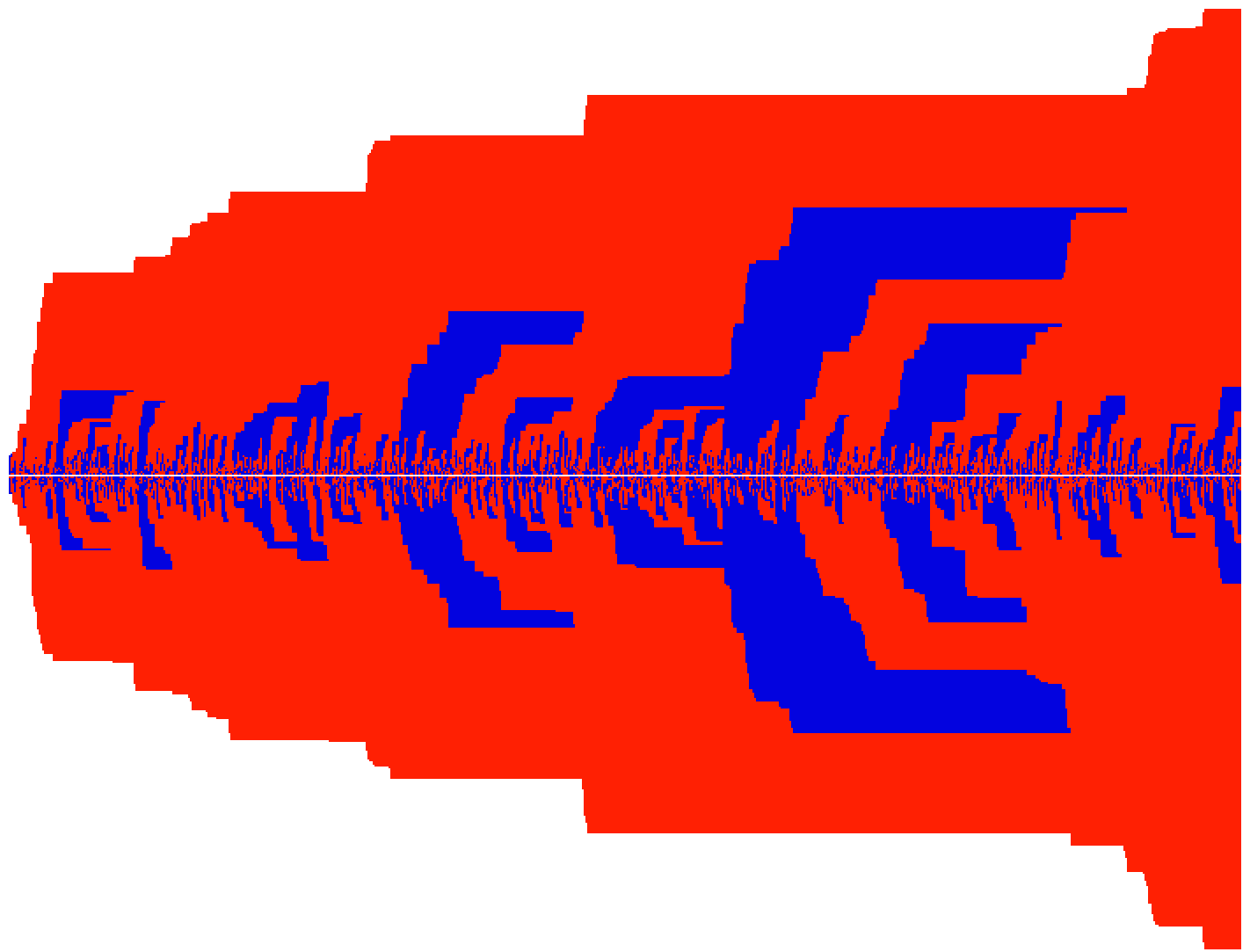} \\
    \end{tabular}
    \caption{Samples of competitive erosion in $\Z$ with (a) alternating red and blue particles; and (b) independent random colors, red and blue each with probability $1/2$. In each case, time increases from left to right, and each column depicts the random color configuration on $\Z$ after the addition of $2\cdot 10^5$ more particles.
    In (a) the color configuration is approximately antisymmetric about $0$ (by Lemma~\ref{l:layers}), and the number of occupied sites is order $n^{1/4}$ (by Theorem~\ref{informal}). 
 In (b) the color configuration seems approximately \emph{symmetric} about $0$, and the number of occupied sites seems to be order $n^{1/2}$.  
 To highlight the size difference between $n^{1/4}$ and $n^{1/2}$, the two figures are drawn at the same scale.
 } 
           \label{f.random}
\end{figure}

\subsection{Three colors in periodic sequence}

Consider $c \geq 1$ mutually antagonistic colors. At the $n$th time step a particle of color $n \;\mathrm{mod}\; c$ is released at the origin. The new particle performs simple random walk in $\Z$ until reaching a site of $\Z - \{0\}$ that is either uncolored or colored differently from itself, and converts that site to its own color.
The case $c=1$ is internal DLA (starting with the origin occupied).
The case $c=2$ is the one studied in this paper. The case $c=3$ is pictured in Figure~\ref{f.mutual}.

\begin{figure}[h]
    \centering
    \includegraphics[scale=0.65]{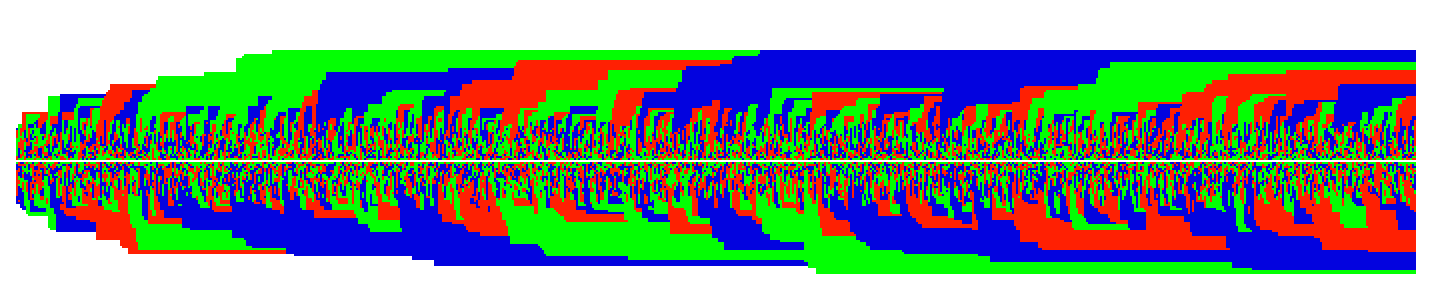} 
    \caption{(a) Sample of competitive erosion in $\Z$ with $3$ mutually antagonistic colors.  
    Time increases from left to right: each column depicts the random color configuration on $\Z$ after the addition of $10^5$ more particles of each color.  Is the number of occupied sites after $n$ particles still of order $n^{1/4}$?    
   }
    \label{f.mutual}
\end{figure}

A different kind of behavior can be seen if the periodic sequence of colors has repeated terms. Figure~\ref{f.parabola} shows the result of period $5$ with color sequence Blue, Red, Blue, Red, Green. In this case it appears that the number of occupied sites is order $n^{1/2}$, nearly all of them Red. 

\begin{figure}[h]
    \centering
    \includegraphics[width=1.0\textwidth, height=0.25\textheight]{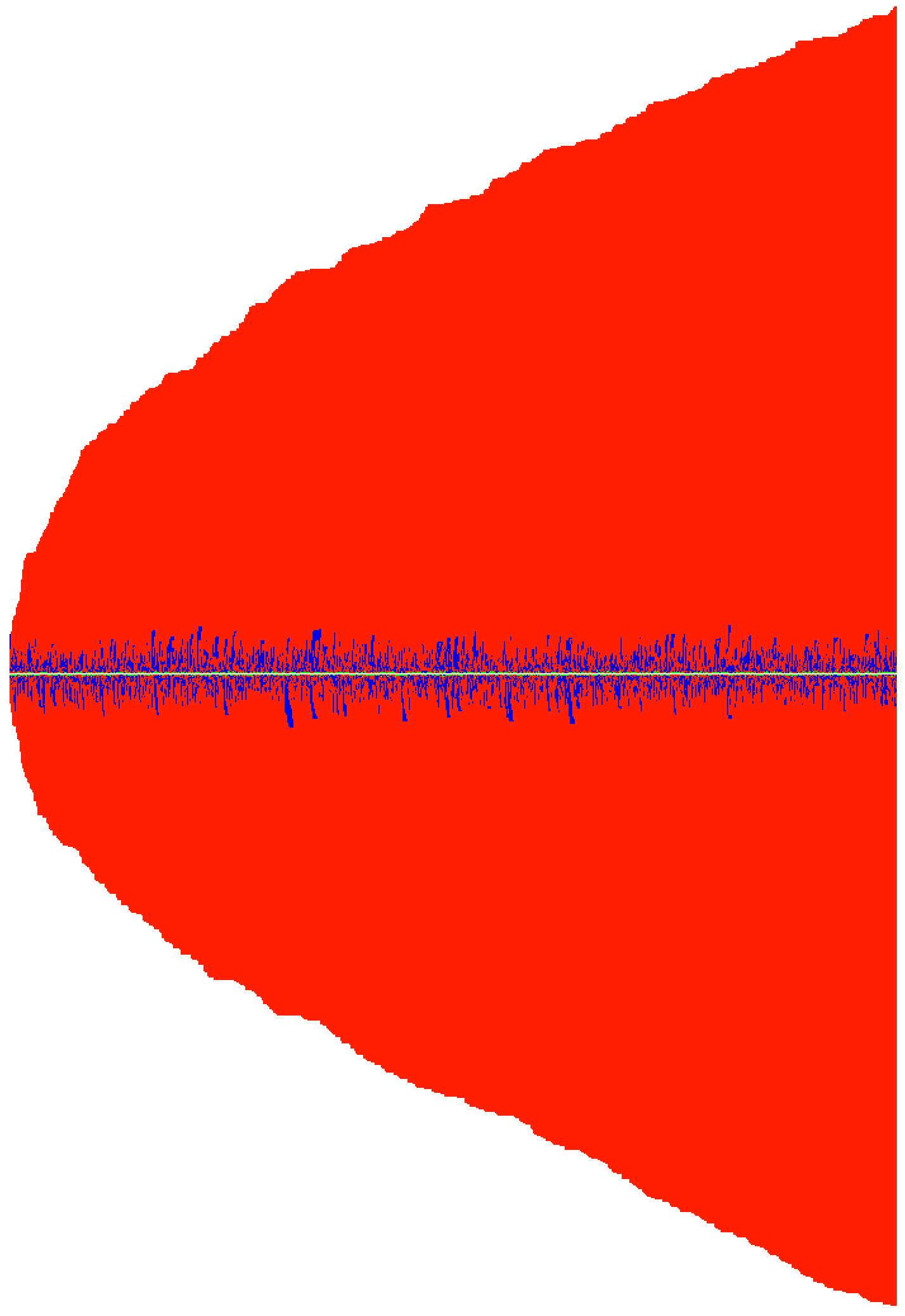} 
    \caption{Sample of competitive erosion in $\Z$ with periodic color sequence Blue, Red, Blue, Red, Green. 
    Time increases from left to right: each column depicts the random color configuration on $\Z$ after the addition of $5 \times 10^4$ more particles.        
    Nearly all sites end up Red, with a few Blue sites and a very few Green sites surviving near the origin. After $n$ particles, is the number of Red sites order $n^{1/2}$? Is the number of Blue and Green sites order $1$?
   }
    \label{f.parabola}
\end{figure}

\subsection{Cyclically antagonistic colors in periodic sequence}

Consider $c \geq 1$ colors as before, with a different stopping rule: A walker of color $k$ stops only upon reaching a site of $\Z$ that is either uncolored or of color $k-1 \;\mathrm{mod}\; c$. If $c \geq 3$ then there is no need to forbid stopping at the origin.

\begin{figure}[h]
\centering
\includegraphics[scale=0.65]{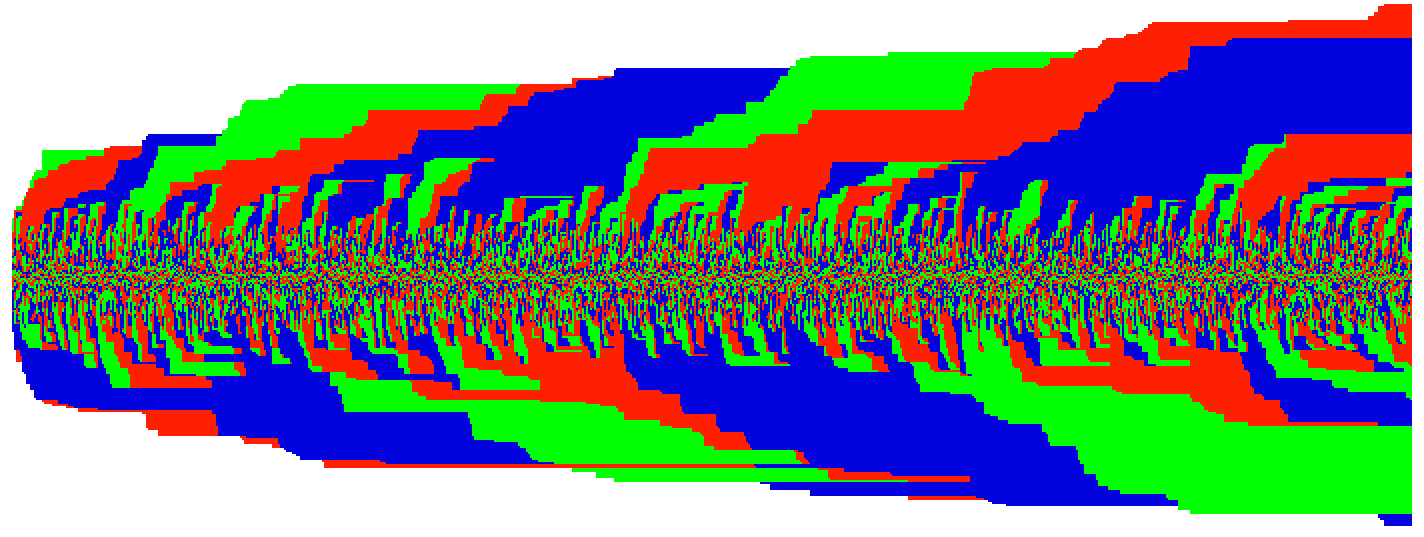}
\caption{Sample of competitive erosion in $\Z$ with periodic color sequence and cyclic antagonism: Blue walkers take over Red and uncolored sites, Red walkers take over Green and uncolored sites, Green walkers take over Blue and uncolored sites. Time increases from left to right: each column depicts the random color configuration on $\Z$ after the addition of $10^5$ more particles of each color. 
Compared with mutual antagonism (Figure~\ref{f.mutual}), more sites become colored. What is the growth rate of the number of colored sites? }
\label{f.cyclic}
\end{figure}

\subsection{More spatial dimensions}

Consider the competitive erosion process with Red and Blue particles alternately emitted from the origin in $\Z^d$ for $d \geq 2$. Each particle in turn performs simple random walk on $\Z^d$ stopped when it first hits an uncolored or oppositely colored site of $\Z^d-\{\mathbf{0}\}$, and converts that site to its own color. Figure \ref{f.spiral} shows the resulting random color configuration on $\Z^2$, and a slice of the resulting configuration on $\Z^3$; each displays surprisingly coherent red and blue territories. The pictures suggest that the set of colored sites grows quite slowly, as most colored sites are repeatedly converted from Red to Blue and back.  

\begin{figure}[h]
\centering
\begin{tabular}{cc}
\includegraphics[scale=0.5]{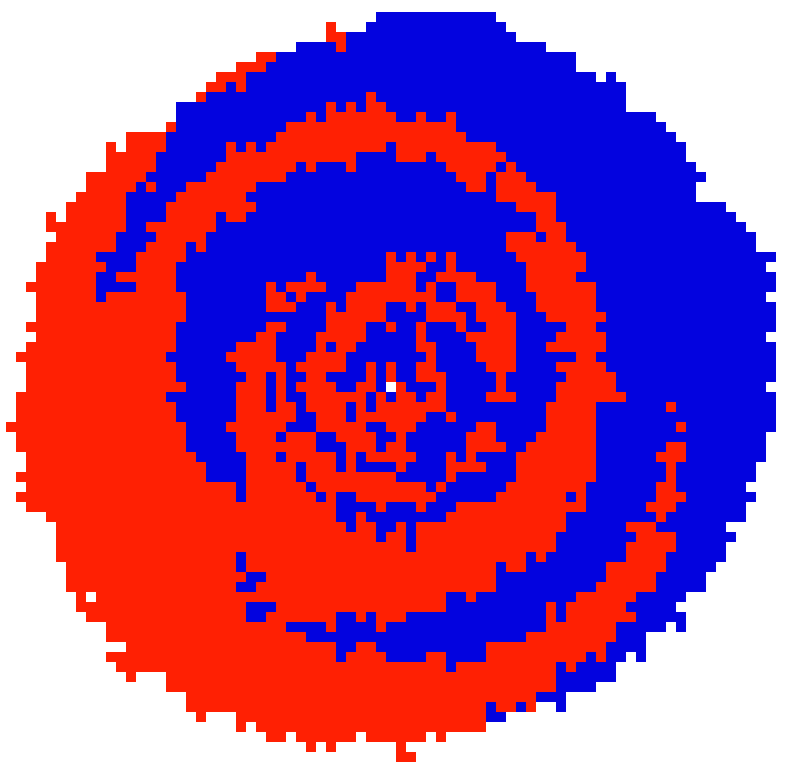}  &
\includegraphics[scale=0.5]{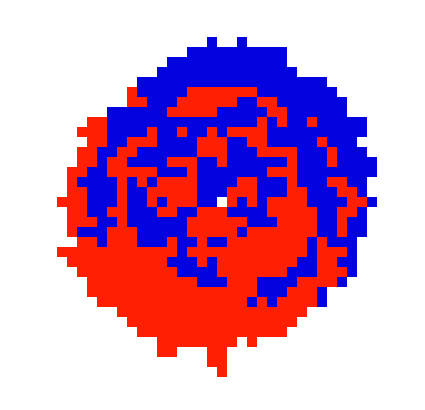} \\
$\Z^2$ & $\Z^3$ (slice through $\mathbf{0}$) 
\end{tabular}
\caption{Samples of competitive erosion in $\Z^2$ (left) and in $\Z^3$ (right: the two-dimensional slice through $\mathbf{0}$ is shown). Each square pixel displays the color of a single site of $\Z^2$ (left) or $\Z^2 \times \{0\}$ (right) after $5 \cdot 10^8$ particles of each color were alternately released at $\mathbf{0}$ (the white pixel in the center). 
Most blue particles convert red sites and vice versa, so that a relatively small number of sites become colored. 
After $n$ particles, is the total number of colored sites of order $n^{d /(2d+2)}$ ? 
}
\label{f.spiral}
\end{figure}

A natural approach to predict the growth rate, and perhaps also to prove the coherence of the red and blue territories, is to use higher-dimensional analogues of the ``signed sum of positions'' \eqref{martingale35}. This leads to a heuristic prediction that after $n$ particles alternating in color are released at the origin in $\Z^d$, the total number of colored sites is order $n^{d/(2d+2)}$.
However, a crucial combinatorial ingredient in our one-dimensional argument was that whenever an uncolored site becomes colored, the existing color configuration is necessarily monochromatic on each side of the origin; this constraint allowed us to compute the precise value $M^t$ must take when a new site becomes colored. 

In dimensions two and above there is no such exact combinatorial constraint.
On the other hand, $\Z^d$ for $d \geq 2$ supports a richer family of discrete harmonic functions.
Any discrete harmonic function $h : \Z^d \to \R$ gives rise to a process $M_h^t$ (defined as the sum of values of $h$ at the red points minus the values of $h$ at the blue points, with double weight given to the currently walking particle). This $M^t_h$ is a martingale except at times when a previously uncolored site becomes colored.  
The existence of coherent red and blue territories constrains the joint distributions of the $(M_h^t)$ as $h$ varies over a space $H$ of harmonic test functions (such as the discrete harmonic polynomials on $\Z^d$). One way to quantify the coherence of the territories would be to find a small subset $S$ of the dual space $H^*$ such that $h \mapsto M_h^t$ belongs to $S$ with high probability.  In principle one could prove this by finding a Lyaponov function which has negative drift where its value is too high.  The challenge is to find a tractable Lyaponov function $f$, such that configurations with a small value of $f$ have coherent territories.

\bibliography{erosion}

\begin{thebibliography}{10}

\bibitem{cr2}
Ton{\'c}i Antunovi{\'c}, Yael Dekel, Elchanan Mossel, and Yuval Peres.
\newblock Competing first passage percolation on random regular graphs.
\newblock {\em Random Structures \& Algorithms}, 50(4):534--583, 2017.

\bibitem{ag}
Amine Asselah and Alexandre Gaudilliere.
\newblock From logarithmic to subdiffusive polynomial fluctuations for internal
  {DLA} and related growth models.
\newblock {\em The Annals of Probability}, 41(3A):1115--1159, 2013.

\bibitem{ag2}
Amine Asselah and Alexandre Gaudilli{\`e}re.
\newblock Sublogarithmic fluctuations for internal dla.
\newblock {\em The Annals of Probability}, 41(3A):1160--1179, 05 2013.

\bibitem{ag3}
Amine Asselah and Alexandre Gaudilli{\`e}re.
\newblock Lower bounds on fluctuations for internal {DLA}.
\newblock {\em Probability Theory and Related Fields}, 158(1-2):39--53, 2014.

\bibitem{benDLA}
Itai Benjamini and Ariel Yadin.
\newblock Upper bounds on the growth rate of diffusion limited aggregation.
\newblock {\em arXiv preprint}.
\newblock \arxiv{1705.06095}.

\bibitem{bianeyor}
Ph. Biane and M.~Yor.
\newblock Valeurs principales associ\'ees aux temps locaux browniens.
\newblock {\em Bull. Sci. Math. (2)}, 111(1):23--101, 1987.

\bibitem{bra1}
Maury Bramson and Joel~L Lebowitz.
\newblock Asymptotic behavior of densities in diffusion-dominated annihilation
  reactions.
\newblock {\em Physical Review Letters}, 61(21):2397, 1988.

\bibitem{bra2}
Maury Bramson and Joel~L Lebowitz.
\newblock Asymptotic behavior of densities for two-particle annihilating random
  walks.
\newblock {\em Journal of Statistical Physics}, 62(1-2):297--372, 1991.

\bibitem{oilwater}
Elisabetta Candellero, Shirshendu Ganguly, Christopher Hoffman, and Lionel
  Levine.
\newblock Oil and water: a two-type internal aggregation model.
\newblock {\em The Annals of Probability}, 45(6A):4019--4070, 2017.

\bibitem{antimatter}
Laurent Canetti, Marco Drewes, and Mikhail Shaposhnikov.
\newblock Matter and antimatter in the universe.
\newblock {\em New Journal of Physics}, 14(9):095012, 2012.

\bibitem{GLPP}
Shirshendu Ganguly, Lionel Levine, Yuval Peres, and James Propp.
\newblock Formation of an interface by competitive erosion.
\newblock {\em Probability Theory and Related Fields}, 168(1-2):455--509, 2017.

\bibitem{gp15}
Shirshendu Ganguly and Yuval Peres.
\newblock Competitive erosion is conformally invariant.
\newblock {\em arXiv preprint}.
\newblock \arxiv{1503.06989}.

\bibitem{cr1}
Olle H{\"a}ggstr{\"o}m, Robin Pemantle, et~al.
\newblock First passage percolation and a model for competing spatial growth.
\newblock {\em Journal of Applied Probability}, 35(3):683--692, 1998.

\bibitem{J14}
Svante Janson.
\newblock Tail bounds for sums of geometric and exponential variables.
\newblock {\em arXiv preprint}.
\newblock \arxiv{1709.08157}.

\bibitem{jls1}
David Jerison, Lionel Levine, and Scott Sheffield.
\newblock Logarithmic fluctuations for internal {DLA}.
\newblock {\em Journal of the American Mathematical Society}, 25(1):271--301,
  2012.

\bibitem{jls2}
David Jerison, Lionel Levine, and Scott Sheffield.
\newblock Internal dla in higher dimensions.
\newblock {\em Electronic Journal of Probability}, 18:14 pp., 2013.

\bibitem{jls3}
David Jerison, Lionel Levine, and Scott Sheffield.
\newblock Internal {DLA} and the {G}aussian free field.
\newblock {\em Duke Mathematical Journal}, 163(2):267--308, 2014.

\bibitem{hk1}
Harry Kesten.
\newblock How long are the arms in {DLA}?
\newblock {\em Journal of Physics A: Mathematical and General}, 20(1):L29,
  1987.

\bibitem{lbg}
Gregory~F Lawler, Maury Bramson, and David Griffeath.
\newblock Internal diffusion limited aggregation.
\newblock {\em Annals of Probability}, pages 2117--2140, 1992.

\bibitem{lpw}
David~Asher Levin, Yuval Peres, and Elizabeth~Lee Wilmer.
\newblock {\em Markov chains and mixing times}.
\newblock American Mathematical Soc., 2009.

\bibitem{MP}
Peter M\"orters and Yuval Peres.
\newblock {\em Brownian motion}, volume~30 of {\em Cambridge Series in
  Statistical and Probabilistic Mathematics}.
\newblock Cambridge University Press, Cambridge, 2010.
\newblock With an appendix by Oded Schramm and Wendelin Werner.

\bibitem{revuz}
Daniel Revuz and Marc Yor.
\newblock {\em Continuous martingales and Brownian motion}, volume 293.
\newblock Springer Science \& Business Media, 2013.

\bibitem{Sadler}
Peter~M Sadler.
\newblock Sediment accumulation rates and the completeness of stratigraphic
  sections.
\newblock {\em The Journal of Geology}, 89(5):569--584, 1981.

\end{thebibliography}
\bibliographystyle{plain}

\end{document}